\documentclass{amsart}
\usepackage[utf8]{inputenc}
\usepackage{fullpage}
\usepackage{bbm}
\usepackage[english]{babel}

\usepackage{amsmath,amsthm}
\usepackage{amsfonts}
\usepackage{mathtools}
\usepackage{amssymb}
\usepackage{mathrsfs}
\usepackage{stmaryrd} 

\usepackage{enumerate}
\usepackage{mdwlist}
\usepackage[normalem]{ulem}
\usepackage{crossreference}
\usepackage{color}

\usepackage{tikz}

\usepackage{perpage} 

\newtheorem{thm}{Theorem}[section]
\newtheorem*{thm*}{Theorem}
\newtheorem{cor}[thm]{Corollary}
\newtheorem*{cor*}{Corollary}
\newtheorem{lem}[thm]{Lemma}
\newtheorem*{lem*}{Lemma}
\newtheorem{prop}[thm]{Proposition}
\newtheorem*{prop*}{Proposition}
\newtheorem{claim}[thm]{Claim}
\theoremstyle{definition}
\newtheorem{defn}[thm]{Definition}

\newtheorem*{defn*}{Definition}

\newtheorem*{notation*}{Notation}

\newtheorem*{fact*}{Facts}

\newtheorem*{question}{Question}

\theoremstyle{definition}

\newtheorem{exa}[thm]{Example}
\newtheorem{rmk}[thm]{Remark}
\newtheorem*{rmk*}{Remark}

\newtheorem{theoremintro}{Theorem}
\newtheorem{corollaryintro}[theoremintro]{Corollary}
\newtheorem{propintro}[theoremintro]{Proposition}

\MakePerPage{footnote} 
\usepackage[colorlinks=true, allcolors=blue]{hyperref}

\makeatletter
\def\moverlay{\mathpalette\mov@rlay}
\def\mov@rlay#1#2{\leavevmode\vtop{%
   \baselineskip\z@skip \lineskiplimit-\maxdimen
   \ialign{\hfil$\m@th#1##$\hfil\cr#2\crcr}}}
\newcommand{\charfusion}[3][\mathord]{
    #1{\ifx#1\mathop\vphantom{#2}\fi
        \mathpalette\mov@rlay{#2\cr#3}
      }
    \ifx#1\mathop\expandafter\displaylimits\fi}
\makeatother

\def \mf {\mathfrak}

\def \pmtb {\begin{pmatrix}}
\def \pmte {\end{pmatrix}}
\def \bsmt {\begin{smallmatrix}}
\def \esmt {\end{smallmatrix}}

\def \bml {\begin{align*}}
\def \eml {\end{align*}}

\newcommand{\wqg}{\mathcal{W}_{\mathbb{Q}}(\Gamma)}
\newcommand{\fpi}[1]{\frac{\pi}{#1}}

\newcommand{\qr}{$\mathbb{Q}$-rank $1$\ }
\newcommand{\qrc}{$\mathbb{Q}$-rank $1$,\ }
\newcommand{\qrd}{$\mathbb{Q}$-rank $1$.\ }

\newcommand{\rr}{\mathbb{R}$-rank$\ 1}

\newcommand{\hb}{\mathcal{HB}}
\newcommand{\h}{\mathcal{H}}
\newcommand{\hbg}{\mathcal{HB}^\Gamma}
\newcommand{\hg}{\mathcal{H}^\Gamma}
\newcommand{\hbl}{\mathcal{HB}^\Lambda}
\newcommand{\hl}{\mathcal{H}^\Lambda}

\newcommand{\mn}{\mathcal{N}}
\newcommand{\mrp}{\mathbb{R}_{>0}}
\newcommand{\mrpp}{\mathbb{R}_{\geq 1}}
\newcommand{\mrnn}{\mathbb{R}_{\geq 0}}

\newcommand{\g}{\gamma}
\newcommand{\G}{\Gamma}
\newcommand{\la}{\lambda}
\newcommand{\La}{\Lambda}

\newcommand{\N}{\{1,\dots,N\}}

\title{Sublinear Rigidity of Lattices in Semisimple Lie Groups}

\author{Ido Grayevsky}
\address{School of Mathematics, University of Bristol, UK}
\email{grayevsky@gmail.com}


\date{\today}
\dedicatory{}

\begin{document}

\begin{abstract}
Let $G$ be a real centre-free semisimple Lie group without compact factors. I prove that lattices in $G$ are rigid under two types of sublinear distortions. The first result is that if $G$ does not admit $\rr$ factors, its lattices are \emph{SBE complete}: if $\La$ is an abstract  finitely generated group that is Sublinearly BiLipschitz Equivalent (SBE) to an irreducible lattice in $G$, then $\La$ can be homomorphically mapped into $G$ with finite kernel and image a lattice in $G$. This generalizes the well known quasi-isometric completeness of lattices. The second result concerns sublinear distortions within $G$ itself, and holds without any restriction on the rank of the factors: if $\La\leq G$ is a discrete subgroup that \emph{sublinearly covers} a lattice $\G\leq G$, then $\La$ is itself a lattice. 
\end{abstract}

\maketitle

\section{Introduction}
The quasi-isometric rigidity and classification of irreducible lattices in semisimple Lie groups was established in the 1990's by the accumulated work of many - Pansu~\cite{pansu_1989}, Schwartz~\cite{Schwartz}, Kleiner-Leeb~\cite{KleinerLeeb}, Eskin~\cite{EskinLatticeClassification}, Eskin-Farb~\cite{EskinFarb}, Druţu~\cite{DrutuQI}, to name a few which are closely related to this paper. Consult~\cite{Farb1997QISummary} for a concise survey of the following results. 

\begin{thm}[Quasi-Isometric Completeness, Theorem I in~\cite{Farb1997QISummary}]\label{thm: QI completeness}
Let $G$ be a real finite-centre semisimple Lie group without compact factors, $\G\leq G$ an irreducible lattice and $\La$ an abstract finitely generated group. If $\G$ and $\La$ are quasi-isometric, then there is a group homomorphism $\Phi: \La\rightarrow G$ with finite kernel whose image $\La':=\Phi(\La)$ is a lattice in $G$. Put differently, there is a lattice $\La'\leq G$ and a finite subgroup $F\leq \La$ such that the sequence 

$$1\rightarrow F\rightarrow \La\rightarrow \La'\rightarrow 1$$

is exact. Moreover, $\La'$ is uniform if and only if $\G$ is uniform, and in particular the classes of uniform and non-uniform lattices in $G$ are quasi-isometrically complete.

\end{thm}

\begin{thm}[Quasi-Isometric Classification, Theorem II in~\cite{Farb1997QISummary}]\label{thm: QI classification}
The quasi-isometry classes
of irreducible lattices in semisimple Lie groups with finite centre are precisely:
\begin{itemize}
    \item One quasi-isometry class for each semisimple Lie group $G$, consisting
of the uniform lattices in $G$.
\item One quasi-isometry class for each commensurability class of irreducible,
nonuniform lattice, except in $G = \operatorname{SL}_2(\mathbb{R})$, where there is precisely one
quasi-isometry class of nonuniform lattice.

\end{itemize}
\end{thm}

The question underlying this work is whether the same results hold for Sublinear Bilipschitz Equivalence (SBE, Definition~\ref{def: SBE}). SBE is a weak form of quasi-isometry, brought forward by Cornulier in the past two decades or so~\cite{cornulier2017sublinear}. In some precise sense, these are the weakest maps which give rise to biLipschitz asymptotic cones. Another key motivation to study SBE comes from their important role in the geometry of solvable Lie groups~\cite{CornulierCones11}. SBE are strictly weaker than quasi-isometries: there are examples of solvable Lie groups that are SBE but not quasi-isometric; such examples are hard to come by and difficult to prove - see Example 1.14 in~\cite{cornulier2017sublinear}. Pallier~\cite{Pallier_2020} was the first to consider SBE of semisimple Lie groups. He proved that two $\rr$ symmetric spaces are SBE if and only if they are homothetic, answering a question of Drutu (Question 1.16 and Remark 1.17 in~\cite{cornulier2017sublinear}). In any case, determining whether the quasi-isometry class of a group differs from its SBE class is a very delicate question.

The results of this work give a complete generalization of Theorem~\ref{thm: QI completeness} for groups $G$ without $\rr$ factors, and establish substantial progress also when $G$ admits a $\rr$ factors or is itself of $\rr$. In the context of Theorem~\ref{thm: QI classification}, I prove a limited classification result in a particular case.  

\subsection{Main Results}
As with quasi-isometric completeness, SBE-completeness follows from the rigidity of two forms of sublinear distortions. The first is \textbf{SBE-rigidity}, namely the property that every self-SBE of a lattice is sublinearly close to an isometry. The second form is \textbf{sublinear rigidity}: a subgroup $\La\leq G$ which \emph{sublinearly covers} a lattice $\G\leq G$ must be a lattice. 

A function $u:\mrnn\rightarrow \mrpp$ is \emph{sublinear} if $\lim_{r\rightarrow\infty}\frac{u(r)}{r}=0$. For ${a,b\in \mrnn}$ denote ${a\vee b:\max\{a,b\}}$, and for a pointed metric space ${(X,x_0,d_X)}$ with $x,x_1,x_2\in X$, denote  ${|x|_X:=d_X(x,x_0)}$ and  ${|x_1-x_2|_X:=d_X(x_1,x_2)}$ (I discard the subscripts when they are clear from the context).

\begin{defn}\label{def: SBE}
Let $(X,d_X,x_0), (Y,d_Y,y_0)$ be pointed metric spaces, $L\in \mrpp$ a constant,  ${u:\mrnn\rightarrow \mathbb{R}_{\geq 1}}$ a sublinear function. A map $f:X\rightarrow Y$ is an \emph{ $(L,u)$-SBE} if the following conditions are satisfied:

\begin{enumerate}
    \item $f(x_0)=y_0$,
    \item $\forall x_1,x_2\in X,$ 
    $$\frac{1}{L}|x_1-x_2|_X-u\big(|x_1|_X\vee|x_2|_X\big)\leq |f(x_1)-f(x_2)|_Y\leq L|x_1-x_2|_X+u\big(|x_1|_X\vee|x_2|_X\big),$$
   
   \item $\forall y\in Y\  \exists x\in X$ such that $|y-f(x)|_Y\leq u(|y|_Y)$.
   \end{enumerate}
   
   A map is an \emph{SBE} if it is an $(L,u)$-SBE for some $L$ and $u$ as above.
\end{defn}

SBE give rise to sublinear distortions of various sorts. A significant one is captured by the following simple definition, which is fundamental to this work.

\begin{defn}
For a function $u:\mrnn\rightarrow \mrp$ and a subset $Y\subset X$, define the \emph{$u$-neighbourhood of $Y$} to be

$$\mn_u(Y):=\{x\in X\mid d(x,Y)\leq u(|x|)\}$$

A subset $Y\subset X$ is said to \emph{sublinearly cover} $Z\subset X$ if $Z\subset\mn_u(Y)$ for some sublinear function $u$.
\end{defn}

It is more natural to formulate some results in the language of symmetric spaces. For the definition of the `compact core' of a lattice in the following statement, see Section~\ref{sec: cusps compact core rational Tits}.

\begin{theoremintro}[SBE-Rigidity]\label{thm: sbe close to isometry Intro} \label{sbe: thm: self sbe is close to isometry}

Let $X$ be a symmetric space of noncompact type  without $\rr$ factors, $\G\leq \mathrm{Isom}(X)$ an irreducible lattice and $X_0\subset X$ the compact core of $\G$ in $X$. For any $(L,u)$-SBE $f: X_0\rightarrow X_0$ there is a sublinear function $v=v(L,u)\in O(u)$ and an isometry $g:X\rightarrow X$ such that $d\big(f(x),g(x)\big)\leq v(|x|)$ for all $x\in X_0$. 
\end{theoremintro}

An immediate corollary is the high rank analogue of Pallier's result.

\begin{corollaryintro}\label{cor: symmetric space classification}
    Let $X$ and $Y$ be symmetric spaces of noncompact type and of high rank. If they are SBE, then they are isometric up to rescaling of the metrics in each de Rham factor.  
\end{corollaryintro}

Theorem~\ref{thm: sbe close to isometry Intro}  is the key ingredient in the following theorem. 

\begin{theoremintro}\label{intro: thm: from SBE to sublinearly close}
Let $G$ be a real semisimple Lie group without compact or $\rr$ factors. Let $\G\leq G$ be an irreducible lattice and $\La$ an abstract finitely generated group, both viewed as metric spaces with some word metrics. Assume that $\La$ is $(L,u)$-SBE for a sublinear subadditive function $u$. Then there is a group homomorphism  $\Phi:\La\rightarrow G$ with finite kernel, and a sublinear function $v\in O(u)$ such that, for $\La':=\mathrm{Im}(\Phi)$, it holds that $\G\subset \mn_v(\La')$ and $\La'\subset \mn_v(\G)$.
\end{theoremintro}

The conclusion in Theorem~\ref{intro: thm: from SBE to sublinearly close} motivates the notion of sublinear rigidity, which concerns sublinear distortions within the group $G$ itself.  

\begin{theoremintro}[Sublinear Rigidity]\label{thm: main Intro}\label{thm: main}
Let $G$ be a real finite-centre semisimple Lie group without compact factors. Let $\G\leq G$ be an irreducible lattice, $\La\leq G$ a discrete subgroup that sublinearly covers $\G$. If $\G$ is a \qr lattice, assume further that $\La$ is irreducible, and if $\G$ has $\mathbb{Q}$-rank at least $2$, assume that $G$ has trivial centre. Then $\La$ is a lattice in $G$; moreover, it is uniform if $\G$ is.
\end{theoremintro}

Theorem~\ref{thm: main Intro} does not impose any restriction on the rank of the factors of $G$. In particular, it holds for all groups of $\rr$, including $\mathrm{SL}_2(\mathbb{R})$. For clarity I remark that throughout this paper, by a \emph{\qr lattice} I mean either an arithmetic lattice of \qrc or any non-uniform lattice in a $\rr$ group, whether arithmetic or not. For precise definitions, see Section~\ref{sec: prelims: generalities on lattices} and Theorem~\ref{qr: prelims: sym spc: thm: Leuzinger compact core}. The notion of irreducibility in the non-standard context of a general (non-lattice) subgroup is explained in Section~\ref{sec: indecomposible horospherical lattice} (see Definition~\ref{qr: defn: indecompsable and irreducibale lattice}). 

The proof is split into three exhaustive non-mutually exclusive cases (Lemma~\ref{lem: cases}): (a) $\G$ has Kazhdan's property (T) (Theorem~\ref{thm: main for T}), (b) $\G$ is uniform (Theorem~\ref{U: thm: main}), and (c) $\G$ is a \qr lattice (Theorem~\ref{qr: thm: main for qr}). The case of \qr lattices is the cornerstone of this paper, and contains the bulk of the novel contributions. The key ingredient is the following proposition, which might be of independent interest.

\begin{propintro}\label{prop: Lambda cocompact horosphreres Intro} \label{qr: prop: Lambda cocompact horosphreres}
Let $G$ be a real finite-centre semisimple Lie group without compact factors. Let $\G\leq G$ an irreducible \qr lattice, $\La\leq G$ a discrete subgroup that sublinearly covers $\G$. If for all $D>0$ $\G\not\subset\mn_D(\La)$, then there exists a $\G$-parabolic horosphere $\h$ such that ${\big(\La\cap \mathrm{Stab}_G(\h)\big)\cdot x_0}$ intersects $\h$ in a cocompact metric lattice. Moreover, the open horoball bounded by $\h$ does not intersect the orbit $\La x_0$.
\end{propintro}

Combining Theorem~\ref{intro: thm: from SBE to sublinearly close} and Theorem~\ref{thm: main} yields SBE-completeness for groups without $\rr$ factors. 

\begin{corollaryintro}[SBE-Completeness]\label{thm: SBE main Intro}
Let $G$ be a real centre-free semisimple Lie group without compact or $\rr$ factors. Let $\G\leq G$ be an irreducible lattice and $\La$ an abstract finitely generated group, both viewed as metric spaces with some word metrics. Assume there is an $(L,u)$-SBE $f: \Lambda\rightarrow \G$ with $u$ a subadditive sublinear function. Then there is a group homomorphism $\Phi:\Lambda\rightarrow G$ with finite kernel whose image $\La':=\Phi(\La)$ is a lattice in $G$. Moreover, $\La'$ is uniform if and only if $\G$ is uniform. 

\end{corollaryintro}

I stress right away that while the proof of SBE-completeness requires Theorem~\ref{thm: main}, it only uses the case of groups with property (T), which is quite straight-forward and is proved in Section~\ref{sec: T}.

\subsubsection{Refinements}
I obtain slightly stronger results in some favourable settings. As Theorem~\ref{intro: thm: from SBE to sublinearly close} indicates, sublinear neighbourhoods arise naturally in the presence of SBE. Moreover, observe that if $v$ were the identity function $v(r)=r$, then by definition $v(|g|)=v\big(d(g,e_G)\big)=d(g,e_G)$. In particular $G\subset \mn_v(\{e_G\})$, whereas the trivial subgroup is, after all, not a lattice.  This suggests that sublinearity is in some sense optimal. 

However, for uniform lattices and for lattices with property (T), one can relax the sublinearity assumption and consider $\varepsilon$-linear neighbourhoods instead.

\begin{theoremintro}
    
[\texorpdfstring{$\varepsilon$-linearity}{}, Theorems~\ref{T: thm: main epsilon} and~\ref{U: thm: main epsilon}]\label{thm: intro: epsilon-linearity}

Let $G$ be a real finite-centre semisimple Lie group without compact factors, $\G\leq G$ an irreducible lattice, $\La\leq G$ a discrete subgroup. Fix  $\varepsilon>0$ and assume $\G\subset\mn_u(\La)$ for the function $u(r)=\varepsilon r$.

\begin{enumerate}   
    \item If $\G$ has Kazhdan's property (T) and $G$ has trivial centre, then there is a constant $\varepsilon(G)$ such that if $\varepsilon<\varepsilon(G)$ then $\La$ is a lattice. 
      \item If $\G$ is uniform and $\varepsilon<1$, then $\La$ is a uniform lattice. 
\end{enumerate}
\end{theoremintro}

\begin{rmk}
The concept of $\varepsilon$-linearity and corresponding `generalized quasi-isometries' was used, as a technical tool, in Eskin and Farb's work~\cite{EskinFarb}~\cite{EskinLatticeClassification}. While much of their work is carried for maps which are even more general than SBE, it is not clear whether their approach can yield a sublinear bound as in Theorem~\ref{thm: sbe close to isometry Intro}, which is necessary for the proof of Corollary~\ref{thm: SBE main Intro}.

\end{rmk}
The second refinement concerns the SBE-classification of lattices. I prove a limited commensurability result, relying on a reduction to the bounded case:
\begin{theoremintro}\label{thm: commensurability and uniform statement intro}
Let $G,\G$ and $\La$ be as in Theorem~\ref{thm: main Intro}. If $\G$ is of \qr  then: 
\begin{enumerate}
    \item If $\G\not\subset\mn_D(\La)$ for any $D>0$, then also $\La$ is of \qrd
    
    \item Assume $G$ is not locally isomorphic to $\mathrm{SL}_2(\mathbb{R})$. If $\G\subset\mn_D(\La)$ for some $D>0$ and in addition $\G$ sublinearly covers $\La$, then $\La$ is commensurable to $\G$.
\end{enumerate}
\end{theoremintro}

Note that the case where $\G$ and $\La$ each sublinearly cover the other is the conclusion of Theorem~\ref{intro: thm: from SBE to sublinearly close}, and that $\G\not\subset \mn_D(\La)$ for any $D>0$ is the hypothesis of Proposition~\ref{prop: Lambda cocompact horosphreres Intro}.

\subsection{About the Proof}\label{sec: outline of proof}

\subsubsection{SBE Completeness}

The proof of Theorem~\ref{thm: sbe close to isometry Intro} and the road map for Corollary~\ref{thm: SBE main Intro} draw mostly on Drutu's proof for Theorem~\ref{thm: QI completeness}~\cite{DrutuQI}. Apart from the property of sublinear rigidity for groups with property (T) (Theorem~\ref{thm: main for T}), the main new ingredient is a method for gaining control on the bounds that arise in the course of her proof. More concretely, this type of control is a higher-rank analogue of the following simple question: it is well known that every $(L,c)$-quasi-geodesic in a $\rr$ symmetric space lies in the $D=D(L,c)$ neighbourhood of a geodesic. How does the bound $D$ depend on $L$ and $c$? This question, which is completely irrelevant in the context of quasi-isometric rigidity, is critical in the SBE setting (see e.g.\ Section I.3 in~\cite{pallierTheses2019}). This is the essence of Lemma~\ref{sbe: lem: uniform control on the sublinear constants for fixed x}, a stronger, uniform version of Theorem~\ref{thm: sbe close to isometry Intro}, which is required for  Theorem~\ref{intro: thm: from SBE to sublinearly close}. 

The scheme towards SBE completeness (Corollary~\ref{thm: SBE main Intro}) is parallel to the quasi-isometry case: each $\la\in \La$ naturally gives rise to a self-SBE $f_\lambda: X_0\rightarrow X_0$ of the compact core of $\G$. Each such self-SBE is close to an isometry $g_\La\in G$ by Theorem~\ref{thm: sbe close to isometry Intro}, whose proof is based on the extension theorem of Tits~\cite{tits1974buildings}. Theorem~\ref{intro: thm: from SBE to sublinearly close} proves that the correspondence $f_\la\mapsto g_\la$ yields an embedding of $\La$ as a discrete subgroup of isometries in $G$ that sublinearly covers $\G$. By Theorem~\ref{thm: main Intro}, specifically Theorem~\ref{thm: main for T} for lattices with property (T), $\La$ is a lattice. 

\subsubsection{Sublinear Rigidity.}

The proof of Theorem~\ref{thm: main} is essentially different from the earlier works of Schwartz and Eskin on `bounded rigidity' (see discussion in Section~\ref{sec: qr: bounded case}). When $G$ has property (T), a criterion of Leuzinger~\cite{leuzingerCriticalExponent} states that a discrete subgroup is a lattice as soon as it has the exponential growth rate of a lattice (see Theorem~\ref{T: thm: Leuzinger criterion}). It is straightforward to verify that sublinear distortion cannot affect this growth rate (Corollary~\ref{T: cor: La has same critical exponent as Ga}). The proof for uniform lattices is essentially a geometric observation. The case of \qr lattices presents significant difficulties and the proof is technically demanding. It exploits delicate geometric features of the symmetric space and its lattices, brings forth the role of horospheres, and highlights why sublinear distortions in this setting are very rigid.
One thing to note is that the argument is carried mostly within the symmetric space itself and not in some boundary object. The following is a detailed sketch that conveys the key ideas while avoiding the technical difficulties.

\paragraph{\textbf{Strategy for \qr Lattices.}}

The rationale stems from a conjecture of Margulis, proven by Hee Oh~\cite{Oh1998DiscreteSG} for a large class of groups and recently resolved in full generality by Benoist and Miquel~\cite{BenoistMiquelArithmeticity} (see Theorem~\ref{qr: thm: Benoist-Miquel arithmeticity}). It states that in a high $\mathbb{R}$-rank semisimple algebraic group,  a discrete Zariski-dense subgroup is a lattice as soon as it intersects a horospherical subgroup in a lattice. This could be seen as an algebraic converse to the geometric structure of the orbit in $X$ of a \qr lattice, which intersects some parabolic horospheres in a cocompact metric lattice (see Section~\ref{sec: qr: compact core and rational Tits Building} for details). Proposition~\ref{prop: Lambda cocompact horosphreres Intro} is the geometric analogue of the horospherical lattice criterion and, up-to some non-trivial translation work, completes the proof when $G$ is of high $\mathbb{R}$-rank. A strong consequence of Proposition~\ref{prop: Lambda cocompact horosphreres Intro} is that every $\G$-conical limit point is $\La$-conical (Corollary~\ref{qr: cor: every Gamma conical is Lambda conical}). A simple use of a criterion for geometrically finite groups in Kapovich-Liu~\cite{KapovichLiuGeometricFI} (see Theorem~\ref{qr: thm: Kapovich-Liu conical limit points}) concludes the proof when $G$ is of $\rr$.

\paragraph{\textbf{The ABC of Sublinear Constraints.}}

Fix a point $x_0\in X=G/K$, identify $\G$ and $\La$ with their respective $x_0$-orbits, and denote $d_\g:=d(\g,\La)=d_X(\g x_0,\La x_0)$. By definition, the interior of a ball $B(\g x_0,d_\g)$ does not intersect $\La x_0$. I call such sets \emph{$\La$-free}. Moreover, these balls intersect $\La x_0$ (only) in the bounding sphere: call such sets \emph{tangent to $\La$}. The $\La$-free and, respectively, $\G$-free regions in $X$ are the main objects of interest in this proof. Theorem~\ref{thm: main Intro} is known in the case of bounded $\{d_\g\}_{\g\in\G}$ (see Section~\ref{sec: qr: bounded case}), so it makes sense to think about large $\La$-free regions as `problematic'. The proof relies on two easy observations that complete each other. 

\begin{enumerate}
    \item The sublinear constraint implies that  $d_{\g_n}\rightarrow \infty$ forces $|\g_n|\rightarrow \infty$, suggesting that `problematic' $\La$-free regions should appear only `far away' in the space. 
    
    \item On the other hand $\La$ is a group, and being $\La$-free is a $\La$-invariant property. In particular any metric situation that can be described in terms of the $\La$-orbit (e.g.\ $B(\g,d_\g)$ is a $\La$-free ball tangent to $\La$) can be pulled back to $x_0$. Consequently, one finds `problematic' regions near $x_0$. 
    
\end{enumerate}

The moral can be formulated into a general principle that lies at the heart of the argument. The sublinear constraint $d_\g\leq u(|\g|)$ gives rise to many constraints of `sublinear' nature. Inside any fixed bounded neighbourhood of $x_0$, such constraints simply give uniform bounds. Since $\La$ and $\G$ are groups, bounds `near' $x_0$ often turn out to be global bounds that depend only on the group and not on a specific orbit point. Put differently: the trick is to describe metric configurations in terms of the $\La$ and $\G$ orbits. Using the group invariance, one can move these configurations around the space to a place where the sublinear constraint can be exploited.

\paragraph{\textbf{The Argument.}}\label{Paragraph: proof of cocompact horosphere overview}
The above paragraph would suffice the reader to produce a complete proof for uniform lattices. For non-uniform lattices, denote by $\la_\g$ the $\La$-orbit point closest to $\g\in \G$. For $\h$ a parabolic (cusp) horosphere of $\G$, let $\G_\h:=\{\g\in \G\mid \g x_0\in \h\}$, and note that $\G_\h\cdot x_0\subset \h$ is a metric lattice. The ultimate goal is to show that the set $\{\la_\g x_0\}_{\g\in\G_\h}$ is more or less a metric lattice on some horosphere $\h'$ parallel to $\h$. One proceeds by the following steps. 

\begin{enumerate}
    \item \emph{Finding $\La$-free horoballs} (Section~\ref{sec: Lambda free horoballs}): The arbitrarily large $\La$-free balls $B(\g_n,d_{\g_n})$ are translated to $x_0$. The compactness of the unit tangent space at $x_0$ yields a converging direction which is the base point at infinity of a $\La$-free horoball. Translation by $\La$ yields, at each $\La$-orbit point, a $\La$-free horoball tangent to it.
    
    \item \emph{Controlling angles} (Section~\ref{sec: Gamma points deep in Lambda free horoballs}): For every $\g\in \G$ with $d_\g$ uniformly large enough, one associates a point at infinity $\xi\in X(\infty)$ such that $\xi$ is the base point of a $\La$-free horoball tangent to $\la_\g x_0$. The angle between the geodesics $[\la_\g x_0,\g x_0]$ and $[\la_\g x_0,\xi)$ is shown to be small as $d_\g$ grows large. This is used to show that arbitrarily large $d_\g$ give rise to arbitrarily deep $\G$-orbit points inside $\La$-free horoballs. A key step is Lemma~\ref{U: lem: large Lambda free balls near base point imply large Gamma free balls}, producing a $\G$-free corridor between $\la_\g x_0$ and $\g x_0$.
    
    \item \emph{$\La$-cocompact horospheres} (Section~\ref{sec: cocompact horosphere}): One uses uniform bounds near $x_0$ to prove that for \emph{every} $\La$-free horoball $\mathcal{HB}$ that is (almost) tangent to $\La$, its bounding horosphere $\mathcal{H}:=\partial\mathcal{HB}$ must lie in a uniformly bounded neighbourhood of $\La x_0$. If $d_\g$ is large enough for some $\g$ that lies on a parabolic horosphere $\hg$ of $\G$, then any $\g'\in \G_{\hg}$ also admits large $d_{\g'}$. Since the bounds from the previous steps only depend on $d_{\g'}$, all $\la_{\g'}$ are forced to lie on the same horosphere $\hl$ parallel to $\hg$. One concludes that $\La x_0$ intersects $\hl$ on the nose in a cocompact metric lattice.
\end{enumerate}

\subsection{Possible Improvements, Related and Future Work}

The proof suggests three improvements to the statement of Theorem~\ref{thm: main Intro}. The trivial centre assumption for lattices of high $\mathbb{Q}$-rank could be relaxed to allow finite centre, if the same relaxation is applicable to Leuzinger's work on critical exponents in groups with property (T)~\cite{leuzingerCriticalExponent}. The irreducibility of $\La$, assumed in the case of \qr lattices, may be derived directly from the irreducibility of $\G$. Lastly, the geometric characterization of the $\mathbb{Q}$-rank of a lattice (Corollary D in \cite{LeuzingerRationalRank}) suggests that the $\mathbb{Q}$-rank of $\La$ must equal that of $\G$.

Other problems that arise naturally from this work seem to require new ideas:

\begin{question}
Let $G$ be a real finite-centre semisimple Lie group without compact factors that admits a $\rr$ factor. Are the classes of uniform and non-uniform lattices of $G$ SBE-complete? In particular, is this true when $G$ is of $\rr$?
\end{question} 

The restriction in Theorem~\ref{thm: sbe close to isometry Intro} (and consequently in Corollary~\ref{thm: SBE main Intro}) to groups without $\rr$ factors stems from the use of Tits' extension theorem stating that certain maps on the spherical building must arise as boundaries of isometries. In the presence of $\rr$ factors, the building structure degenerates and cannot be of use in this context. In Section~\ref{sec: SBE for rank 1 factors} I discuss other strategies for SBE-rigidity in $\rr$, which are closely related to Pallier's works~\cite{Pallier_2020, PallierQuasiconformality}.

Pallier and Qing~\cite{PallierQing} recently showed that the \emph{sublinear Morse boundary} is an SBE invariant. The sublinear Morse boundary itself (see~\cite{QingSubLinMorseBoundary},~\cite{qing2020sublinearly})  is yet another manifestation of the importance of sublinear distortions in the study of metric spaces. 

Another challenge is to find non-trivial examples of the settings in Theorem~\ref{thm: main Intro} and Corollary~\ref{thm: SBE main Intro}:

\begin{question}
Let $G$ be a real finite-centre semisimple Lie group without compact factors, $\G\leq G$ a lattice.
\begin{enumerate}
    \item Does there exist a finitely generated group that is SBE to $\G$ but not quasi-isometric to it? Or, at least, not known to be quasi-isometric to a lattice?
    
    \item Does there exist $\La\leq G$ discrete that sublinearly covers $\G$ but which does not boundedly cover it? Or a subgroup which $\varepsilon$-linearly covers $\G$ but which does not sublinearly cover it? 
\end{enumerate}

\end{question}

Focusing on the proof, it would be very interesting if the geometric approach proving sublinear rigidity for \qr lattices could be applied to any $\mathbb{Q}$-rank. Although there are apparent places where the proof uses the unique geometry of \qr lattices, most of the geometric arguments that lead to Proposition~\ref{prop: Lambda cocompact horosphreres Intro} seem to be susceptible to the higher $\mathbb{Q}$-rank setting. Such a generalization of the proof would definitely shed more light on the mysterious lattice arising from growth considerations in property (T) groups, and in particular on the question of commensurability of $\La$ and $\G$ in that case. It would also be interesting to see whether one can push the geometric argument in order to establish a complete geometric analogue of the Benoist-Miquel-Oh criterion. Namely, can one find a direct geometric proof that $\La$ admits finite co-volume (perhaps similarly to Schwartz's argument in the bounded case, see Section~\ref{sec: qr: bounded case} below).

Lastly, one could possibly relate this work to the work of Fraczyk and Gelander~\cite{GelnderFraczykInjRad}, who proved that a discrete subgroup (of a simple high rank Lie group) is a lattice if and only if it has bounded injectivity radius. While their result seems very much related to the condition $\G\subset \mn_u(\La)$, the nature of their work does not give explicit bounds on the injectivity radius. Specifically, given $r>0$ one cannot tell directly  how `far' one must wander in $X$ in order to find a point with injectivity radius $r$. Perhaps one could use the sublinear results of this work to say something about the relation between $|x|_X$ and $\mathrm{InjRad(x)}$.

\subsection{Organization of Paper}
Section~\ref{sec: prelims} gives the necessary preliminaries. These are mainly geometric and mostly concern the proof of Theorem~\ref{thm: main Intro} in the case of \qr lattices, but they are also helpful for understanding some of the geometric statements in other parts of this work. Section~\ref{sec: SBE chapter} is devoted to the proof of Theorem~\ref{thm: sbe close to isometry Intro} and Theorem~\ref{intro: thm: from SBE to sublinearly close}.  The proof of Corollary~\ref{thm: SBE main Intro} uses the sublinear rigidity for lattices with property (T), which is proved in Section~\ref{sec: T}. The rest of Theorem~\ref{thm: main Intro}, namely the cases of uniform lattices and of \qr lattices, appear in Sections~\ref{sec: u} and~\ref{sec: qr}, respectively. The key step is Proposition~\ref{prop: Lambda cocompact horosphreres Intro}, which is proved in Section~\ref{sec: qr: cocompact horosphere}. Theorem~\ref{thm: intro: epsilon-linearity} is proved in Theorem~\ref{T: thm: main epsilon} for property (T) groups and Theorem~\ref{U: thm: main epsilon} for uniform lattices. The proof of Theorem~\ref{thm: commensurability and uniform statement intro} appears in Section~\ref{sec: commensurability theorem proof}, relying on the work done in Section~\ref{sec: qr}.

\subsection{Acknowledgments}
This paper is part of my DPhil thesis, supervised by Cornelia Druţu. I thank her for suggesting the question of SBE-completeness and for guiding me in my first steps in the theory of Lie groups.  I thank Uri Bader, Tsachik Gelander and the Midrasha on groups at the Weizmann institute, where I learned the basics of symmetric spaces. I thank my thesis examiners Emmanuel Breuillard and Yves Cornulier for their careful inspection and numerous remarks. I thank Elon Lindenstrauss for telling me about Leuzinger's result~\cite{leuzingerCriticalExponent} on property (T), and Or Landesberg, Omri Solan, Elyasheev Leibtag, Itamar Vigdorovich and Tal Cohen for many discussions on different aspects of this paper. Specifically I thank Elyasheev Leibtag for his help with the theory of algebraic groups needed for Section~\ref{sec: Zarsiky density subsection}. I thank Gabriel Pallier for explaining his examples in~\cite{PallierQuasiconformality} of some unusual sublinearly quasi-conformal maps on the boundaries of $\rr$ symmetric spaces, and for his interest in this work. 

Finally, I thank Uri Bader and the anonymous referee for their comments, which led me to improve the exposition and tighten the text, making the paper more concise and easier to follow.

I thank the following institutions and funding bodies for their financial support: the Mathematical Institute at the University of Oxford where I studied for a DPhil; the Israel Science Foundation (grants ISF 2919/19, ISF 2990/21, ISF 1577/23 and ISF 3187/24); and the Royal Society (grant NIF\textbackslash R1\textbackslash 242363). I would like to thank the Isaac Newton Institute for Mathematical Sciences for the support and hospitality during the programme {\em Operators, Graphs, Groups} when revision of this work was undertaken. This work was supported by EPSRC grant no EP/Z000580/1

\section{Preliminaries}\label{sec: prelims}

For standard definitions and facts about fundamental domains, see Chapter $5.6.4$ in \cite{DrutuKapovichGGT}. The facts about fundamental domains for \qr lattices appear in Raghunathan \cite{raghunathan1972discrete} and in Prasad's work on their rigidity~\cite{PrasadQRank1Rigidity}. In notations and generalities I follow Borel~\cite{BorelBookLinearAlgGroups} (algebraic groups) Helgason~\cite{helgasonLieGroups,helgasonDifferentialSymmetric} (Lie groups), and Eberline~\cite{eberlein1996geometry} (symmetric spaces).

\subsection{Generalities on Semisimple Lie Groups and their Lattices}\label{sec: prelims: generalities on lattices}

Let $G$ be a real centre-free semisimple Lie group without compact factors. A discrete subgroup $\G\leq G$ is a \emph{lattice} if $\G\backslash G$ carries a finite volume $G$-invariant measure. Up to finite index $\G$ is torsion free, so equivalently $\G$ is a lattice if and only if $\G\backslash X$ is a finite volume Riemannian manifold, where $X=G/K$ is the symmetric space of noncompact type corresponding to $G$. A lattice is \emph{irreducible} if its projection to every simple factor of $G$ is dense.  The group $G$ can be viewed as an algebraic group via the adjoint representation. If $G$ is of high rank, i.e.\ of $\mathbb{R}$-rank greater than $1$, then by the Margulis arithmeticity theorem every irreducible lattice of $G$ is \emph{arithmetic}. The $\mathbb{Q}$-rank of $\G$ is the $\mathbb{Q}$-rank of the $\mathbb{Q}$-structure associated to $(G,\G)$ by the arithmeticity theorem. If $G$ admits a $\rr$ factor, then a non-uniform irreducible lattice of $G$ is of \qr~\cite{PrasadQRank1Rigidity}. Finally, the group $G$ has Kazhdan's property (T) if and only if it does not admit an $\mathrm{SO}(n,1)$ or an $\mathrm{SU}(n,1)$ factor, and an irreducible lattice $\G\leq G$ has property (T) if and only if $G$ has property (T). One concludes: 

\begin{lem}\label{lem: cases}
Let $G$ be a real centre-free semisimple Lie group without compact factors, and $\G\leq G$ an irreducible lattice. Then at least one of the following occurs: (a) $G$ has property (T) (b) $\G$ is uniform (c) $\G$ is a non-uniform \qr lattice.
\end{lem}

\begin{cor}
Theorem~\ref{thm: main Intro} is an immediate result of Theorems~\ref{thm: main for T},~\ref{U: thm: main}, and~\ref{qr: thm: main for qr}.
\end{cor}

\subsection{Cusps and the Rational Tits Building} \label{sec: qr: prelims} \label{sec: qr: sym spc} \label{sec: qr: compact core and rational Tits Building}

A \emph{flat} in a symmetric space or Euclidean building $X$ is an isometrically embedded copy of a Euclidean space in $X$. The (real) \emph{rank} of $X$ is the maximal dimension of a flat in $X$. Since the geometry of \qr lattices resembles that of lattices in $\rr$,  the reader could for the most part simply have the image of the hyperbolic plane in mind. If one wishes to see flats that are not geodesics, then a product of two hyperbolic planes is enough. Even the product of the hyperbolic plane and $\mathbb{R}$ is helpful, albeit this space has a Euclidean factor hence is not of noncompact type.

\subsubsection{Basic Geometry of Symmetric Spaces of Noncompact Type}

\paragraph{\textbf{Visual Boundary.}} The visual boundary $X(\infty)$ of $X$ is the set of equivalence classes of geodesic rays: two geodesic rays are equivalent if their Hausdorff distance is finite. For a ray $\eta:[0,\infty)\rightarrow X$, $\eta(\infty)$ denotes the equivalence class of $\eta$ in $X(\infty)$. There are two natural topologies on $X(\infty)$ that will be of use. The \emph{cone topology} is the one given by viewing $X(\infty)$ as the set of all geodesic rays emanating from some fixed base point $x_0$, with topology induced by the unit tangent space at $x_0$. There is a natural topology on $\overline{X}:=X\cup X(\infty)$ such that $\overline{X}$ is the compactification of $X$ and the induced topology on $X(\infty)$ is the cone topology. Two `close' geodesic rays fellow travel: 

\begin{lem} \label{qr: prelims: sym spc: cor: small angle geodesic fellow travel}
Given time $T$ and $\varepsilon>0$, there is an angle $\alpha=\alpha(T,\varepsilon)$ so that if $\eta_1,\eta_2$ are two geodesic rays with $\eta_1(0)=\eta_2(0)=x$ for some $x\in X$ and $\measuredangle_x(\eta_1,\eta_2)\leq \alpha$ then $d_X\big(\eta_1(t),\eta_2(t)\big)<\varepsilon$ for all $t\leq T$.
\end{lem}

The second topology is given by the \emph{Tits metric}  $d_T$ on $X(\infty)$, which is defined as follows. Given a totally geodesic submanifold $Y\subset X$, let $Y(\infty)\subset X(\infty)$ be the subset of all points that admit a geodesic ray $\eta$ lying inside $Y$ (or, equivalently, those points that admit a ray lying at bounded Hausdorff distance to $Y$). For any $\xi_1,\xi_2\in X(\infty)$ there exists a flat $F\subset X$ such that $\xi_1,\xi_2\in F(\infty)$. Define $d_T(\xi_1,\xi_2)\in [0,\pi]$ to be the angle between two geodesic rays $\eta_1,\eta_2\subset F$ emanating from some point $x\in F$ and with $\eta_1(\infty)=\xi_1,\eta_2(\infty)=\xi_2$. This is a well defined metric on $X(\infty)$, called the Tits metric. The pair $\big(X(\infty),d_T\big)$ is a geodesic metric space, and isometries of $X$ act on it by isometries. The following relates the cone and Tits topologies.

\begin{prop}[Section $3.1$ in \cite{eberlein1996geometry}]\label{qr: prelims: sym spc: prop: semicontinuity of angular metric}
Let $X$ be a symmetric space of noncompact type. The Tits metric is semicontinuous with respect to the cone topology: for any $\xi,\zeta\in X(\infty)$ and every $\varepsilon>0$, there exists neighbourhoods of the cone topology $U,V\subset X(\infty)$ of $\xi$ and $\zeta$ respectively such that for all $\xi'\in U,\zeta'\in V$ one has $$d_T(\xi',\zeta')\geq d_T(\xi,\zeta)-\varepsilon$$

Moreover, for any flat $F\subset X$, the cone topology and the Tits topology coincide on $F(\infty)$. 
\end{prop}

\paragraph{\textbf{Busemann Functions, Horoballs and Horospheres.}}

A \emph{Busemann function} on $X$ is any function of the form 
$$b_\eta(x)=\lim_{t\rightarrow \infty}d\big(x,\eta(t)\big)-t$$

for some geodesic ray $\eta$ of $X$. For a Busemann function $b_\eta$, a \emph{horoball} $\hb\subset X$ is an open sublevel set of $b_\eta$ and a \emph{horosphere} $\h\subset X$ is a level set of $b_\eta$. The point $\eta(\infty)$ is called the \emph{base point} of $\hb$ and $\h$. This notion is well defined, since two equivalent geodesic rays $\eta_1,\eta_2$ give rise to Busemann functions which differ by a constant, i.e.\  $b_{\eta_1}-b_{\eta_2}=C$ for some $C\in \mathbb{R}$. Every horosphere $\h$ is the topological boundary of a horoball $\hb$ so that $\h=\partial\hb$, in which case $\h$ is referred to as the \emph{bounding horosphere} of $\hb$ and $\hb$ is the horoball \emph{bounded} by $\h$. For every choice of $x\in X,\xi\in X(\infty)$ there is a unique horosphere $\h$ based at $\xi$ with $x\in \h$, which I denote $\h(x,\xi)$ with bounded horoball $\hb(x,\xi)$. The following proposition collects some basic properties.

\begin{prop}[Proposition $1.10.5$ in \cite{eberlein1996geometry}] \label{qr: prelims: sym spc: prop: parallel horospheres and other properties}
Let $x\in X, \xi\in X(\infty)$, $\h=\h(x,\xi)$, $\hb=\hb(x,\xi)$ and $b$ the Busemann function based at $\xi$ with $b(x)=0$.
\begin{enumerate}
    \item For any point $y\in X$, let $\eta$ be the bi-infinite geodesic determined by the geodesic $[y,\xi)$. Then $P_\h(y)=\eta\cap\h$, where $P_\h(y)$ is the unique point closest to $y$ on $\h$.
    
    \item For any point $y\in X$, $b(y)=\pm
    d\big(y,P_\h(y)\big)$. Moreover, $b(y)$ is negative if and only if $y\in \hb$.  
    
    \item If $x'\in X$, then the horospheres $\h=\h(x,\xi)$ and $\h'=\h(x',\xi)$ are equidistant: if $y\in \h,y'\in \h'$, then $d(y,\h')=d(y',\h)$. Such horospheres are called \emph{parallel}.
\end{enumerate}
\end{prop}

A Busemann function $b_\eta$ thus naturally determines a filtration of $X$ by the codimension $1$ manifolds $\{\h_t\}_t\in\mathbb{R}$. Since flats are isometrically embedded, the restriction of this filtration to any flat $F$ such that $\eta(\infty)\in F(\infty)$ is the filtration of $F$ by hyperplanes orthogonal to $\eta$.

\begin{rmk}\label{qr: prelims: rmk: heintze-im hof results about metric on horospheres}
The stabilizer of a point $\xi\in X(\infty)$ acts transitively on the set of horospheres based at $\xi$, so every two such horospheres are isometric. Moreover, there is a close relation between the induced metrics on horospheres with the same base point. Briefly, if $d_\h$ denotes the induced distance on a horosphere $\h\subset X$, then $d_\h\big(P_\h(x),P_\h(y)\big)$ for any two points $x,y\in \h'$ can be bounded uniformly below and above as a function of the distance $d_{\h'}(x,y)$ and the curvature bounds on $X$. See Heintze-Im hof~\cite{GeometryHorospheres} for precise statements. 
\end{rmk}

Using the above properties one can show that two horospheres are parallel if and only if they are based at the same point (see also Proposition 3.4.3 in~\cite{eberlein1996geometry}). In particular for every $x\in X,\xi\in X(\infty)$ it holds that  $\mathrm{Stab}_G\big(\h(x,\xi)\big)\subset \mathrm{Stab}_G(\xi)$. In addition, if $\h,\h'$ are two parallel horospheres based at the same $\xi\in X(\infty)$ and $C\subset\h$ is a cocompact metric lattice in $\h$, then $\pi_{\h'}(C)$ is a cocompact metric lattice in $\h'$.

\paragraph{\textbf{Parabolic and Horospherical Subgroups.}}
The isometries of $X$ are classified into elliptic, hyperbolic, and parabolic isometries. Most significant for this paper are the parabolic isometries, i.e.\ those $g\in G$ whose displacement function $x\mapsto gx$ does not attain a minimum in $X$. Every such isometry fixes (at least) one point in $X(\infty)$. A group $P\leq \mathrm{Isom}(X)$ is called \emph{geometrically parabolic} if it is of the form  $G_\xi:=\mathrm{Stab}_G(\xi)$ for some $\xi\in X(\infty)$. As $G_\xi$ act transitively on $X$, it acts transitively on the set of geodesic rays in the equivalence class of $\xi$. The same holds also for the identity component $G_\xi^\circ$. An element $g\in G_\xi$ permutes the set of horoballs based at $\xi$, inducing  a translation on the corresponding  filtration $\{\h_t\}_{t\in\mathbb{R}}$ of $X$ by horospheres based at $\xi$. Put differently, for every $g\in G_\xi$ there is $l(g)\in \mathbb{R}$ such that $g\h_t=\h_{t+l(g)}$.
In particular, the group $G_\h:=\mathrm{Stab}_G(\h)$ acts transitively on $\h$, and the same holds for $G_\h^\circ$.

The following is a fundamental structure theorem for geometrically parabolic groups. Denote $\mf{g}:=\mathrm{Lie}(G)$, and let $\mf{g}=\mf{k}\oplus\mf{p}$ be a Cartan decomposition corresponding to the maximal compact subgroup $K\leq G$. Recall that the Lie exponential map $\exp:\mf{g}\rightarrow G$ gives rise to a family of $1$-parameter subgroups of the form $\exp(tX)$ for each $X\in \mf{p}$.

\begin{prop}[Proposition $2.17.3$ in \cite{eberlein1996geometry}]\label{qr: prelims: algebra: prop: stabilizer at infinity by limit of conjugations}
Let $x\in X(\infty)$, and let $X\in \mf{p}$ be the tangent vector at $x_0$ to the unit speed geodesic $[x_0,\xi)$. Let $h_\xi^t$ be the $1$-parameter subgroup defined by $t\mapsto \exp(tX)$. Then an element $g\in G$ fixes $\xi$ if and only if $\lim_{t\rightarrow \infty}h_\xi^{-t}gh_\xi^t$ exists.
\end{prop}

\begin{prop}[Langlands Decomposition, Propositions $2.17.5$ and $2.17.25$ in \cite{eberlein1996geometry}]\label{qr: prop: Eberline Generalized Iwasawa}
Let $\xi\in X(\infty)$ and $h_\xi^t$ as in Proposition~\ref{qr: prelims: algebra: prop: stabilizer at infinity by limit of conjugations}. Let $F$ be a flat containing $[x_0,\xi)$ and $A\leq G$ the maximal abelian subgroup such that $Ax_0=F$. Denote $G_\xi:=\mathrm{Stab}_G(\xi)$, and define $T_\xi:G_\xi\rightarrow G$ by $g\mapsto\lim_{n\rightarrow \infty} h_\xi^{-n}gh_\xi^n$. Then $T_\xi$ is a homomorphism, and there are subgroups $N_\xi,A_\xi,K_\xi\leq G_\xi$ such that: 

\begin{enumerate}
    \item $A_\xi=\exp\big(Z(X)\cap \mf{p}\big)$, where $Z(X)$ is the centralizer of $X$ in $\mf{g}$. Moreover, every element $a\in A_\xi$ lies in some conjugate $A^g=gAg^{-1}$ with the property that  $[x_0,\xi)\subset F^g:=A^gx_0$.
    
    \item $K_\xi\leq K=\mathrm{Stab}_G(x_0)$ is the compact subgroup fixing the bi-infinite geodesic determined by $[x_0,\xi)$.
    
    \item $K_\xi A_\xi=A_\xi K_\xi$.
    
    \item $N_\xi=\mathrm{Ker}(T_\xi)$. It is a connected normal subgroup of $G_\xi$.

    \item $G_\xi=N_\xi A_\xi K_\xi$, and the indicated decomposition of an element is unique.
    
    \item $G= N_\xi A_\xi K$, and the indicated decomposition of an element is unique. In case $\xi$ is a regular point at $X(\infty)$, this decomposition is the Iwasawa decomposition. 
    
    \item $G_\xi$ has finitely many connected components, and $G_\xi^\circ=(K_\xi A_\xi)^\circ N_\xi$.
    
    \item $G_\xi$ is self normalizing. 

\end{enumerate}
\end{prop}

For the definition of regular and singular tangent vectors or points at $X(\infty)$, see Definition~\ref{qr: sym spc: def: regular geodesics}.

Viewing $G$ as an algebraic group, the geometrically parabolic subgroups are exactly the (algebraically) non-trivial parabolic subgroups, i.e.\ proper subgroups of $G$ that contain a normalizer of a maximal unipotent subgroup. Proposition~\ref{qr: prop: Eberline Generalized Iwasawa} is a geometric formulation of the algebraic Langlands decomposition of parabolic groups. Recall that a horospherical subgroup is the unipotent radical of a non-trivial parabolic group, or equivalently groups of the form $U_g:=\{u\in G\mid \lim_{n\rightarrow \infty}g^{-n}ug^n=id_G\}$. The latter implies that $N_\xi$ is a horospherical subgroup of $G$.

\paragraph{\textbf{Limit Set.}}
An important set associated to a discrete group $\Delta\leq \operatorname{Isom}(X)$ is the \emph{limit set} $\mathcal{L}_\Delta:=\overline{\Delta x} \cap X(\infty)$ (the closure in the compactification $\overline{X}=X\cup X(\infty)$). It is clear that $\mathcal{L}_\Delta$ does not depend on the choice of $x\in X$. The limit set of any lattice is always the entire $X(\infty)$. In fact, much more is true: 

\begin{defn}
Let $\Delta\leq G=\mathrm{Isom}(X)$, and $SX$ the unit tangent bundle. A vector $v\in SX$ is \emph{$\Delta$-periodic} if there is $\delta\in \Delta$ and $s>0$ such that $\delta\eta(t)=\eta(t+s)$ for all $t\in \mathbb{R}$, where $\eta$ is the bi-infinite geodesic determined by the vector $v$. 

For a flat $F\subset X$ (including geodesics), denote $\Delta_F:=\{\delta\mid \delta F=F\}$. The flat $F$ is called a \emph{$\Delta$-periodic flat} if there exists a compact set $C\subset F$ such that $\Delta_F C=F$.

\end{defn}

\begin{prop}[Propositions $4.7.3., 4.7.5, 4.7.7$ in \cite{eberlein1996geometry}, Lemma $8.3'$ in \cite{mostow1973strong}]\label{qr: prelims: sym spc: prop: lattice periodic vecotrs are dense}
If $\G\leq G=Isom(X)$ is a lattice, then:
\begin{enumerate}
    \item The subset in $SX$ of $\G$-periodic vectors is dense. 
    
    \item Let $F\subset X$ be any flat, $\eta$ any bi-infinite geodesic in $F$, and denote $v=\dot{\eta}(0)\in SX$ the initial velocity vector.  There is a sequence $v_n\in SX$ of regular vectors such that
    \begin{enumerate}
        
        \item $\lim_{n\rightarrow \infty}v_n=v$.
        
        \item The bi-infinite geodesics $\eta_n$ determined by $v_n$ are all $\G$-periodic.
        
        \item Denote by $F_n$ the (unique) flat containing $\eta_n$. Each $F_n$ is $\G$-periodic. 
    \end{enumerate}
    
    Put differently, the set of $\G$-periodic flats is dense in the set of flats of $X$. 
\end{enumerate}
\end{prop}

Points in the limit set of a group are classified according to how the orbit approaches them.  

\begin{defn}
Let $\Delta\leq G$ be a discrete subgroup, $\xi\in \mathcal{L}_\Delta$. 
\begin{enumerate}
    \item The point $\xi$ is called \emph{conical} if for some (hence any) $x$ and some (hence every) geodesic ray $\eta$ with $\eta(\infty)=\xi$ there is a number $D=D(x,\eta)$ such that for every $T\in \mrp$ there is $t>T$ for which $B\big(\eta(t),D\big)\cap \Delta x\ne\emptyset$. Since $\Delta$ is discrete, this is equivalent to $\Delta x\cap \mn_D(\eta)$ being infinite. 
    
    \item The point $\xi$ is called \emph{horospherical} if for every horoball $\hb$ based at $\xi$ and every $x\in X$, $\Delta x\cap \hb$ is non-empty. In particular, a conical limit point is horospherical.
    
    \item The point $\xi$ is \emph{non-horospherical} if it is not horospherical.
\end{enumerate}
\end{defn}

As a corollary of Proposition~\ref{qr: prelims: sym spc: prop: lattice periodic vecotrs are dense}, one has: 

\begin{cor}\label{qr: prelims: sym spc: cor: lattice conical limit points are dense}
If $\G\leq \mathrm{Isom}(X)$ is a lattice, then
\begin{enumerate}
    \item The set of $\G$-conical limit points is dense in $X(\infty)$ with the cone topology. 
    
    \item $\mathcal{L}_\G=X(\infty)$.
   
\end{enumerate}
\end{cor}

I finish this section with some results on \emph{geometrically finite} subgroups of isometries in $\rr$.

\begin{defn}
Let $X$ be a $\rr$ symmetric space and $\Delta\leq Isom(X)$ a discrete subgroup. Denote by $\overline{\mathrm{Hull}(\Delta)}$ the closed convex hull in $\overline{X}=X\cup X(\infty)$ of the limit set $\mathcal{L}_\Delta$, and $\mathrm{Hull}(\Delta)=X\cap \overline{\mathrm{Hull}(\Delta)}$. By virtue of negative curvature, $\mathrm{Hull}(\Delta)$ is the union of all geodesics $\eta$ such that $\eta(\infty),\eta(-\infty)\in \mathcal{L}_\Delta$. 
The \emph{convex core} of $\Delta$ is defined to be $\Delta\backslash \mathrm{Hull}(\Delta)\subset \Delta\backslash X$, i.e., the quotient of $\mathrm{Hull}(\Delta)$ by the $\Delta$-action.

\end{defn}

\begin{defn}[Bowditch~\cite{BOWDITCHGeometricalFinitness}, see Theorem $1.4$ in \cite{KapovichLiuGeometricFI}]\label{qr: prelims: sym spc: def: geometrically finite group}

Let $X$ be a $\rr$ symmetric space, i.e.\ a symmetric space of pinched negative curvature. A discrete group $\Delta\leq G=\mathrm{Isom}(X)$ is \emph{geometrically finite} if for some $\delta>0$, the uniform $\delta$-neighbourhood in $\Delta\backslash X$ of the convex core $\mn_\delta\big(\Delta\backslash \mathrm{Hull}(\Delta)\big)$ has finite volume, and there is a bound on the orders of finite subgroups of $\Delta$. A group is \emph{geometrically infinite} if it is not geometrically finite. 

\end{defn}

\begin{cor}\label{qr: prelims: sym spc: cor: geometric finite and whole limit set implies lattice}
Let $X$ be a $\rr$ symmetric space. If $\Delta\leq \mathrm{Isom}(X)$ is geometrically finite and admits $\mathcal{L}_\Delta=X(\infty)$, then $\Delta$ is a lattice in $\mathrm{Isom}(X)$.
\end{cor}

Sublinear distortion does not affect the limit set, as the following lemma shows. 

\begin{lem}\label{qr: lem: limit set of Lambda equals limit set of Gamma}
Let $\G,\La\leq G$ be discrete subgroups and $u:\mrnn\rightarrow\mrp$ a sublinear function. If $\G\subset \mn_u(\La)$ then $\mathcal{L}_\G\subset \mathcal{L}_\La$, i.e.\  every $\G$-limit point is a $\La$-limit point. In particular, if $\G$ is a lattice then $\mathcal{L}_\La=X(\infty)$.
\end{lem}

\begin{proof}
I show that given a point $\xi\in X(\infty)$ and a sequence $\g_n\in \G$ such that $\g_n x_0\rightarrow \xi$ (in the cone topology on $\overline{X}$), there is a corresponding sequence $\la_n\in \La$ with $\la_n x_0\rightarrow \xi$. Define $\la_n:=\la_{\g_n}$ to be the closest point to $\g_n$ in $\La$, and denote $x_n:=\g_n x_0, x_n'=\la_n x_0$. Let $\eta_n:=[x_0,x_n]$ and $\eta_n':=[x_0,x_n']$ be unit speed geodesics, $T_n:=d(x_0,x_n)$ so that $\eta_n(T_n)=x_n$. 

Convergence in the cone topology $x_n\rightarrow \xi$ is equivalent to the fact that the geodesics $\eta_n$ converge to $\eta:=[x_0,\xi)$ uniformly on compact sets. This means in particular that $T_n\rightarrow\infty$.  Non-positive curvature guarantees that the functions $F_n(t):=d\big(\eta(t),\eta_n(t)\big), F_n'(t):=d\big(\eta(t),\eta_n'(t)\big)$ and $G_n:=d\big(\eta_n(t),\eta_n'(t)\big)$ are convex. Since $G_n(0)=F_n(0)=F_n'(0)=0$, each function is either constant $0$ or monotonically increasing, so proving uniform convergence of $\eta_n'$ to $\eta$ amounts to proving $\lim_n F_n'(T)=0$ for every $T\in \mrnn$.

Triangle inequality gives $F_n'(T)\leq F_n(T)+G_n(T)$, and by assumption $\lim_n F_n(T)=0$. Notice that $G_n(T_n)=d(x_n,x_n')\leq u(|\g_n|)=u(T_n)$. Writing $T=\frac{T}{T_n}\cdot T_n$,  convexity of $G_n$ implies

$$G_n(T)\leq (1-\frac{T}{T_n})G_n(0)+\frac{T}{T_n}G_n(T_n)\leq 0+\frac{T}{T_n}u(T_n)=T\cdot\frac{u(T_n)}{T_n}$$

As $\lim_n T_n=\infty$ it follows from sublinearity that $\lim_n G_n(T)=0$. I conclude that $\eta_n'$ converge to $\eta$ uniformly on compact sets, therefore $\xi$ lies in the limit set of $\La$.
\end{proof}

\begin{rmk}
In Section~\ref{sec: qr: cocompact horosphere} I obtain stronger results when $\G$ is a \qr lattice. The strength of Lemma~\ref{qr: lem: limit set of Lambda equals limit set of Gamma} is that it does not assume anything on $\G$ other than that it is sublinearly covered by $\La$.
\end{rmk}

\subsubsection{Cusps, Compact Core, and the Rational Tits Building}\label{sec: cusps compact core rational Tits}

In this section I present some of the structure theory of non-compact quotients of $X$. The focus is on the structure of `cusps' in non-compact finite volume quotients of symmetric spaces, and the `location' of cusps on the visual boundary. 

\paragraph{\textbf{Cusps and Compact Core.}}

Let $\G\leq G$ be a non-uniform lattice, and consider the finite volume locally symmetric space $\mathcal{V}=\G\backslash X$. The term `cusps' is an informal name given to those areas in $\mathcal{V}$ that `escape to infinity'. Another description is that cusps are the ends of the complement of a large enough compact set in $\mathcal{V}$. In strictly negative curvature, i.e.\ in $\rr$ locally symmetric spaces, cusps have a precise description as submanifolds of the form $C\times \mrnn$ for a compact manifold $C$, and metrically $(C,t)$ gets narrower as $t\rightarrow \infty$. There are finitely many cusps, each corresponding to a point at $X(\infty)$ called a `parabolic point'. See e.g.\ Introduction in \cite{BallmannGromovSHroeder} or \cite{LatticesEberline}. A fundamental feature of the cusps is that one can `chop' them off the quotient manifold $V$ to obtain a nice compact subset. This could be done in such a way so that:

\begin{enumerate}
    \item The lifts of the chopped parts to the universal cover $X$ are disjoint.
    
    \item Each cusp is covered in $X$ by the $\G$-orbit of a horoball, that is, the lift of a cusp is the $\G$-orbit of a horoball. The respective base points are called \emph{parabolic points} of $\G$ in $X(\infty)$.
    
    \item $\G$ acts on $X \setminus \big(\bigcup_{i\in I} \hb_i\big)$ cocompactly, where $\{\hb_i\}_{i\in I}$ is the set of horoballs coming from the lifts of cusps.
\end{enumerate} 

Since there are only finitely many cusps and $\G$ discrete, the set $I$ is countably infinite. See Section $12.6$ in \cite{DrutuKapovichGGT} for an illustration in the case where $X$ is the real hyperbolic space $\mathbb{H}^n$. Formally, one has:

\begin{thm}[Theorem $3.1$ in \cite{LatticesEberline}, see also Introduction therein]\label{qr: prelims: sym spc: thm: Eberline compact core rank 1}
Assume $X$ is of $\rr$, and $\G\leq G$ a non-uniform lattice. The space $\mathcal{V}=\G\backslash X$ has only finitely many (topological) ends and each end is parabolic and Riemannian collared. In particular, each cusp is a quotient of a horoball $\hb$ based at a parabolic limit point $\xi$ such that $\G\cap G_\xi$ acts cocompactly on $\h=\partial\hb$.
\end{thm}

For symmetric spaces of high rank, a similar construction is available~\cite{LeuzingerLocSymSpaces}. By removing a countable family of horoballs from $X$, one obtains a subspace on which $\G$ acts cocompactly. There are two main differences from the situation in $\rr$. One is that an orbit map $\g\mapsto \g x$ is a quasi-isometric embedding of $\G$, with its word metric, into $X$. This property plays a preliminary role in the proofs of quasi-isometric rigidity for non-uniform lattices in high rank. The precise statement is the following.

\begin{thm}[Lubozki-Mozes-Raghunathan, Theorem A in \cite{lubotzky_mozes_raghunathan_2000}]\label{prelims: thm: lubozki mozes raghunathan LMS QIE}
Let $G$ be a semisimple Lie group of high $\mathbb{R}$-rank, $d_G$ a left invariant metric induced from some Riemannian metric on $G$. Let $\G$ an irreducible lattice, $d_\G$ a word metric on $\G$. Then $d_{G\restriction \G\times \G}$ and $d_\G$ are Lipschitz equivalent. 
\end{thm}

The second difference in high rank is that the horoballs cannot be made disjoint, except in the special case of \qr lattices. Recall that the $\mathbb{Q}$-structure of $(G,\G)$ is a $\mathbb{Q}$-structure on $G=\mathbf{G}(\mathbb{R})$ in which $\G$ is an arithmetic lattice. The following theorem sums up the relevant properties for \qr lattices. 

\begin{thm} [Theorem $4.2$ and Proposition $2.1$ in \cite{Leuzinger1995Exhaustion}, see also Remarks $3$ and $4$ in \cite{LeuzingerLocSymSpaces}, Section $13$ in \cite{raghunathan1972discrete}, and Proposition $2.1$ in \cite{PrasadQRank1Rigidity}]\label{qr: prelims: sym spc: thm: Leuzinger compact core}
Assume $G$ is of high rank, and $\G\leq G$ an irreducible torsion-free non-uniform lattice. On the locally symmetric space $\mathcal{V}=\G\backslash X$ there exists a continuous and piece-wise real analytic exhaustion function
$h: \mathcal{V} \rightarrow [0, \infty)$ such that, for any $s>0$, the sublevel set $\mathcal{V}(s) := \{h < s\}$ is
a compact submanifold with corners of $\mathcal{V}$. Moreover the boundary of $\mathcal{V}(s)$,
which is a level set of $h$, consists of projections of subsets of horospheres
in $X$.

The $\G$-action on the above set of horospheres has finitely many orbits, and the following are equivalent: 
\begin{enumerate}
    \item The corresponding horoballs bounded by these horospheres can be taken to be disjoint.
    \item For each (any) such horosphere $\h$ the action of $\G\cap \mathrm{Stab}_G(\h)$ on $\h$ is cocompact.
    \item The $\mathbb{Q}$-structure of $(G,\G)$ has \qrc i.e.\ $\G$ is a \qr lattice.
\end{enumerate}
\end{thm}

\begin{defn}
In the setting of Theorem~\ref{qr: prelims: sym spc: thm: Eberline compact core rank 1} and  Theorem~\ref{qr: prelims: sym spc: thm: Leuzinger compact core}, the horoballs and horospheres that appear in the statement are called (global) \emph{ horoballs (horospheres) of $\G$}. Their base points are called \emph{parabolic limit points of $\G$}. For any choice of $s>0$ such that $\mathcal{V}(s)$ is connected, the lift of $\mathcal{V}(s)$ to $X$ is called a \emph{compact core of $\G$}, a $\G$-invariant subset of $X$ on which $\G$ acts cocompactly. 
\end{defn}

The proof of the key Proposition~\ref{prop: Lambda cocompact horosphreres Intro} makes essential use of the geometric properties of \qr lattices arising from the above description:

\begin{cor}  \label{qr: prelims: sym spc: cor: every ball intersects finitely many horoballs of Gamma}
Let $\G\leq G$ be a \qr lattice, $x\in X$, and $\xi$ a parabolic limit point. There is a unique horosphere $\h$ based at $\xi$ such that both following conditions hold: 
\begin{enumerate}
    \item $\G x\cap \h$ is a cocompact metric lattice in $\h$.
    
    \item $\G x\cap \hb=\emptyset$, where $\hb$ is the horoball bounded by $\h$.
\end{enumerate} 

Fixing $x$ and letting $\xi$ vary, call such $\h$ an \emph{$x$-horosphere of $\G$}, and $\hb$ an \emph{$x$-horoball of $\G$}. One has: 
\begin{enumerate}
    
    \item For every $C$ there is a bound $K=K(x,C)$ so that $B(x,C)$ intersects at most $K$ $x$-horospheres of $\G$.

    \item There is $D=D(x)>0$ such that $\h\subset \mn_D(\G x\cap \h)$ for any $x$-horosphere $\h$ of $\G$. The constant $D$ is called \emph{the compactness number of $(\G, x)$}.
    
    \item There is a number $N=N(\G)$ such that every point $x\in X$ admits exactly $N$ $x$-horospheres of $\G$ that intersect $x$. These are called \emph{the horospheres of $(\G, x)$}.
    
\end{enumerate}

\end{cor}

\begin{proof}
Let $x\in X$. Since $\G$ is a \qr lattice, the stabilizer in $\G$ of a parabolic limit point $\xi\in X(\infty)$ acts cocompactly on each horosphere based at $\xi$, and in particular on $\h_x:=\h(x,\xi)$. Let $D(x,\h)$ be such that $\h_x\subset\mn_D(\G x\cap \h_x)$. A priori $D(x)$ depends on $\h$, but the fact that $\G$ acts by isometries implies that for every horosphere of the form $\h':=\g\h=\h(\g x,\g\xi)$ for some $\g\in \G$, one has $\h'\subset\mn_D(\G x\cap \h')$. So $D(x)$ depends only on the $\G$-orbit of $\h$. Since $\G$ acts on the set of parabolic limit points with finitely many orbits (that is to say $\G\backslash X$ has finitely many cusps) one may take $D=D(x)$ to be the maximum of the respective bounds on each orbit. This gives the required compactness number.

Thinking of horoballs of $\G$ as lifts of ends of the complement of some compact subset of $\mathcal{V}=\G\backslash X$, one sees that there is some horoball $\hb$ based at $\xi$ so that $\G x\cap \hb=\emptyset$. Denote $\h=\partial \hb$, and $y=P_{\overline{\hb}}(x)\in \h$ be the projection on the closed convex set that is the closure of the horoball $\hb$. Finally, Let $\eta=[x,y]$ and denote $l=d(x,y)$. The existence of the required horosphere is equivalent to the fact that the following non-empty set admits a maximum:

$$\h_{x,\xi}:=\{t\in [0,l]\mid\G x\cap \h\big(\eta(t),\xi\big)\ne\emptyset\}$$

Indeed $0\in \h_{x,\xi}$ and it is a bounded set, so it admits a supremum $T$. Moreover, this set is discrete: if towards contradiction $t$ were an accumulation point, then for any small $\varepsilon>0$ the geodesic segment $\eta_{\restriction(t-\varepsilon,t+\varepsilon)}$ would intersect infinitely many horospheres based at $\xi$ which $\G x$ intersects $D$-cocompactly. Orbit points on different horospheres are in particular different points, therefore the set $B\big(\eta(t),D+\varepsilon\big)\cap\G x$ would be infinite, contradicting discreteness of $\G$. I conclude that $T$ is a maximum, and that $\h\big(\eta(T),\xi\big)$ is the unique desired horosphere.

The argument above generally shows that there cannot be an accumulation point in $X$ of horospheres that $\G x$ intersects $D$-cocompactly. In particular, for every $C$, the ball $B(x,C)$ intersects only finitely many $x$-horospheres of $\G$, say $K(C)$. 

Finally note that the horoballs of $\G$ are the $\G$-translates of finitely many horoballs. In the terminology of the statement, infinitely many horospheres of $(\G, x)$ imply that infinitely many of them are in the same $\G$-orbit. Suppose these are $\{\h_n\}_{n\in\mathbb{N}}$, with base points $\xi_n$ that are evidently pairwise different. Let  $\g_n\in\G$ for which $\g_n\h_1=\h_n$. Since $\G\cap \mathrm{Stab}_G(\h_n)$ acts cocompactly on $\h_n$, there is $\g_n'\in \G\cap \mathrm{Stab}_G(\h_n)$ that maps $\g_n x$ to $B(x,D)$. Discreteness of $\G$ implies that the set $\g_n'\g_n x$ is finite, and since $\G$ is torsion-free this implies $\g_n'\g_n=\g_m'\g_m$ for infinitely many $n,m\in\mathbb{N}$. However $\g_ n'\g_n\xi_1=\xi_n$ and $\xi_n\ne\xi_m$ for all $n\ne m$, a contradiction.
\end{proof}

\begin{rmk}
For the most part, I am interested in a fixed base point $x_0$ and the $x_0$-horospheres and horoballs. By a slight abuse of terminology I omit $x_0$ and call these objects \emph{horospheres of $\G$} and \emph{horoballs of $\G$}, respectively, denoting the associated cocompactness number $D_\G$.
\end{rmk}

The next lemma uses Corollary~\ref{qr: prelims: sym spc: cor: every ball intersects finitely many horoballs of Gamma} in order to upgrade a metric lattice of $\h$ to a metric lattice coming from $\mathrm{Stab}_G(\h)$ only.

\begin{lem}\label{qr: lem: cocompact metric lattice of horospheremay be intersected with the stabilizer of horosphere}
Assume that a torsion-free discrete group $\Delta\leq G=\mathrm{Isom}(X)$ admits the following: 
\begin{enumerate}
    \item There is a bound $N$ such that at each point $x\in \Delta x_0$ there are at most $N$ horoballs that are tangent to $x$ and are $\Delta$-free, i.e.\ whose interior does not intersect $\Delta x_0$.
    \item There is a horosphere $\h\subset X$ such that
        \begin{enumerate}
            \item The set $\Delta x_0\cap \h$ is a cocompact metric lattice in $\h$.
            \item The open horoball $\hb$ bounded by $\h$ is $\Delta$-free.
        \end{enumerate}
    \end{enumerate}  

Then $\big(\Delta\cap \mathrm{Stab}_G(\h)\big)\cdot x_0$ is also a cocompact metric lattice in $\h$.
\end{lem}

\begin{proof}
This is the Pigeonhole Principle. Fix $x\in\Delta x_0\cap \h$, and let $\{\hb_i\}_{i\in \{1,\dots,N\}}$ be the finite set of $\Delta$-free horoballs tangent to $x$. Assume w.l.o.g that $\h$ is the bounding horosphere of $\hb_1$. Fix a $\Delta$-orbit point $\delta_0 x_0\in \Delta x_0\cap \h$. Up to translating by some element of $\Delta$, I may assume $x=x_0$. For any other $\Delta$-orbit point $\delta x_0\in \Delta x_0\cap \h$ let $i(\delta)\in \{1,\dots,N\}$ be the index of the horoball $\delta^{-1} \hb_1$ (since the action is by isometries, hypothesis $(1)$ indeed implies $\delta^{-1} \hb_1\in \{\hb_i\}_{i\in \{1,\dots,N\}}$). Define $\delta_i$ to be the element in $\Delta$ for which:
\begin{enumerate}
    \item $i(\delta_i)=i$, i.e.\ $\delta_i^{-1}(\hb_1)=\hb_i$
    \item $d(\delta_i x_0,x_0)$ is minimal among all such $\delta\in \Delta$ which satisfy $i(\delta)=i$.
\end{enumerate}

For some indices $i\in \{1,\dots,N\}$ there is a $\delta\in \Delta$ with $i=i(\delta)$, while for others there might not be. I only care about those $i$ for which there is such $\delta$. Assume w.l.o.g that these are $i\in \{1,\dots M\}$ for $M\leq N$, and let $L:=\max_{1\leq i\leq M}\{d(x_0,\delta_i x_0)\}<\infty$.  For such an index $i_0$, the $\Delta$-orbit points that share the same $i(\delta)=i_0$ are in the same $\Delta\cap \mathrm{Stab}_G(\h)$ orbit, namely

$$\{\delta x_0\mid i(\delta)=i_0\}\subset \big(\Delta\cap \mathrm{Stab}_G(\h)\big)\cdot \delta_{i_0} x_0$$

Indeed, if $\delta^{-1}{\hb_1}=\hb_{i_0}$ then by definition $\delta\delta_{i_0}^{-1}\hb_1=\delta \hb_{i_0}=\hb_1$, hence $\delta\delta_{i_0}^{-1}\in \mathrm{Stab}_G(\h)$ is an element mapping $\delta_{i_0}x_0$ to $\delta x_0$. Let now $\delta x_0\in \h$. Its distance from the orbit $\big(\Delta\cap \mathrm{Stab}_G(\h)\big)\cdot x_0$ is $\big(\Delta\cap \mathrm{Stab}_G(\h)\big)$-invariant, therefore

$$d\Big(\delta x_0,\big(\Delta\cap \mathrm{Stab}_G(\h)\big)\cdot x_0\Big)=d\Big(\delta_{i(\delta)} x_0,\big(\Delta\cap \mathrm{Stab}_G(\h)\big)\cdot x_0\Big)\leq d(\delta_{i(\delta)},x_0)$$

The right-hand side is uniformly bounded by $L$, proving that 

$$(\Delta x_0\cap \h)\subset\mn_{L}\Big(\big(\Delta\cap \mathrm{Stab}_G(\h)\big)\cdot x_0\Big)$$

The fact that $\Delta x_0\cap \h$ is a cocompact metric lattice in $\h$ renders $(\Delta\cap \mathrm{Stab}_G(\h)\big)\cdot x_0$ a cocompact metric lattice in $h$ as well, as claimed. 

\end{proof}

\paragraph{\textbf{Real and Rational Tits Buildings.}}

The location of the parabolic points in $X(\infty)$ also plays an important role in the geometry of $X$. In case $\G$ is an arithmetic lattice, the natural framework to consider these points is the so called \emph{rational Tits building}. This is a building structure on the subset of parabolic points at $X(\infty)$, sometimes referred to as `rational points'. They are exactly those points in $X(\infty)$ whose stabilizers  are $\mathbb{Q}$-defined (algebraic) parabolic groups of $G$ (see Section~\ref{sec: qr: geometry to algebra} for more details). I present this object, denoted $\wqg$, together with the more familiar \emph{real Tits building} structure on $X(\infty)$ with the Tits metric. The main goal is to present the results of Hattori~\cite{HattoriLimitSet}, that give a good description of the rational Tits building in terms of conical and horospherical limit points. In case $G$ is of $\rr$, by $\wqg$ I mean the (countable) set of parabolic limit points of $\G$ (so that $X(\infty)\setminus\wqg$ is comprised of conical limit points only, see Theorem~\ref{qr: sym spc: thm: in rank 1 all horospherical limit points are conical}).

\begin{defn}\label{qr: sym spc: def: regular geodesics}
A geodesic ray $\eta$ is said to be \emph{regular} if it is contained in a unique maximal flat $F\subset X$. The corresponding $\eta(\infty)\in X(\infty)$ and initial velocity  $\dot{\eta}(0)$ are also called regular. A point $\xi\in X(\infty)$ or a tangent vector are \emph{singular} if they are not regular. Regularity does not depend on the choice of a representative geodesic ray $\eta$ for $\xi$. Denote by $\mathcal{S}\subset X(\infty)$ the subset of singular points at $X(\infty)$. A \emph{Weyl chamber} of $X(\infty)$, or an open \emph{spherical chamber}, is any connected component in the Tits topology of $X(\infty)\setminus \mathcal{S}$.

\end{defn}

\begin{prop}[Propositions $2.2$ and $3.2$ in \cite{JiSymmetricSpacesBuildings} and Section $8$ in \cite{BallmannGromovSHroeder}]\label{qr: prelims: sym spc: prop: real building structure}
The Weyl chambers induce a simplicial complex structure on $X(\infty)$ that is a \emph{spherical Tits building}. The apartments of the building are exactly the sets of the form $F(\infty)\subset X(\infty)$ for all flats $F\subset X$, and the chambers are exactly the Weyl chambers at $X(\infty)$. Moreover, the Tits metric completely determines the building structure, and vice versa, and $\big(X(\infty),d_T\big)$ is a metric realization of the Tits building at $X(\infty)$.
\end{prop}

For a non-uniform lattice of $\G\leq G$ the \emph{rational Tits building} $\wqg$ is a building structure on the subset of parabolic points. It is not in general a sub-building of the real spherical building. Recall that flats of $X$ correspond to real maximal split tori in $G$. Since $G$ is an algebraic group defined over $\mathbb{Q}$, one can consider the maximal $\mathbb{Q}$-split tori. The \emph{rational flats} of $X$ are then the $G(\mathbb{Q})$-orbits of maximal $\mathbb{Q}$-split tori of $G$, and the \emph{rational boundary} are all points $\xi\in X(\infty)$ such that $\xi\in F_\mathbb{Q}(\infty)$ for some rational flat $F_\mathbb{Q}$ of $X$. One defines \emph{regular rational directions} and  \emph{rational Weyl chambers} in an analogous way to the real case, this time taking only rational flats into account. For further details see  \cite{JiSymmetricSpacesBuildings}, and Section $2$ in \cite{HattoriLimitSet}.

\begin{thm} [Theorem A in \cite{HattoriLimitSet}] \label{qr: thm: HattoriA}
Let $X=G/K$ be a symmetric space of noncompact type and of high rank, and let $\G\leq G$ be an irreducible non-uniform lattice. Then $\wqg$ does not include horospherical limit points. The $\frac{\pi}{2}$-neighbourhood 

$$\mathcal{N}_\fpi2\big(\wqg\big):=\{\xi\in X(\infty)\mid d_T\big(\xi,\wqg\big)< \fpi2 \}$$

does not include conical limit points. 
\end{thm}

In \qr the converse statement also holds:

\begin{thm} [Theorem B in \cite{HattoriLimitSet}] \label{qr: thm: HattoriB}
Let $X=G/K$ be a symmetric space of noncompact type and of high rank, and let $\G\leq G$ be an irreducible non-uniform lattice. Let

$$\mathcal{V}=\{\xi\in X(\infty)\mid d_T\big(\xi,\wqg\big)\geq\frac{\pi}{2}\}$$ 

Suppose that  $\G$ is a \qr lattice. Then $\mathcal{V}$ consists of conical limit points only. 

\end{thm}

In groups of $\rr$ one has the following well known fact: 

\begin{thm}[Proposition $5.4.2$ and Theorem $6.1$ in \cite{bowditch1995geometrical}, see also Theorem $12.29$ in \cite{DrutuKapovichGGT}]\label{qr: sym spc: thm: in rank 1 all horospherical limit points are conical}
Let $X$ be a $\rr$ symmetric space, $\G\leq \mathrm{Isom(X)}$ a lattice. Then every $\xi\in X(\infty)$ is either conical or non-horospherical. 
\end{thm}

\begin{cor} \label{qr: cor: limit points characterization of wqg rational building}
When $\G\leq G$ is a \qr lattice, the following holds: 
\begin{enumerate}
    \item $\wqg=\{\xi\in X(\infty)\mid \mn_{\frac{\pi}{2}}(\xi) \text{ does not contain conical limit points} \}$
    \item Any two points $\xi,\xi'\in \wqg$ are at Tits distance $=\pi$.
\end{enumerate}
\end{cor}

\begin{proof}

When $G$ is of $\rr$ both statements hold trivially due to Theorem~\ref{qr: sym spc: thm: in rank 1 all horospherical limit points are conical}. In high rank, they follow from the following observation: for any point $\xi'\in\wqg$ and any point $\zeta \in X(\infty)$ with $d(\zeta,\xi')=\frac{\pi}{2}$, $\zeta$ is conical. To see this notice that $\zeta$ lies on the boundary of a horosphere based at $\xi'$: take a flat $F\subset X$ with $\xi',\zeta\in F(\infty)$. Any geodesic with limit $\zeta$ is contained in (a Euclidean) horosphere based at $\xi'$. The fact that $\G$ is cocompact on the horospheres based at $\wqg$ implies that $\zeta$ is conical.

The second item of the corollary follows: let $\xi,\xi'\in\wqg$ and $c:[0,\alpha]$ a Tits geodesic joining them. There is a flat $F\subset X$ containing both $\xi,\xi'$ as well as $c\subset F(\infty)$. Every point $\zeta$ that is at Tits distance $\frac{\pi}{2}$ from either $\xi$ or $\xi'$ is conical, and no point inside the $\fpi{2}$ neighbourhood of either $\xi$ or $\xi'$ is conical. In $F(\infty)$ the Tits metric is the same as the Tits metric on the Euclidean space with the same rank. Therefore one may prolong the geodesic $c$ so that $c(0)=\xi,c(\alpha)=\xi'$ and $c(\pi)=\xi''$. If $d_T(\xi,\xi')<\pi$, then there is a point along this prolonged geodesic that is at Tits distance exactly $\fpi{2}$ from $\xi$ (so it is conical by the first paragraph), but at Tits distance strictly less than $\fpi{2}$ from $\xi'$ (so it cannot be conical by Theorem~\ref{qr: thm: HattoriA}). Therefore $d_T(\xi,\xi')=\pi$.

For the first item, one containment is just Hattori's Theorem~\ref{qr: thm: HattoriA}. For the other containment, an argument similar to the one above works. Assume for some $\xi\in X(\infty)$ that $\mn_{\fpi{2}}(\xi)$ consists of non-conical limit points. In particular $\xi$ itself is not conical, and by Theorem~\ref{qr: thm: HattoriB} it holds that  $d(\xi,\wqg)<\frac{\pi}{2}$. Let $\xi'\in\wqg$ be a point realizing this distance. As above, this gives rise to a flat $F$ containing $\xi,\xi'$ and another point $\zeta$ that is at Tits distance $\fpi{2}$ from $\xi'$ but at Tits distance strictly less than $\fpi{2}$ from $\xi$. The first forces $\zeta$ to be conical, and the latter forces it to be non-conical, a contradiction. 
\end{proof}

Hattori's characterization relies on a useful lemma, relating the (linear) penetration rate of a geodesic into a horoball to the Tits distance. 

\begin{lem}[Lemma 3.4 in \cite{HattoriLimitSet}] \label{qr: prelims: sym spc: lem: Hattori penetration}

Let $X$ be a symmetric space of noncompact type and of high rank. Let $\eta_1,\eta_2: [0,\infty)\rightarrow X$ be two geodesic rays, $\alpha:=d_T\big(\eta_1(\infty),\eta_2(\infty)\big)$ and $b_2$ the Busemann function corresponding to $\eta_2$.  Then there exists a positive constant $C_1$, depending only on $\eta_1$ and $\eta_2$, such that:
\begin{enumerate}
    \item If $\alpha>\fpi2$, then for all $t\geq 0$
    $$b_2\big(\eta_1(t)\big)\geq-t\cdot\cos{\alpha}-C_1$$ 
    
    \item If $\alpha=\fpi2$, then $b_2\big(\eta_1(t)\big)$ is monotone non-increasing in $t$ and $-C_1\leq b_2\big(\eta_1(t)\big)$.

    \item If $\alpha<\fpi2$, then for all $t\geq 0$
    $$b_2\big(\eta_1(t)\big)\leq-t\cdot\cos{\alpha}-C_1$$ 
\end{enumerate}

\end{lem}

\begin{rmk}\label{qr: prelims: sym spc: rmk: Hattori penetration also good for rank 1}
Lemma~\ref{qr: prelims: sym spc: lem: Hattori penetration} holds also for $X$ of 
$\rr$, where all $\xi,\zeta\in X(\infty)$ admit $d_T(\xi,\zeta)=\pi$.
\end{rmk}

\section{SBE-Rigidity and the Embedding into \texorpdfstring{$G$}{}} \label{sec: SBE chapter}

Following Cornulier \cite{cornulier2017sublinear}, Pallier \cite{Pallier_2020} makes the following definition, which is the standard assumption when working with SBE: 

\begin{defn} \label{def: admissible}
A function $u:\mrnn\rightarrow \mrpp$ is \emph{admissible} if it satisfies the following conditions:
\begin{itemize}
    \item $u$ is non-decreasing.
    \item $u$ grows sublinearly: $\limsup\limits_{r\mapsto\infty} \frac{u(r)}{r} =0$.
    \item $u$ is doubling: $\limsup_{r\rightarrow \infty}\frac{u(2r)}{u(r)}<\infty$.
\end{itemize}
\end{defn}

The focus in this paper is on the sublinearity of $u$. If $u$ is subadditive and sublinear, there is $v\in O(u)$ which is subadditive and non-decreasing, which in particular implies it is doubling. For the rest of this section, by an SBE I mean an $(L,u)$-SBE where $u$ is sublinear, subadditive and, without loss of generality, admissible. 
I remark that throughout this paper,  subadditivity (and admissibility) is only used in one place, namely in Lemma~\ref{sbe: lem: uniform control on the SBE constants for all Lambda}. Most results, including most results of this section, do not require this mild assumption. 

In analogy with quasi-isometries, $L$ and $u$ are called the \emph{SBE constants}. An $(L,u)$-SBE $f$ admits an SBE-inverse, defined in a similar fashion to quasi-inverses of quasi-isometries (see e.g.~\cite{DrutuKapovich}). If $f:X\rightarrow Y$ is an $(L,u)$-SBE, then the SBE-inverse $g:Y\rightarrow X$ is also an $(L,u)$-SBE with the property that for all $x\in X,y\in Y$ it holds $d\big(gf(x),x\big)\leq u(x)$ and $d\big(fg(y),y\big)\leq u(y)$.

The proofs I present in this section are quite indifferent to whether the lattice $\G$ is uniform or not. In order to have a unified framework and to ease notation, I fix the convention that for both uniform and non-uniform lattices, $X_0$ denotes the compact core of the lattice. This just means that $X=X_0$ in case $\G$ is uniform.

\begin{rmk}
    This section heavily relies on Druţu~\cite{DrutuQI}, and while the novelty in her work concerns non-uniform lattices, one readily sees that her arguments hold for uniform lattices as well (in that case they are reduced to the work of Kleiner-Leeb~\cite{KleinerLeeb}).
\end{rmk}

Kleiner-Leeb and Druţu's proofs of Theorem~\ref{thm: sbe close to isometry Intro} for quasi-isometries are susceptible to an SBE generalization since they are performed in the asymptotic cone of $X$. The SBE adaptation results in Theorem~\ref{thm: sbe close to isometry Intro}, namely that every self-SBE $f:X_0\rightarrow X_0$ is close to an isometry $g:X\rightarrow X$. But unlike the quasi-isometry case, `close' now is only \emph{sublinearly close} (on $X_0$):

\begin{defn}
Let $(X,x_0)$ be a pointed metric space. Two maps $f,g:X\rightarrow X$ are said to be
\emph{sublinearly close} on $X$ if there is a sublinear function $u$ such that $d\big(f(x),g(x)\big)\leq u(|x|)$ for all $x\in X$.
\end{defn}

One would have liked to continue as in the quasi-isometry case: every $\lambda\in \Lambda$ defines a self-SBE $f_\lambda$ of $X_0$ by a conjugation trick. Each of these is sublinearly close to some $g_\lambda \in\operatorname{Isom}(X)$, and one may hope that the map $\lambda\mapsto g_\lambda$ has the desired properties depicted in Theorem~\ref{intro: thm: from SBE to sublinearly close}. There is however one additional obstacle that is unique to the SBE setting, namely that the SBE constants of $f_\la$ depend on $\la$. The resulting sublinear bound on $d\big(f_\la(x),g_\la(x)\big)$ is therefore not strong enough to prove that $\lambda\mapsto g_\lambda$ is even a homomorphism. One needs to gain  uniform control on the SBE constants of the $f_\la$, which is the essence of the following lemma.

\begin{lem}\label{sbe: lem: uniform control on the sublinear constants for fixed x}
Let $\{f_r\}_{r\in \mrp}$ be a family of $(L',v_r)$-SBE $f_r:X_0\rightarrow X_0$, where $v_r=L'\cdot v+v(r)$ for some sublinear function $v\in O(u)$ and a constant $L'$. Let $g_r$ be the associated isometry given by Theorem~\ref{sbe: thm: self sbe is close to isometry}. Then for any $x\in X_0$, there is a sublinear function $u^x\in O(u)$ such that $d\big(f_r(x),g_r(x)\big)\leq u^x(r)$.
\end{lem}

With this type of control one can prove Theorem~\ref{intro: thm: from SBE to sublinearly close}.

\begin{rmk}
Combined with Theorem~\ref{sbe: thm: self sbe is close to isometry}, a different way to phrase the above statement is to say that the function $D:\La\times X_0\rightarrow \mrnn$ defined by $D(\la,x)=d\big(f_\la(x),g_\la (x)\big)$ is sublinearly bounded in each variable.
\end{rmk}

\begin{rmk}
Uniform control is often needed when working with SBE, see e.g.\ Section I.3 in~\cite{pallierTheses2019}.   
\end{rmk}

\subsection{SBE are Close to Isometries}\label{sec: an sbe is close to an isometry}
 
In this Section I prove Theorem~\ref{thm: sbe close to isometry Intro} and its uniform version Lemma~\ref{sbe: lem: uniform control on the sublinear constants for fixed x}. 
While the adaptation of Druţu's arguments in \cite{DrutuQI} to the SBE setting is somewhat delicate, giving a complete detailed proof would require reproducing the argument in full. Instead, I only indicate at the required adaptations, which should suffice a reader who is familiar with Kleiner-Leeb and Druţu's arguments to produce a complete proof. I present the routines of the modification and exemplify them in a specific representative example, followed by a road map for proving Theorem~\ref{sbe: thm: self sbe is close to isometry} and Lemma~\ref{sbe: lem: uniform control on the sublinear constants for fixed x}. I use Druţu's terminology freely, without giving the precise definitions and statements that  go into her proof.


\subsubsection{Asymptotic Cones}

\begin{defn}\label{def: asymptotic cones}
Let $(X,d)$ be a metric space. Fix an ultrafilter $\omega$, a sequence of points $x_n\in X$ and a sequence of scaling factors $\imath_n\xrightarrow[{\omega}]{} \infty$. The \emph{asymptotic cone of $X$ w.r.t. $x_n,\imath_n$}, denoted $C(X)$, is the metric $\omega$-ultralimit of the sequence of pointed metric spaces $(X,\frac{1}{\imath_n} d,x_n)$. The metric on $C(X)$ is denoted $d_\omega$.
\end{defn}

The strength of SBE is that they induce biLipschitz homeomorphisms between the respective asymptotic cones. 

\begin{lem}[\cite{Pallier2022LogStructures}, Proposition 6]\label{SBE: lem: SBE induce nilipschitz cone maps}
Let $f:X\rightarrow Y$ be an $(L,u)$-SBE. Then for any choice of scaling factors $\imath_n$ and base points $x_n$ such that $\lim_\omega \frac{u(x_n)}{\imath_n}=0$, $f$ induces an $L$-biLipschitz map $C(f): C(X)\rightarrow C(Y)$ between the corresponding asymptotic cones $C(X)=(X,\frac{1}{\imath_n}d_X,x^0_n
)$ and $C(Y)=(Y,\frac{1}{\imath_n}d_Y,y^0_n)$.
\end{lem}

\begin{rmk}
    The version of Lemma~\ref{SBE: lem: SBE induce nilipschitz cone maps} for asymptotic cones with fixed based points is well known. I will need to take cones with moving base points, i.e.\ it might be that $x^0_n \ne x^0_m$. The reader may check that the equality $\lim_\omega \frac{u(x_n)}{\imath_n}=0$ holds in every cone that appears in the proof. 
\end{rmk}

\subsubsection{The Argument}

\paragraph{\textbf{A High-Level Description.}}
There are two gaps to fill when elevating an SBE $f_0:X_0\rightarrow X_0$ to an isometry $g_0\in G=\mathrm{Isom}(X)$. First, $\G$ is non-uniform and so $f_0$ is not even defined on the whole space $X$. And obviously, $f_0$ is just an SBE. 

Assume first that $\Gamma$ is uniform and that $f=f_0$ is defined on the whole space $X=X_0$. The induced map $C(f):C(X)\rightarrow C(X)$ is biLipschitz, and the work of Kleiner and Leeb \cite{KleinerLeeb} allows one to conclude that $C(f)$ is, up to a scalar, an isometry. In turn, this isometry induces an isometry $\partial g$ on the spherical building structure of $X(\infty)$, and the extension theorem of Tits~\cite{tits1974buildings} associates to $\partial g$ a unique isometry $g\in \mathrm{Isom(X)}$ that has $\partial g$ as its boundary map. By construction, it is not difficult to see that $g$ and $f$ are `close'. 

In case $\G$ is non-uniform, an SBE $f:\G\rightarrow \G$ does not readily yield a cone map on $C(X)$, but only on $C(X_0)$. Overcoming this difficulty requires substantial work and is the heart of Druţu's proof. In short, she uses dynamical results stating that the vast majority of flats in $X$ are close enough to $X_0$. As mentioned in Section~\ref{sec: qr: sym spc}, the building structure on $X(\infty)$ is determined by the boundaries of flats $F\subset X$. To some extent this is true also for the (Euclidean) building structure of $C(X)$. Therefore the fact that the majority of flats in $X$ are `close enough' to $X_0$ results in the fact that $C(X_0)$ composes a large enough portion of $C(X)$.

\paragraph{\textbf{Flat Rigidity.}}
The adaptations are required mainly to prove the so called \emph{flat rigidity}, namely that $f_0$ maps a flat $F\subset X_0$ close to another flat $F'\subset X_0$. This is proved by passing to the cone map and using an analogous result for biLipschitz images of apartments in Euclidean buildings.

The proof consists of $6$ steps:
\begin{enumerate}
    \item The \emph{horizon} of an image of a Weyl chamber is contained in the \emph{horizon} of a finite union of Weyl chambers, and the number of chambers in this union depends only on the Lipschitz constant. (Lemmata $3.3.5$, $3.3.6$ in \cite{DrutuQI}, see below a sketchy definition of horizon).
    
    \item The horizon of an image of a flat coincides with the horizon of a finite union of Weyl chambers, and the number of chambers in this union depends only on the Lipschitz constant. 
    
    \item \label{item: fan}The union of Weyl chambers in the previous step limits to an apartment in the Tits building at $X(\infty)$. Such a union is called a \emph{fan over an apartment}. 
    
    \item For each Weyl chamber $W$ there corresponds a unique chamber $W'$ such that $q_0(W)$ and $W'$ have the same horizon. This amounts to an induced map on the Weyl chambers of the Tits building at $X(\infty)$ (Lemma $4.2.1$ in \cite{DrutuQI}).
    
    \item \label{item: close flats} Given a flat $F$ through a point $x$, the unique flat $F'$ asymptotic to the union of Weyl chambers obtained in step~\ref{item: fan} is at uniform bounded distance from $f_0(x)$. The bound depends only on the quasi-isometry constants. The flat $F'$ is called the flat \emph{associated} to $F$.
    
    \item If $F_1$ and $F_2$ are two flats through $x$ which intersect along a hyperplane $H$, then the boundaries at $X(\infty)$ of the associated flats $F_1'$ and $F_2'$ intersect along a hyperplane of the same codimension as $H$.  
\end{enumerate}

For a precise definition of horizon see \cite{DrutuQI}, section 3. For now, it suffices to say the following. The horizon of a set $A\subset X$ is contained in the horizon of a set $B\subset X$ if, looking far away at $A$ from some point $x\in X$, $A$ appears to be contained in an $\varepsilon$-neighbourhood of $B$. This intuition is made precise by considering the angle at $x$ that a point $a\in A$ makes with the set $B$. Two sets have the same horizon if each set's horizon is contained in the other. In the case $A$ and $B$ have the same horizon, an important aspect is the distance $R$ starting from which $A$ and $B$ seem to be $\varepsilon$-contained in one another. Call this distance the \emph{horizon radius}. It depends on $x$ and $\varepsilon$.

The proofs for most of these steps have similar flavour: in any asymptotic cone $C(X)$, $f_0$ induces a biLipschitz map. Kleiner and Leeb \cite{KleinerLeeb} developed a rich theory about such maps between cones of high rank symmetric spaces. One assumes towards contradiction that some assertion fails (say, in step~\ref{item: close flats}, assume that there is no bound on the distance between $f_0(x_n)$ and the associated flat $F'_n$). This gives an unbounded sequence of scalars (say, $d\big(f_0(x_n),F'_n\big)=\imath_n\rightarrow \infty$). These scalars are used to define a cone in which one obtains a contradiction to some fact about biLipschitz maps (say, that the point $[f_0(x_n)]_\omega$ is at $d_\omega$-distance $1$ from $[F'_n]_\omega$, while it should lie in  $[F'_n]_\omega$).

Typically, the bounds obtained this way depend on the quasi-isometry constants. \emph{A priori}, they also depend on the specific point $x\in X$ or the specific flat $F\subset X$ (e.g.\ the horizon radius for the chambers in step \ref{item: fan} or the bound on $d\big(f_0(x),F'\big)$ in step \ref{item: close flats}). However, in the quasi-isometry setting the bounds are actually independent of the choice of point/flat/chamber -  one can pre-compose $f_0$ with an isometry translating any given point/flat/chamber to a fixed point/flat/chamber (resp.), \emph{without changing the quasi-isometry constants} (see e.g.\ Remark 3.3.11 in \cite{DrutuQI}). 

Passing to the SBE setting, the essential difference is exactly that the bounds one obtains depend on the specific point, Weyl chamber or flat. Indeed they should depend on $|x|$, as much as in the quasi-isometry case they depend on the additive constant. It is sensible to guess though that the bounds only grow sublinearly in $|x|$, which is enough in order to push the argument forward. 


\subsubsection{Adapting Cone Arguments}

To adapt for the SBE setting, split each step into three sub-steps: the first two amount to proving Theorem~\ref{sbe: thm: self sbe is close to isometry}, and the third step amounts to proving Lemma~\ref{sbe: lem: uniform control on the sublinear constants for fixed x}.

\paragraph{\textbf{Sub-Step 1.}} Repeat the argument of the quasi-isometry setting \emph{verbatim}, to obtain a bound $c=c(x)$ which depends on the point $x$.

\paragraph{\textbf{Sub-Step 2.}} The next step is to show the constants from step 1 grow like the function $u$. Assume towards contradiction that there is a sequence of points $x_n$ for which $\lim_\omega \frac{u(|x_n|)}{c(x_n)}=0$. One can then use a cone with moving base points  $C(X)=Cone\big(X,x_n,c(x_n)\big)$, and proceed as in step 1 to obtain a contradiction. 

\paragraph{\textbf{Sub-Step 3.}} Fix $x\in X$ and a sequence of SBE as in Lemma~\ref{sbe: lem: uniform control on the sublinear constants for fixed x}, i.e.\ $\{f_n\}_{n\in\mathbb{N}}$ with the same Lipschitz constant and with sublinear constants $v_n(s)=v(s)+u'(n)$, for some sublinear functions $v,u'$. Denote by $c_n(x)$ the constants that were achieved in the previous steps for $x$ and the SBE map $f_n$, and assume towards contradiction that $|c_n(x)|$ is not bounded above by any function in $O(u')$. This implies that, up to a subsequence,  $\lim_\omega \frac{u'(n)}{c_n(x)}=0$. One concludes that the cone map $C\big((f_n)_n\big)$ is a biLipschitz homeomorphism, and gets a contradiction in the same manner as previous steps.

\begin{exa} \label{exmp: distance to associated flat}

I demonstrate this procedure in full for a specific claim.

\begin{prop}[SBE version of Proposition $4.2.7$ in~\cite{DrutuQI}] \label{prop: sublinear bound on distance to associated flat}
Let $f:X\rightarrow X$ be an $(L,u)$-SBE, and $F\subset X$ a flat through $x$ to which $f$ associates a fan over
an apartment, $\cup_{i=0}^pW_i$. If $F'$ is the maximal flat asymptotic to
the fan, then $d\big(f(x),F'\big)\leq c(x)$ where $c(x)=c(|x|)\in O(u)$.

Moreover, let $f_n: X\rightarrow X$ be a sequence of $(L,v_n)$-SBE for $v_n=v+u'(n)$ for $v,u'$ some sublinear functions. The constant $c_n(x)$ associated to $x$ and the SBE $f_n$ achieved in the first part of the proposition admits $c_n(x)\leq u^x(n)$ for some sublinear function $u^x\in O(u')$.
\end{prop}

The flat $F'$ is said to be the flat \emph{associated to $F$ by $f$}.

\begin{proof}

Proceed in three (sub-)steps.

\paragraph{\textbf{Step 1.}}  

Fix $x\in X$ and assume towards contradiction that there exists a sequence $F_n$ of flats through $x$
and a sequence $f_n: X\rightarrow X$ of $(L,u)$-SBE
such that $c_n:= d\big(f_n(x), F_n'\big)\rightarrow \infty$. In $\mathrm{Cone}(X, x, c_n^{-1})$ one can show that $[\cup_{i=0}^pW_i^n]$, the union of Weyl chambers associated to $f_n(F_n)$, is a maximal flat (see Proposition $4.2.6$ in \cite{DrutuQI}). Denote  $F_\omega:=[\cup_{i=0}^pW_i^n]$. Furthermore, since the biLipschitz flat $[f_n(F_n)]\subset \mathrm{Cone}(X,x,c_n^{-1})$ is
contained in it, it coincides with it:  $[f_n(F_n)] = F_\omega$. 
On the other hand, since the Hausdorff distance between $\cup_{i=0}^pW_i^n$  and
$F_n'$ is by assumption $c_n=d(x, F_n')$, in the cone the maximal flats $F_\omega$ and $F_\omega':=[F_n']$ are at Hausdorff distance 1. 

This implies that $F_\omega=F_\omega'$ (see  Corollary $4.6.4$ in \cite{KleinerLeeb}). But since $d\big(f_n(x), F_n'\big)=c_n$ the limit point $y_\omega:=C(f_n)(x_\omega)=[f_n(x)]$ , which is contained in
$F_\omega$, is at distance 1 from $F_\omega'$ - a contradiction.

\paragraph{\textbf{Step 2.}} Assume $c(x)$ is taken to be the smallest possible for each $x$, and then modify the function $c$ so that $c(x)=\sup_{y:|y|=|x|}c(y)$ (one has to get convinced that the original map $x\mapsto c(x)$ is bounded on compact sets, which it is). The function $c:X\rightarrow \mathbb{R}$ now only depends on $|x|$. I wish to show that $c(|x|)=O\big(u(x)\big)$. Assume towards contradiction that this is not the case, and obtain a sequence $x_n$ with $|x_n|\rightarrow \infty$ such that $\lim_\omega \frac{c(x_n)}{u(|x_n|)}=\infty$. Denote $c_n=c(x_n)$ and consider $\mathrm{Cone}(X,x_n,c_n^{-1})$. By the definition of $c(x)$ there is a sequence of flats $F_n$ through $x_n$ and SBE $f_n$ (all $f_n$ have the same constants) such that $c_n=d\big(f_n(x_n),F_n'\big)$. A computation shows that the assumption $\lim_\omega\frac{u(|x_n|)}{c_n}=0$ implies $C\big((f_n)_n\big)$ is a biLipschitz homeomorphism, and one may proceed as in step $1$. The flat $F_\omega':=[F_n']$ is at distance $1$ from the point $[f_n(x_n)]$, which lies on the maximal flat $[\cup_{i=0}^pW_i^n]$. The latter flat is, on the one hand, at Hausdorff distance $1$ from $F_\omega'$ (by the definition of the scaling factors $c_n$), so they must coincide. On the other hand, $[\cup_{i=0}^pW_i^n]$  coincides with $F_\omega:=C(f_n)(F_n)=[f_n(F_n)]$, so $F_\omega=F_\omega'$, contradicting the fact that $d_\omega([f_n(x_n)],F_\omega')=1$. Thus $c\in O(u)$, as wanted.

\paragraph{\textbf{Step 3.}}
I now turn to gaining uniform control on the growth of $c(x)$ as a function of the sublinear constants. Consider a sequence $f_n$ as in the statement. Fix $x\in X$, and denote by $c_n=c_n(x)$ the constant obtained in step $1$ with respect to the SBE constants $(L,v_n)$. Assume towards contradiction that, up to passing to a subsequence, $\lim_\omega\frac{u'(n)}{c_n}=0$. The proof goes exactly as in step $1$, with the sole difference that now, as in step 2, one might need convincing in the fact that in $C(X)=(X,x,c_n^{-1})$, the cone map $C(f_n):=C\big((f_n)_n\big)$ is a biLipschitz homeomorphism. But indeed for any two cone points $(x_n),(y_n)\in C(X)$ it holds: 

$$d_\omega\big(C(f_n)(x_n),C(f_n)(y_n)\big)=\lim_\omega \frac{1}{c_n}d\big(f_n(x_n),f_n(y_n)\big)\leq \lim_\omega \frac{1}{c_n}L d(x_n,y_n)+v_n(|x_n|\vee|y_n|)$$

By definition $\lim_\omega \frac{1}{c_n}Ld(x_n,y_n)=Ld_\omega \big((x_n),(y_n)\big)$, and it remains to show $\lim_\omega \frac{1}{c_n}v_n(|x_n|\vee|y_n|)=0$. By definition of $v_n$ it amounts to proving $\lim_\omega \frac{1}{c_n}v(|x_n|)=0=\lim_\omega \frac{1}{c_n}v(|y_n|)$ and $\lim_\omega \frac{u'(n)}{c_n}=0$. The former follows from the fact that $(x_n),(y_n)\in C(X)$ and therefore  both $\lim_\omega\frac{1}{c_n}|x_n|$ and $\lim_\omega\frac{1}{c_n}|y_n|$ are finite. The latter follows from the assumption on the $c_n$. In the same manner one can verify $C(f_n)$ is surjective. One obtains a contradiction identical to the one in Step $1$.

\end{proof}

\end{exa}
\subsubsection{Proof of Theorem~\ref{sbe: thm: self sbe is close to isometry} and Lemma~\ref{sbe: lem: uniform control on the sublinear constants for fixed x}}
\begin{proof}[Proof of Theorem~\ref{sbe: thm: self sbe is close to isometry} and Lemma~\ref{sbe: lem: uniform control on the sublinear constants for fixed x}]
Following Sections 3,4 and 5 in \cite{DrutuQI} carefully, and making the SBE adaptations as in Example~\ref{exmp: distance to associated flat}, one obtains uniform flat rigidity for SBE. This means that every SBE flat in $X_0$ is sublinearly close to a flat, where the specific bounds depend sublinearly in the location of the flats. The complete list of relevant claims from~\cite{DrutuQI} is:

\begin{itemize}
    \item \textbf{Section 3.} All claims starting from Lemma $3.3.5$ through Corollary $3.3.10$. All statements should consider, instead of a quasi-isometry, a general $(L,u)$-SBE $f$ and, when relevant, a general point $x\in X$ with $f(x)=y$ (i.e., $f(x)$ does not necessarily equal $x$). Also when relevant one should consider a family $f_n$ of $(L,v_n)$-SBE as depicted in Lemma~\ref{sbe: lem: uniform control on the sublinear constants for fixed x} above.
    
    \item \textbf{Section 4.} Propositions $4.2.6, 4.2.7, 4.2.9$. When relevant, statements should be modified so that the distance between $f(x)$ and an associated flat of it should be a sublinear function which depends only on $|x|$ and the SBE constants. Also when relevant one should consider a family $f_n$ of $(L,v_n)$-SBE as above.
    
    \item \textbf{Section 5.} Lemma $5.4.1$, where the constant $D=D(x)$ should be uniformly linear (not sublinear) in $c=c(x)$. Proposition $5.4.2$, where the constant $D=D(x)$ should be replaced by a sublinear function $D(|x|)$. When relevant, one should consider a family $f_n$ of $(L,v_n)$-SBE as above.
    
\end{itemize}

This yields a proof for uniform flat rigidity in the SBE setting. 

The final and key part in Druţu's argument, indeed the one distinguishing it from Kleiner-Leeb~\cite{KleinerLeeb}, concerns the fact that $f_0$  may be defined only on $X_0$ and not on the entire $X$. The considerations for this issue are intertwined in her proof, and only involve dynamical and geometric facts about the lattice $\G$. Concretely, it is shown that a large enough proportion of flats in $X$ are in some appropriate sense close to $X_0$. Since this aspect has nothing to do with the self quasi-isometry, it does not require any adaptation to include the setting of $f_0$ being SBE. 

\end{proof}

\subsubsection{Some Remarks On \texorpdfstring{$\rr$}{} Factors}\label{sec: SBE for rank 1 factors}
Quasi-isometric rigidity holds for groups of $\rr$. It is worth mentioning that Schwartz's proof also relies on `flat rigidity' - but in this case the flats are the horospheres of $\G$. While these are not isometrically embedded flats, the induced metric on horospheres is flat, and flat-rigidity allows Schwartz to construct the boundary map and find the associated isometry. 

A key ingredient for quasi-isometric rigidity in $\rr$ is the quasi-conformality of boundary maps of quasi-isometries. In the sublinear bilipschitz equivalence case the situation is more delicate: Pallier defined the notion of \emph{sublinearly quasi-conformal} maps, such that a boundary map of a sublinear bilipschitz equivalence is sublinearly quasi-conformal. However, he found examples of sublinearly quasi-conformal maps that are not quasi-conformal, and therefore do not arise as boundaries of quasi-isometries. It is not known whether they can arise as boundaries of sublinear bilipschitz equivalence (see Appendix A in~\cite{PallierQuasiconformality}). Moreover, while these exotic examples may be differentiable, the differential may not have the required properties to run Schwartz's argument through.

Flat rigidity, in the meaning of horosphere rigidity in symmetric spaces of $\rr$, also holds for SBE. Considering the compact core $X_0\subset X$ of $\G$, one can use Proposition $5.6$ in Druţu-Sapir~\cite{DrutuSAPIR} in order to show that horospheres are mapped sublinearly close to horospheres.  Therefore self-SBE of compact cores are very special, which gives hope that they might retain the required analytic properties of quasi-isometries.

However, as in the case of flat rigidity in high rank, the distance between an SBE horosphere and the associated horosphere is not uniformly bounded but rather grows sublinearly with the distance of the horosphere to the base point. the lack of a uniform bound seems like a real obstruction to generalizing Schwartz's argument for \qr lattices in Hilbert modular groups~\cite{SchwartzDiophantine}.

\subsection{From SBE to Sublinearly Close Subgroups}\label{sec: from SBE to sublinearly close}
In this section I prove Theorem~\ref{intro: thm: from SBE to sublinearly close}. Apart from some calculations of sublinear nature, the proof follows the classical argument that already appeared in Schwartz~\cite{Schwartz} (see Section $10.4$ therein). I need the well known fact that sublinearly close isometries are equal: 

\begin{lem} \label{sbe: lem: sublinear distance isometries}
Let $X$ be a symmetric space of noncompact type. Let $\G\leq \mathrm{Isom}(X)$ be a non-uniform irreducible lattice, $X_0$ a compact core of $\G$. Let $g,h\in G=\mathrm{Isom}(X)$ and $u$ a sublinear function such that for every $x\in X_0$, $d\big(g(x),h(x)\big)\leq u(|x|)$. Then $g=h$.
\end{lem}

\begin{proof}

One may assume $h=id_X$. For any $\G$-periodic geodesic ray $\eta:[x_0,\xi)$ there is, by definition, some $T>0$ and a sequence $\g_n\in G$ for which $\eta(nT)=\g_n x_0$. In particular $x_n:=\g_n x_0\in X_0$ and so $d\big(x_n,g(x_n)\big)\leq u(|x_n|)=u(nT)$. On the other hand, the distance function $t\mapsto d\big(\eta(t),g\cdot \eta(t)\big)$ is convex. A convex sublinear function is bounded, and one concludes that for every $\G$-periodic geodesic rays $\eta$, $[\eta]=[g\cdot \eta]$ in $X(\infty)$. As a homeomorphism, the density of $\G$-periodic points in $X(\infty)$ imply $\partial g=id_{X(\infty)}$. Therefore $g=id_X$ (\cite{eberlein1996geometry}, Section $3.10$).
\end{proof}

\begin{rmk}
The proof of the implication  $\partial g=id_{X(\infty)}\Rightarrow g=id_X$ appears in~\cite{eberlein1996geometry} as part of the proof of the extension theorem of Tits. While the latter is restricted to spaces with no $\mathbb{R}$-rank $1$ factors, the implication $\partial g=id_{X(\infty)}\Rightarrow g=id_X$ only uses the fact that $X$ has no Euclidean factors (see pg. $251$ in  \cite{eberlein1996geometry}).

\end{rmk}

\paragraph{\textbf{The Map $\Phi:\La\rightarrow G$.}}

The orbit map $q_0:\G\rightarrow X_0$ defined by $\g\mapsto \g x_0$ is a quasi-isometric embedding: this is \u{S}varc-Milnor in case $\G$ is uniform and $X_0=X$, and  Lubotzki-Mozes-Raghunathan (Theorem~\ref{prelims: thm: lubozki mozes raghunathan LMS QIE} above) if $\G$ is non-uniform. An SBE $f:\La\rightarrow \G$ thus gives rise to an SBE $\La\rightarrow X_0$, which I also denote by $f$.

For each $\la\in \La$ let $f_\lambda:=f\circ L_\la\circ f^{-1}:X_0\rightarrow X_0$, where $L_\la$ is the left multiplication by $\lambda$. The left translation $L_\la$ is an isometry, hence $f_\lambda$ is a self-SBE of $X_0$. By Theorem~\ref{sbe: thm: self sbe is close to isometry}, there exists a unique isometry $g_\la\in \mathrm{Isom}(X)$ that is sublinearly close to $f_\la$. Define the map $\Phi:\Lambda\rightarrow G$ by $\la\mapsto g_\la$. I begin by controlling the SBE constants of the $f_\lambda$. 

\begin{lem}\label{sbe: lem: uniform control on the SBE constants for all Lambda}
For each $\la\in \La$, $f_\la$ is an $(L^2,v_\la)$-SBE, for $$v_\la(|x|):=(L+1)u(|x|)+u(|\la|)$$

In particular $v_\la\in O(u)$. 
\end{lem}

\begin{proof}
The proof is a straightforward computation. Up to an additive constant I may assume $f^{-1}$ is an $(L,u)$-SBE with $f^{-1}(e_\G)=e_\La$. Let $x_1,x_2\in X_0$, and assume w.l.o.g $|x_2|\leq|x_1|$. By the properties of an SBE, this also means that for $i\in \{1,2\}$:

\begin{equation}\label{eq: f(x2) in terms of x1}
|f^{-1}(x_i)|\leq L|x_i-x_0|+u(|x_i|)\leq L|x_1|+u(|x_1|)  
\end{equation}

Notice that $f_\la(x)=f\big(\la\cdot f^{-1}(x)\big)$, and $f$ is an $(L,u)$-SBE. The following inequalities, justified below, give the required upper bounds:

\begin{equation}\label{eq: q_lambda constants}
\begin{aligned}
\big|f_\lambda (x_1)-f_\lambda (x_2)\big| & \leq L\big|\lambda\cdot f^{-1}(x_1)-\lambda\cdot f^{-1}(x_2)\big|+u\big(|\lambda \cdot f^{-1}(x_1)|\vee|\lambda \cdot f^{-1}(x_2)|\big)  \\
  & \leq L^2|x_1-x_2|+Lu\big(|x_1|\big) + u\Big(|\lambda|+L|x_1|+u(|x_1|)\Big) \\
  & \leq L^2|x_1-x_2|+(L+1)u\big(|x_1|\big)+u(|\lambda|).
\end{aligned}
\end{equation}

From the first line to the second line I used:
\begin{enumerate}
    \item For the first term: left multiplication in $\La$ is an isometry, and $f^{-1}$ is an $(L,u)$-SBE.
    \item For the second term: triangle inequality, left multiplication in $\La$ is an isometry, and Inequality~\ref{eq: f(x2) in terms of x1}.

\end{enumerate}  

From the second to the third line I used the properties of $u$ as a subadditive and admissible function, precisely that it is subadditive and doubling, so $u\big((L+1)|x_1|\big)\leq (L+1)u(|x_1|)$ for all large enough $x_1$. 

\end{proof}

\begin{rmk}
The proof of Lemma~\ref{sbe: lem: uniform control on the SBE constants for all Lambda} is the only place where I use the subadditivity of $u$.
\end{rmk}

Combining Lemma~\ref{sbe: lem: uniform control on the SBE constants for all Lambda} with Lemma~\ref{sbe: lem: uniform control on the sublinear constants for fixed x}, one has:

\begin{cor} \label{sbe: cor: sublinear bound on f la for a fixed x}
For any $x\in X_0$ there is a sublinear function $u^x\in O(u)$ such that $$d\big(f_\la(x),g_\la(x)\big)\leq u^x(|\la|)$$.
\end{cor}

\begin{claim}\label{sbe: claim: Phi is homomorphism}
$\Phi:\La\rightarrow G$ is a group homomorphism.
\end{claim}

\begin{proof}
Let $\la_1,\la_2\in \La$. Denote:  

\begin{enumerate}
    \item  $f_1=f_{\la_1},f_2=f_{\la_2},f_{12}=f_{\la_1\la_2}$. By Lemma~\ref{sbe: lem: uniform control on the SBE constants for all Lambda}, these are all $O(u)$ SBE with the same Lipschitz constant $L':=L^2$ and sublinear constants $v_1,v_2,v_{12}\in O(u)$.
    
    \item $g_1=\Phi(\la_1),g_2=\Phi(\la_2),g_{12}=\Phi(\la_1\la_2)$.
    
    \item $u_1,u_2,u_{12}$ the sublinear functions that bound the respective distances between any $g$ and $f$, e.g.\ $|g_1(x)-f_1(x)|\leq u_1(|x|)$, given by Theorem~\ref{sbe: thm: self sbe is close to isometry}.
\end{enumerate}  

One has to prove that $g_{\la_2}\circ g_{\la_1}=g_{\la_1\la_2}$. In light of Lemma~\ref{sbe: lem: sublinear distance isometries}, it is enough to find a sublinear function $v$ such that for all $x\in X_0$ $|g_1g_2(x)-g_{12}(x)|\leq v(|x|)$. By triangle inequality and the above definitions and notations, it is enough to show that each of the following terms is bounded by a function sublinear in $x$: 
\begin{enumerate}
    \item $|g_1g_2(x)-g_1f_2(x)|=|g_2(x)-f_2(x)|\leq u_2(|x|)$ ($g_1$ is an isometry).
    \item $|g_1f_2(x)-f_1f_2(x)|\leq u_1\big(|f_2(x)|\big)\leq u_1\big(L^2|x|+v_2(|x|)\big)$.
    \item $|f_1f_2(x)-f_{12}(x)|$.
    \item $|f_{12}(x)-g_{12}(x)|\leq u_{12}(|x|)$.
\end{enumerate}

Clearly items $1,2,4$ are bounded by a sublinear function in $|x|$. It remains to bound $|f_1f_2(x)-f_{12}(x)|$. The map $\La\rightarrow Aut(\La$) given by $\la\mapsto L_\la$ is a group homomorphism, i.e.\ $L_{\la_1\la_2}=L_{\la_1}L_{\la_2}$, so it remains to bound: 

$$|f_1f_2(x)-f_{12}(x)|=|f L_{\la_1} f^{-1}f L_{\la_2} f^{-1} (x)-f L_{\la_1}L_{\la_2} f^{-1}(x)|.$$

As a composition of an isometry with an SBE, $fL_\la$ , it is still an SBE. Denote the SBE constants of $fL_{\la_1}$ by $L',v$ (clearly one can take $L'=L$ and $v\in O(u)$, but this is not needed). Writing $y:=L_{\la_2}f^{-1}(x)$, this shows

$$|f L_{\la_1} f^{-1}f (y)-f L_{\la_1}(y)|\leq L|f^{-1}fy-y|+v(|f^{-1}fy|\vee|y|).$$

By definition of an SBE-inverse it holds that  $|f^{-1}f(y)-y|\leq u(|y|)$ and in particular also $|f^{-1}f(y)|\leq |y|+u(|y|)$. I conclude that 

$$|f_1f_2(x)-f_{12}(x)|\leq L u(|y|)+v\big(|y|+u(|y|)\big).$$

The right-hand side is a sublinear function in $|y|$, hence it only remains to show that  $|y|$ is bounded by a linear function in $x$. Indeed 

$$|y|=|L_{\la_2}f^{-1}(x)|\leq |\la_2|+|f^{-1}(x)|\leq |\la_2|+L|x|+u(|x|).$$

This completes the proof, rendering $\Phi$ a group homomorphism. 
\end{proof}

\begin{claim}\label{sbe: claim: Phi has discrete image and finite kernel}
$\Phi$ has discrete image and finite kernel.
\end{claim}

\begin{proof}
I show that for any radius $R>0$, there are finitely many $\la\in \La$ for which $g_\la x_0\in B(x_0,R)$. I.e., that there is a finite number of $\Phi(\La)$-orbit points, with multiplicities, inside an $R$ ball in $X$. In particular the set $\{\la\in \La\mid g_\la x_0=x_0\}$ is finite, and clearly contains $\mathrm{Ker}(\Phi)$. In addition, the actual number of $\Phi(\La)$-orbit points inside that $R$ ball is finite, and one may conclude $\Phi(\La)$ is discrete.

Let $R>0$, and $\la\in \La$. By the defining property of $g_\la$ and the definition of $f_\la$, reverse triangle inequality gives  

$$d\big(x_0,g_\la (x_0)\big)\geq  d\big(x_0,f_\la (x_0)\big)-d\big(g_\la(x_0),f_\la(x_0)\big)\geq |f(\la)|-  u_\la(|x_0|).$$

Corollary~\ref{sbe: cor: sublinear bound on f la for a fixed x} gives $d\big(g_\la(x_0),f_\la(x_0)\big)\leq u^{x_0}(|\la|)$ for some sublinear function $u^{x_0}\in O(u)$. On the other hand $f$ is an SBE, and so $|f(\la)|$ grows close to linearly in $\la$. Formally, 

$$|f(\la)|=d\big(f(\la),x_0\big)=d\big(f(\la),f(e_\La)\big)\geq \frac{1}{L}d(\la,x_0)-u(|\la|\vee|e_\La|)\geq \frac{1}{L}|\la|-u(|\la|).$$

To conclude, one has 
$$d\big(x_0,g_\la (x_0)\big)\geq \frac{1}{L}|\la|-u(|\la|)-u^{x_0}(|\la|)$$

and both $u,u^{x_0}$ are sublinear in $|\la|$. Therefore there is a bound $S\in\mrp$ such that $|\la|>S\Rightarrow \frac{1}{L}|\la|-u(|\la|)-u^{x_0}(|\la|)>R$. The group $\La$ is finitely generated and so only finitely many $\la\in \La$ admit $|\la|\leq S$, hence $g_\la (x_0)\in B(x_0,R)$ only for finitely many $\la\in \La$. 

\end{proof}

\begin{claim} \label{sbe: claim: sublinear bound to Gamma}
There exists a sublinear function $u'\in O(u)$ such that 

$$\G x_0\subset \mn_{u'}\big(\Phi(\La) x_0\big)$$

\end{claim}

\begin{proof}
I claim that there is a sublinear function $u_0$, depending only on $f$ and $q_0$, such that for all $\g\in G$,  $d\big(g_\la(x_0),\g(x_0)\big)\leq u_0(|\g|)$.

The only control on $g_\lambda$ is via $f_\la$, and so I use triangle inequality to get:

$$d\big(g_\la(x_0),\g(x_0)\big)\leq d\big(g_\la(x_0),f_\la (x_0)\big)+d\big(f_\la(x_0),\g(x_0)\big)$$

By Corollary~\ref{sbe: cor: sublinear bound on f la for a fixed x}, $ d\big(g_\la(x_0),f_\la (x_0)\big)\leq u^x_0(|\la|)$ for a sublinear function $u^{x_0}$.

In what follows I distinguish between the SBE $f_\G:\La \rightarrow\G$ and its composition  with the orbit quasi-isometry $f_0=q_0\circ f_\G$. Define $\la_\g:=f_\G^{-1}(\g)$, and recall that without loss of generality $q_0\circ f_\G^{-1}(x_0)=e_\La$. Moreover, $\G$ is assumed to be torsion-free, and so there is no ambiguity or trouble in defining the restriction of the map $q_0^{-1}$ to the orbit $\G x_0$ to be of the form $q_0^{-1}(\g x_0)=\g$. All together, this gives 

$$d\big(f_{\la_\g}(x_0),\g(x_0)\big)=d\big(f_0(\la_\g),\g(x_0)\big)=d\big(q_0\circ f_\G f_\G^{-1}(\g), q_0(\g) \big).$$

Since $f_\G$ is an SBE $d\big(f_\G f_\G^{-1}(\g),\g\big)\leq u(\g)$. The fact that $q_0$ is an $(L',C)$-quasi-isometry implies that 

$$d\big(q_0\circ f_\G f_\G^{-1}(\g), q(\g)\big)\leq Ld\big(f_\G f_\G^{-1}(\g),\g\big)+C\leq L'u(|\g|)+C.$$

Combining everything, one has

$$d\big(g_{\la_\g}(x_0),\g(x_0)\big)\leq u^{x_0}(|\la_\g|)+L'u(|\g|)+C.$$

As before, $|\la_\g|=|f^{-1}(\g)|\leq |\g|+u(|\g|)\leq 2|\g|$, where the last inequality holds for all large enough $\g$.  What matters is that $|\la_\g|$ is linear in $|\g|$. I conclude that indeed $\G x_0\subset \mn_{u'}\big(\Phi(\La) x_0\big)$ for the sublinear function $u'=u^{x_0} +L'u+C\in O(u)$, as wanted. (To be pedantic, $u'=u^{x_0}\circ 2 +L'u+C\in O(u)$ where $2$ is the `multiplication by $2$' function, $r\mapsto 2r$). 

\end{proof}

\begin{claim} \label{sbe: claim: sublinear bound to Lambda}
There exists a sublinear function $u'\in O(u)$ such that 

$$\Phi(\La) x_0\subset \mn_{u'}(\G x_0).$$
\end{claim}

\begin{proof}
Let $\la\in \La$. Let $\g=f(\la)$ and consider the distance $d(g_\la x_0,\g x_0)$. Triangle inequality and Corollary~\ref{sbe: cor: sublinear bound on f la for a fixed x} give

$$d(g_\la x_0,\g x_0)\leq d(g_\la x_0,f_\la x_0)+d(f_\la x_0,\g x_0)\leq u'(|\la|)+d(f_\la x_0,\g x_0).$$

By definition of $f_\la$ it holds that  $f_\la(x_0)=f(\la)\cdot x_0=\g x_0$ and so $d\big(f_\la(x_0),\g x_0\big)=0$.
\end{proof}

\begin{proof}[Proof of Theorem~\ref{intro: thm: from SBE to sublinearly close}]
The result is the combination of Claims~\ref{sbe: claim: Phi is homomorphism}, \ref{sbe: claim: Phi has discrete image and finite kernel}, \ref{sbe: claim: sublinear bound to Gamma} and~\ref{sbe: claim: sublinear bound to Lambda}
\end{proof}

\section{Lattices with Property (T)}\label{sec: T}

Recall that a lattice in a locally compact group $G$ has property (T) if and only if $G$ has property (T). In this section I prove sublinear rigidity for lattices with property (T): 

\begin{thm}\label{thm: main for T}
Let $G$ be a real centre-free semisimple Lie group without compact factors, $\G\leq G$ a lattice, $\La\leq G$ a discrete subgroup such that $\G\subset\mn_u(\La)$ for some sublinear function $u$. If $\G$ has property (T), then $\La$ is a lattice.
\end{thm}

Combined with Theorem~\ref{intro: thm: from SBE to sublinearly close}, this proves Corollary~\ref{thm: SBE main Intro} (Section~\ref{sec: SBE completeness corollary}).

The focus of this paper is sublinear distortion. However, for lattices with property (T) a slightly stronger result, which I call \emph{$\varepsilon$-linear rigidity}, holds.

\begin{defn}\label{T: def: asymptotically smaller function}
Let $f,g:\mrnn\rightarrow\mrp$ be two monotonically increasing functions. Call $f$  \emph{asymptotically smaller} than $g$ if $\limsup{\frac{f}{g}}\leq 1$, and denote $f\preceq_\infty g$.
\end{defn}

\begin{thm} \label{T: thm: main epsilon}
Let $G$ be a real centre-free semisimple Lie group without compact factors, $\G\leq G$ a lattice and $\La\leq G$ a discrete subgroup. If $\G$ has property (T) then there exists $\varepsilon=\varepsilon(G)>0$ depending only on $G$ such that if $\G\subset\mn_u(\La)$ for some function $u(r)\preceq_\infty  \varepsilon r$, then $\La$ is a  lattice. 
\end{thm}

Clearly Theorem~\ref{T: thm: main epsilon} implies Theorem~\ref{thm: main for T}. I thank Emmanuel Breuillard for suggesting this generalization. From now and until the end of this section, the standing assumptions are those of Theorem~\ref{T: thm: main epsilon}.

\paragraph{\textbf{Lattice Criterion.}} 
I use a criterion by Leuzinger, stating that being a lattice is determined by the exponential growth rate.

\begin{defn}

Given a pointed metric space $(X,d_X,x_0)$, denote: 
\begin{enumerate}
    \item $b_X(r)=|B(x_0,r)|$
    \item $b_X^s(r)=\sup_{x\in X}|B(x,r)|$
\end{enumerate}   
\end{defn}
In the special case of a discrete group $\Delta\leq \operatorname{Isom}(X)$, the orbit $\Delta x_0$ with the induced metric from $X$ is a pointed metric space. In this setting $b_{\Delta x_0}(r)=|B_X(x_o,r)\cap\Delta x_0|$ and $b_{\Delta x_0}^s(r)=b_{\Delta x_0}(r)$. In this section the pointed metric spaces of interest are the $\G$ and $\La$ orbits in the symmetric space $X=G/K$, and I simply write $b_\Delta(r)$ for both $b_{\Delta x_0}(r)$ and $b_{\Delta x_0}^s(r)$.

\begin{defn}

Let $X$ be a symmetric space and $\Delta\leq G=\mathrm{Isom}(X)^\circ$ a subgroup of isometries. The \emph{critical exponent} of $\Delta$ is defined to be  

$$\delta(\Delta):=\limsup_{r\rightarrow \infty}\frac{\log\big(b_{\Delta}(r)\big)}{r}$$
\end{defn}

To a semisimple Lie group $G$ one can associate a quantity $\Vert \rho\Vert$, where $\rho=\rho(G)$ is the half sum of positive roots in the root system of $(\mf{g},\mf{a})$ (see Section $2$ in \cite{leuzingerCriticalExponent}).

\begin{thm}[Theorem 2 in \cite{leuzingerCriticalExponent}]\label{T: thm: Leuzinger criterion with epsilon bound}
Let $G$ be a real centre-free semisimple Lie group without compact factors. Let $\Delta$ be a discrete, torsion-free subgroup of $G$ that is not a lattice. If G has Kazhdan's property $(T)$, then there is a constant $c^*(G)$ (depending on $G$ but not on $\Delta$) such that $\delta(\Delta)\leq 2\Vert\rho\Vert-c^*(G)$.
\end{thm}

It is known that the critical exponent of a discrete subgroup $\Delta\leq G$ is bounded above by $2\Vert\rho\Vert$ (see Section $2.2$ in \cite{leuzingerCriticalExponent}). Moreover, every lattice $\G\leq G$ admits $\delta(\G)=2\Vert\rho\Vert$ (Example $2.3.5$ in \cite{leuzingerCriticalExponent}, Theorem C in \cite{Albuquerque1999}). Combining these facts with Theorem~\ref{T: thm: Leuzinger criterion with epsilon bound} gives: 

\begin{thm}[Theorem B in \cite{leuzingerCriticalExponent}]\label{T: thm: Leuzinger criterion}

Let $G$ be a real centre-free semisimple Lie group without compact factors. Let $\Delta$ be a discrete, torsion-free subgroup of $G$. If $G$ has Kazhdan's property $(T)$, then $\Delta$ is a lattice iff $\delta(\Delta)= 2\Vert\rho\Vert$.
\end{thm}

\paragraph{\textbf{Argument and the Use of $\varepsilon$-Linearity.}} 

The effect of $\varepsilon$-linear distortions on the exponential growth rate is small enough so that for small $\varepsilon$, the growth rate of $\La$ will stay above the lattice threshold given by Leuzinger's Theorem~\ref{T: thm: Leuzinger criterion with epsilon bound}. The effect of sublinear distortions on the growth rate was manifested in a proposition by Cornulier~\cite{cornulier2017sublinear}, stated here in Proposition~\ref{T: prop: Cornulier growth discrepency}. I tighten its formulation to get a result for $\varepsilon$-linear distortions in Proposition~\ref{T: prop: Upgrade Cornulier growth discrepency}. This is the only use I make of $\varepsilon$-linearity, and the computations involved are straightforward.

\subsection{Proof of Theorem~\ref{T: thm: main epsilon}}

Relating $\G x_0$ and $\La x_0$ is done via the \emph{closest point projection}.

\begin{defn}\label{def: closest point projection}
Let $X$ be a metric space, $Y,Z\subset X$ two closed subsets of $X$. The \emph{closest point projection} of $Z$ on $Y$ is the set-theoretic map $p_Y:Z\rightarrow Y$ defined by $p_Y(z):=y_z$, where $y_z\in Z$ is any point $y\in Y$ realizing the distance $d(z,Y)=d(z,y)$.  If $X$ is a proper metric space and $Y$ discrete, then there are at most finitely many such points. In any case of multiple points, $p_Y$ chooses one arbitrarily.
\end{defn}

Again the case of interest is where $X=G/K$, $x_0=eK$, the two subsets are the orbits $\G x_0$ and $\La x_0$, and the projection is $p_{\La x_0}:\G x_0\rightarrow\La x_0$. Also here there is no risk in denoting this projection by $p_\La$.

Cornulier proves the following growth discrepancy result. 

\begin{prop}[Proposition 3.6 in \cite{cornulier2017sublinear}]\label{T: prop: Cornulier growth discrepency}

Let $X,Y$ be two pointed metric spaces. Let $u$ be a non-decreasing sublinear function and $p:X\rightarrow Y$ a map such that for some $L, R_0>0$:
\begin{enumerate}
    \item $|p(x)|\leq \max(|x|,R_0)$, i.e.\ $p(|x|)\leq |x|$ for all large enough $x\in X$.
    \item $d_Y\big(p(x),p(x')\big)\geq \frac{1}{L}d_X(x,x')-u(\max\{|x|, |x'|\})$
    \end{enumerate}
    
Then for all $r>R_0$, $b_Y(r)\geq b_X(r)/b_X^s\big(L u(r)\big)$. 
\end{prop}

I need a slightly modified version:

\begin{prop}\label{T: prop: Upgrade Cornulier growth discrepency}

Let $X,Y$ be two pointed metric spaces. Let $u$ be a non-decreasing function that admits $u(r)\preceq_\infty \varepsilon r$ for some $\varepsilon<1$, and $p:X\rightarrow Y$ a map such that for some $L, R_0>0$:
\begin{enumerate}
    \item $|p(x)|\leq \max(|x|+u(|x|),R_0)$, i.e.\ $|p(x)|\leq|x|+u(|x|)$ for all large enough $x\in X$.
    
    \item $d_Y\big(p(x),p(x')\big)\geq \frac{1}{L}d_X(x,x')-u(\max\{|x|, |x'|\})$
    \end{enumerate}
    
Then for all $r>R_0$, 
$$b_Y(r)\geq b_X\big(r-u(r)\big)/b_X^s\Big(L u\big(r-u(r)\big)\Big)$$

\end{prop}

\begin{proof}
The proof of Proposition $3.6.$ in \cite{cornulier2017sublinear} carries over without any changes. 
\end{proof}

    



\begin{cor} \label{T: cor: La has same critical exponent as Ga}
Let $u(r)\preceq_\infty \varepsilon  r$ for some $\varepsilon<\frac{1}{2}$. Assume that $\G\subset \mn_u(\La)$. Then $\delta(\La)\geq (1-4\varepsilon) \delta(\G)$.
\end{cor}

\begin{rmk}
Restricting to $\varepsilon<\frac{1}{2}$ stems from a $2$ factor that appears in the proof and could possibly be dropped using a slightly more sophisticated approach. For my needs this is more than enough, since in any case I eventually restrict attention to a small interval around $0$. 
\end{rmk}

Corollary~\ref{T: cor: La has same critical exponent as Ga} can be formulated in a slightly more general fashion. Using the notation $\delta(W):=\limsup_{r\rightarrow \infty}\frac{\log\big(b_W(r)\big)}{r}$ for a general subset $W$ in a general metric space $X$, the following version holds: 

\begin{cor}\label{T: cor: general formulation of critical exponent discrepancy}
Let $(X,x_0)$ be a pointed metric space, $Y,Z\subset X$ two subsets, and $u(r)\preceq_\infty \varepsilon r$ for some $\varepsilon<\frac{1}{2}$. Assume that $b_Z^s(r)=b_Z(r)$. If $Z\subset \mn_u(Y)$, then $\delta(Y)\geq (1-4\varepsilon) \delta(Z)$. Moreover, if $u$ is sublinear, then $\delta(Z)=\delta(Y)$.
\end{cor}

In particular, Corollary~\ref{T: cor: La has same critical exponent as Ga} holds even when the group $G$ does not have property $(T)$. Corollary~\ref{T: cor: La has same critical exponent as Ga} follows from Corollary~\ref{T: cor: general formulation of critical exponent discrepancy} because the fact that $\G$ is a group of isometries implies  $b_\G^s(r)=b_\G(r)$.

\begin{proof}[Proof of Corollary~\ref{T: cor: general formulation of critical exponent discrepancy}]
Consider the closest point projection $p_Y:Z\rightarrow Y$, denote $y_z:=p_Y(z)$ and observe: 

\begin{enumerate}
    \item $|y_z|\leq |z| + u(|z|)$.
    \item $d\big(y_z,y_{z'} \big)\geq  d(z,z')-2u(\max\{|z|,|z'|\})$.
\end{enumerate}

The first item follows from the fact $y_z\in \overline{B}\big(z,u(|z|)\big)$ and triangle inequality. The second item follows from the quadrilateral inequality, i.e.\ using triangle inequality twice along the quadrilateral $[z,z',y_{z'},y_z]$.

I can now use Proposition~\ref{T: prop: Upgrade Cornulier growth discrepency} with constant $L=1$ and function $u'=2u$ to get 

$$b_Y(r)\geq b_Z\big(r-u'(r)\big)/b_Z^s\Big(u'\big(r-u'(r)\big)\Big).$$

Since I assume $b_Z^s=b_Z$, I can omit the superscript $s$ in the last expression. It remains to prove: 

$$\limsup_{r\rightarrow\infty}\frac{1}{r} \log\bigg(b_Z\big(r-u'(r)\big)/b_Z\Big(u'\big(r-u'(r)\big)\Big)\bigg)\geq (1-4\varepsilon) \delta(Z).$$

The proof of this inequality involves nothing more than $\log$ rules and arithmetic of limits: 

\begin{equation} \label{T: eq: negligible denominator}
\begin{split}
&\limsup_{r\rightarrow\infty}\frac{1}{r} \log\bigg(b_Z\big(r-u'(r)\big)/b_Z\Big(u'\big(r-u'(r)\big)\Big)\bigg) \\
& = \limsup_{r\rightarrow\infty}\frac{1}{r}\Bigg( \log\bigg(b_Z\big(r-u'(r)\big)\bigg)-\log\bigg(b_Z\Big(u'\big(r-u'(r)\big)\Big)\bigg)\Bigg) \\
& \geq \limsup_{r\rightarrow\infty}\Bigg[\frac{1}{r} \log\bigg(b_Z\big(r-u'(r)\big)\bigg)-\limsup_{s\rightarrow\infty}\bigg[\frac{1}{s}\log\bigg(b_Z\Big(u'\big(s-u'(s)\big)\Big)\bigg)\bigg]\Bigg] \\
& = \limsup_{r\rightarrow\infty}\frac{1}{r} \log\bigg(b_Z\big(r-u'(r)\big)\bigg) - \limsup_{s\rightarrow\infty}\frac{1}{s}\log\bigg(b_Z\Big(u'\big(s-u'(s)\big)\Big)\bigg) \\
& \geq  \limsup_{r\rightarrow\infty}\frac{1}{r} \log\bigg(b_Z\big(r-2\varepsilon r\big)\bigg) - \limsup_{s\rightarrow\infty}\frac{1}{s}\log\bigg(b_Z\Big(2\varepsilon s\Big)\bigg) \\
& = (1-2\varepsilon) \delta(Z)-2\varepsilon \delta(Z) = (1-4\varepsilon)\delta(Z)
\end{split}
\end{equation}

where:
\begin{enumerate}
    \item First equality is by rules of $\log$.
    \item Second and third inequalities are by arithmetic of limits: let $(a_n)_n,(b_n)_n$ be two sequences of positive numbers, and $A=\limsup_n a_n,\  B=\limsup_n b_n$. Then $\limsup(a_n-b_n)\geq \limsup_n (a_n-B)=A-B$.
    \item Fourth inequality: $u'(r)<2\varepsilon(r)$ for all large enough $r$.
    \item Fifth equality: definition of $\delta$.
\end{enumerate}

This completes the proof in the general case, which is what is needed for the proof of Theorem~\ref{T: thm: main epsilon}. For the more refined statement in the case $u$ is sublinear, one has to show a bit more. From inequality~\ref{T: eq: negligible denominator} (specifically from the fourth line of the inequality) it is clearly enough to prove:
\begin{enumerate}
    \item  $\limsup_{r\rightarrow\infty}\frac{1}{r} \log\bigg(b_Z\big(r-u'(r)\big)\bigg)=\delta(Z)$.
    \item $\limsup_{s\rightarrow\infty}\frac{1}{s}\log\bigg(b_Z\Big(u'\big(s-u'(s)\big)\Big)\bigg)=0$.
\end{enumerate}

Starting from the second item, indeed it holds that
$$\frac{1}{s} \log\bigg(b_Z\Big(u'\big(s-u'(s)\big)\Big)\bigg) = \frac{u'\big(s-u'(s)\big)}{s}\cdot\frac{\log \bigg(b_Z \Big(u'\big(s-u'(s)\big)\Big)\bigg)}{u'\big(s-u'(s)\big)}.$$

Clearly $\limsup$ of the right factor in the above product is bounded by $\delta(Z)$, and in particular it is uniformly bounded. On the other hand sublinearity of $u'$ implies that the left factor tends to $0$. I conclude that this product tends to $0$ as $s$ tends to $\infty$. 

It remains to prove $\limsup_{r\rightarrow\infty}\frac{1}{r} \log\bigg(b_Z\big(r-u'(r)\big)\bigg) = \delta(Z)$. In a similar fashion, 

$$\frac{1}{r} \log\Big(b_Z\big(r-u'(r)\big)\Big) = \frac{r-u'(r)}{r}\cdot\frac{\log\Big(b_Z\big(r-u'(r)\big)\Big)}{r-u'(r)}.$$

The left factor limits to $1$ by sublinearity of $u'$. The right factor is nearly the expression in the definition of $\delta(Z)$, and I want to prove that indeed taking $\limsup$ of it equals $\delta(Z)$. \emph{A priori} $\{r-u'(r)\}_{r\in\mrp}$ is just a subset of $\mrp$, so changing variable and writing $t:=r-u'(r)$ requires a justification. But there is no harm in assuming that $u'$ is a non-decreasing continuous function, hence $\mathbb{R}_{\geq R}\subset\{r-u'(r)\}_{r\in \mrp}$ for some $R\in \mrp$. Therefore for any sequence $r_n\rightarrow \infty$ there is a sequence $r_n'$ with $r_n=r_n'-u'(r_n')$ for all large enough $n$ (note that in particular $r_n'\rightarrow \infty$). In the other direction, for every sequence $r_n'\rightarrow \infty$ there is clearly a sequence $r_n\rightarrow \infty$ for which $r_n=r_n'-u'(r_n')$. I conclude 
$$\limsup_{r\rightarrow \infty}\frac{\log\Big(b_Z\big(r-u'(r)\big)\Big)}{r-u'(r)}=\limsup_{r\rightarrow \infty}\frac{1}{r}\cdot \log\big(b_Z(r)\big)=\delta(Z)$$

completing the proof.

\end{proof}

\begin{proof}[Proof of Theorem~\ref{T: thm: main epsilon}]
Define $\varepsilon(G)=\frac{c^*(G)}{4\cdot 2\Vert\rho\Vert}$, and assume $u(r)\preceq_\infty \varepsilon(G) r$. Notice that $\varepsilon(G)<\frac{1}{2}$, and since $\delta(\G)=2\Vert\rho\Vert$ Corollary~\ref{T: cor: La has same critical exponent as Ga} gives

$$\delta(\La)\geq \big(1-4\varepsilon(G)\big)\cdot 2\Vert\rho\Vert=2\Vert\rho\Vert-4\varepsilon(G)\cdot 2\Vert\rho\Vert\geq 2\Vert\rho\Vert-c^*(G)$$

By Theorem~\ref{T: thm: Leuzinger criterion with epsilon bound}, $\La$ is a lattice.  
\end{proof}

Recall that since its hypothesis is stronger, Theorem~\ref{thm: main for T} follows from Theorem~\ref{T: thm: main epsilon}.

\begin{rmk}
Whether or not there are interesting groups that coarsely cover a lattice is a key question that arises naturally from this paper. The first question that comes to mind is whether there exist groups that are not commensurable to a lattice but that sublinearly, or even $\varepsilon$-linearly, cover one. Perhaps the growth rate point of view could be used to rule out groups that cover a lattice $\varepsilon$-linearly but not sublinearly. 
\end{rmk}

\subsection{Proof of Corollary~\ref{thm: SBE main Intro}}\label{sec: SBE completeness corollary}

\begin{proof}[Proof of Corollary~\ref{thm: SBE main Intro}]
Let $G$ be a centre-free semisimple Lie group without compact or $\rr$ factors, $\G$ an irreducible lattice in $G$ and $\La$ an abstract finitely generated group $(L,u)$-SBE to $\G$, with $u$ subadditive. Theorem~\ref{intro: thm: from SBE to sublinearly close} yields a group homomorphism $\Phi: \La\rightarrow G$ with finite Kernel such that $\Phi(\La)$ and $\G$ sublinearly cover each other. Since $G$ has property (T), $\Phi(\La)$ is a lattice by Theorem~\ref{thm: main for T}. If either $\G$ or $\Phi(\La)$ is uniform, it is clear that the other must be uniform as well - e.g. by the linear rate of penetration into horoballs (Lemma~\ref{qr: prelims: sym spc: lem: Hattori penetration}). 
\end{proof}

\section{Uniform Lattices}\label{sec: u}
In this section I prove sublinear rigidity for uniform lattices: 

\begin{thm} \label{U: thm: main}
Let $G$ be a finite-centre semisimple Lie group without compact factors. Let $\G\leq G$ be a lattice, $\La\leq G$ a discrete subgroup such that $\G\subset\mn_u(\La)$ for some sublinear function $u$. If $\G$ is uniform, then $\La$ is a uniform lattice. 
\end{thm}

As in the case of lattices with property (T), uniform lattices admit the stronger version of $\varepsilon$-linear rigidity, for any $\varepsilon<1$:

\begin{thm} \label{U: thm: main epsilon}
In the setting of Theorem~\ref{U: thm: main}, the conclusion holds also under the relaxed assumption that $u(r)\preceq_\infty  \varepsilon r$ for any $0<\varepsilon<1$.
\end{thm}

Clearly Theorem~\ref{U: thm: main epsilon} implies Theorem~\ref{U: thm: main}. Throughout this section, the standing assumptions are those of Theorem~\ref{U: thm: main epsilon}.

\paragraph{\textbf{Lattice Criterion.}}

A discrete subgroup is a uniform lattice if and only if it admits a relatively compact fundamental domain. The criterion I use is the simple observation that when $\G$ is uniform and $u$ is bounded (i.e.\ $\G\subset \mn_D(\La)$ for some $D>0$), then $\La$ is a uniform lattice.

\paragraph{\textbf{Argument and Use of $\varepsilon$-Linearity.}}

The strategy is to prove the $\varepsilon$-linearity of $u$ forces $\G\subset \mn_D(\La)$, i.e.\ that $\G$ actually lies inside a bounded neighbourhood of $\La$. The proof is by way of contradiction. If there is no such $D>0$ then there are arbitrarily large balls that do not intersect $\La$. One finds such large $\La$-free balls that are all tangent to some fixed arbitrary point $x\in X$ (see Figure~\ref{fig: basic setting}). The $\varepsilon$-linearity then gives rise to concentric $\G$-free balls that are arbitrarily large, contradicting the cocompactness of $\G$.

\begin{rmk}
The main difference from the non-uniform case is that for a non-uniform lattice $\G$, the space $X$ does admit arbitrarily large $\G$-free balls. This situation requires different lattice criteria and much extra work. Still, and while essentially no more than a few lines,  the proof for the uniform case lays the foundations for and presents the logic of the much more involved case of \qr lattices. 
\end{rmk}

\subsection{Notations and Terminology}

\label{sec: U: Notations}



The following definitions will be used repeatedly in both this section and in Section~\ref{sec: qr}. They mainly fix terminology and notation of the geometric situation illustrated in Figure~\ref{fig: basic setting}. 

\begin{defn}
Let $H\leq G=\mathrm{Isom}(X)^\circ$. A set $U\subset X$ is called \emph{$H$-free} if $H x_0\cap \mathrm{Int}(U)=\emptyset$, where $\mathrm{Int}(U)$ is the topological interior of $U$. That is, $U$ is called $H$-free if its interior does not intersect the $H$-orbit $H x_0$. 
\end{defn}

\begin{defn}
Denote $P_{\La}(\g x_0)=P_\La(\g)=\la_\g x_0$.

\begin{enumerate}
    \item $d_\g:=d(\g x_0,\la_\g x_0)$.
    
    \item $B_\g:=B(\gamma x_0,d_\g)$. It is a $\La$-free ball centred at $\g x_0$ and tangent to $\la_\g x_0$.
    
    \item $x_\g':=\la_\g^{-1}\g x_0$. Notice $|x_\g'|=d_\g$.
    
    \item $B_\g':=\la_\g^{-1}B_\g=B(x_\g',d_\g)$. It is $\La$-free as a $\La$-translate of the $\La$-free ball $B_\g$, and is tangent to $x_0$.
   
    \item For $s\in \mrp$ and a ball $B=B(x,r)$, denote $sB:=B(x,sr)$, the rescaled ball with same centre and radius $sr$.
    
    \item For a sequence $\g_n$, denote by $\la_n,d_n,B_n,B_n',x_n'$ the respective $\la_{\g_n}, d_{\g_n}$, etc.
\end{enumerate}
\end{defn}

\begin{figure}
    \centering
    \begin{tikzpicture}[x=0.75pt,y=0.75pt,yscale=-1,xscale=1]
\path (0,300); 

\draw  [fill={rgb, 255:red, 0; green, 0; blue, 0 }  ,fill opacity=1 ] (273,128) .. controls (273,126.9) and (273.9,126) .. (275,126) .. controls (276.1,126) and (277,126.9) .. (277,128) .. controls (277,129.1) and (276.1,130) .. (275,130) .. controls (273.9,130) and (273,129.1) .. (273,128) -- cycle ;
\draw   (241,70.5) .. controls (241,35.43) and (269.43,7) .. (304.5,7) .. controls (339.57,7) and (368,35.43) .. (368,70.5) .. controls (368,105.57) and (339.57,134) .. (304.5,134) .. controls (269.43,134) and (241,105.57) .. (241,70.5) -- cycle ;
\draw  [color={rgb, 255:red, 208; green, 2; blue, 27 }  ,draw opacity=1 ] (257.67,70.5) .. controls (257.67,44.63) and (278.63,23.67) .. (304.5,23.67) .. controls (330.37,23.67) and (351.33,44.63) .. (351.33,70.5) .. controls (351.33,96.37) and (330.37,117.33) .. (304.5,117.33) .. controls (278.63,117.33) and (257.67,96.37) .. (257.67,70.5) -- cycle ;
\draw    (306.5,70.5) -- (275,126) ;
\draw  [fill={rgb, 255:red, 0; green, 0; blue, 0 }  ,fill opacity=1 ] (302.5,70.5) .. controls (302.5,69.4) and (303.4,68.5) .. (304.5,68.5) .. controls (305.6,68.5) and (306.5,69.4) .. (306.5,70.5) .. controls (306.5,71.6) and (305.6,72.5) .. (304.5,72.5) .. controls (303.4,72.5) and (302.5,71.6) .. (302.5,70.5) -- cycle ;
\draw    (285,28) -- (304.5,68.5) ;
\draw [color={rgb, 255:red, 208; green, 2; blue, 27 }  ,draw opacity=1 ]   (394,64) -- (332.99,69.81) ;
\draw [shift={(331,70)}, rotate = 354.56] [color={rgb, 255:red, 208; green, 2; blue, 27 }  ,draw opacity=1 ][line width=0.75]    (10.93,-3.29) .. controls (6.95,-1.4) and (3.31,-0.3) .. (0,0) .. controls (3.31,0.3) and (6.95,1.4) .. (10.93,3.29)   ;
\draw    (409,20) -- (351.92,36.77) ;
\draw [shift={(350,37.33)}, rotate = 343.63] [color={rgb, 255:red, 0; green, 0; blue, 0 }  ][line width=0.75]    (10.93,-3.29) .. controls (6.95,-1.4) and (3.31,-0.3) .. (0,0) .. controls (3.31,0.3) and (6.95,1.4) .. (10.93,3.29)   ;
\draw   (41,171.5) .. controls (41,136.43) and (69.43,108) .. (104.5,108) .. controls (139.57,108) and (168,136.43) .. (168,171.5) .. controls (168,206.57) and (139.57,235) .. (104.5,235) .. controls (69.43,235) and (41,206.57) .. (41,171.5) -- cycle ;
\draw    (104,171.5) -- (72.5,227) ;
\draw  [fill={rgb, 255:red, 0; green, 0; blue, 0 }  ,fill opacity=1 ] (102.5,171.5) .. controls (102.5,170.4) and (103.4,169.5) .. (104.5,169.5) .. controls (105.6,169.5) and (106.5,170.4) .. (106.5,171.5) .. controls (106.5,172.6) and (105.6,173.5) .. (104.5,173.5) .. controls (103.4,173.5) and (102.5,172.6) .. (102.5,171.5) -- cycle ;
\draw    (54,103) -- (83.59,132.59) ;
\draw [shift={(85,134)}, rotate = 225] [color={rgb, 255:red, 0; green, 0; blue, 0 }  ][line width=0.75]    (10.93,-3.29) .. controls (6.95,-1.4) and (3.31,-0.3) .. (0,0) .. controls (3.31,0.3) and (6.95,1.4) .. (10.93,3.29)   ;
\draw  [fill={rgb, 255:red, 0; green, 0; blue, 0 }  ,fill opacity=1 ] (75,227) .. controls (75,225.62) and (73.88,224.5) .. (72.5,224.5) .. controls (71.12,224.5) and (70,225.62) .. (70,227) .. controls (70,228.38) and (71.12,229.5) .. (72.5,229.5) .. controls (73.88,229.5) and (75,228.38) .. (75,227) -- cycle ;
\draw    (167,212) .. controls (235.95,217.91) and (259.3,184.04) .. (268.59,144.8) ;
\draw [shift={(269,143)}, rotate = 102.68] [color={rgb, 255:red, 0; green, 0; blue, 0 }  ][line width=0.75]    (10.93,-3.29) .. controls (6.95,-1.4) and (3.31,-0.3) .. (0,0) .. controls (3.31,0.3) and (6.95,1.4) .. (10.93,3.29)   ;

\draw (277,112) node [anchor=north west][inner sep=0.75pt]    {$x_{0}$};
\draw (235.36,107.38) node [anchor=north west][inner sep=0.75pt]  [font=\footnotesize,rotate=-1.84]{};
\draw (278,60) node [anchor=north west][inner sep=0.75pt]    {$x_{\gamma } '$};
\draw (285,85) node [anchor=north west][inner sep=0.75pt]    {$=d_{\gamma }$};
\draw (304.5,23.67) node [anchor=north west][inner sep=0.75pt]  [font=\footnotesize,rotate=-60.58]  {$=s\cdot d_{\gamma }$};
\draw (401.9,55) node [anchor=north west][inner sep=0.75pt]  [color={rgb, 255:red, 208; green, 2; blue, 27 }  ,opacity=1 ]  {$\Gamma -free$};
\draw (397,-4) node [anchor=north west][inner sep=0.75pt]    {$\Lambda -free$};
\draw (74,153) node [anchor=north west][inner sep=0.75pt]    {$\gamma x_{0}$};
\draw (93,187) node [anchor=north west][inner sep=0.75pt]    {$=d_{\gamma }$};
\draw (29,78) node [anchor=north west][inner sep=0.75pt]    {$\Lambda -free$};
\draw (45,229) node [anchor=north west][inner sep=0.75pt]    {$\lambda _{\gamma } x_{0}$};
\draw (228,202) node [anchor=north west][inner sep=0.75pt]    {$L_{\lambda _{\gamma }^{-1}}$};

\end{tikzpicture}

    \caption[Basic setting of $\La$-free balls]{Basic Setting and Lemma~\ref{U: lem: large Lambda free balls near base point imply large Gamma free balls}. A $\Lambda$-free ball about $\g x_0$ of radius $d_\g$, translated by $\la_\g^{-1}$ to a ball tangent to $x_0$. The linear ratio between $|x_\g'|=d_\g$ and the $\La$-free radius forces the red ball to be $\G$-free.}
    \label{fig: basic setting}
\end{figure}
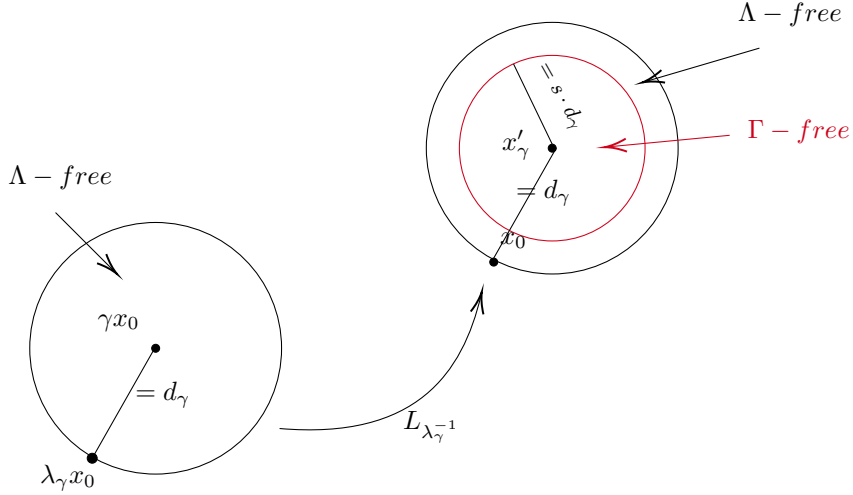

\subsection{Proof of Theorem~\ref{U: thm: main epsilon}}

\begin{lem} \label{U: lem: large Lambda free balls near base point imply large Gamma free balls}

Let $x\in X$. There exists $S=S(x,u)\in (0,1)$ such that for every $s\in (0,S)$ there is $R=R(s,S)$ such that if $r>R$ and $B=B(y,r)$ is a $\La$-free ball tangent to $x$, then $sB$ is $\G$-free.

In particular, the existence of arbitrarily large $\La$-free balls that are all tangent to a fixed point $x\in X$ implies the existence of arbitrarily large $\G$-free balls. 
\end{lem}

There is a slightly stronger version of  Lemma~\ref{U: lem: large Lambda free balls near base point imply large Gamma free balls} if $u$ is sublinear:

\begin{lem}\label{U: lem: sublinear version of large Lambda free balls near base point imply large Gamma free balls}
Let $G,\G,\La$ and $u$ be as in Theorem~\ref{U: thm: main} (in particular, $u$ is a sublinear function and $\G\subset\mn_u(\G)$). For every $x\in X$ and every $s\in (0,1)$ there exists $R=R(x,s)>0$ such that for every $r>R$, if $B=B(y,r)$ is a $\La$-free ball tangent to $x$ then $sB$ is $\G$-free.
\end{lem}

I omit the proof of Lemma~\ref{U: lem: sublinear version of large Lambda free balls near base point imply large Gamma free balls}, which is a slightly simpler version of the proof of Lemma~\ref{U: lem: large Lambda free balls near base point imply large Gamma free balls}.

\begin{proof}[Proof of Lemma~\ref{U: lem: large Lambda free balls near base point imply large Gamma free balls}.]

The proof is more easily read if one assumes $x=x_0$ and $u(r)=\varepsilon r$ so I begin with this case. Assume $B=B(y,R)$ is $\La$-free for some $y\in X, r>0$. The assumption $x=x_0$ gives $|y|=d(y,x_0)=r$. For a fixed $s\in (0,1)$, assume $sB\cap \G x_0\ne\emptyset$. This gives rise to an element $\g\in \G$ such that:
\begin{enumerate}
    \item $|\g|=d(\g x_0, x_0)\leq d(y,x_0)+d(\g x_0,y)=(1+s)r$. In particular, $d(\g x_0,\La x_0)\leq \varepsilon(1+s)r$.
    \item $B\big(\g x_0,(1-s)r\big)\subset B(y,r)$, so it is $\La$-free.
\end{enumerate}

I conclude that for such $s$ one has the inequality $(1-s)r\leq \varepsilon(1+s)r$, i.e.\ $\frac{1-s}{1+s}\leq \varepsilon$. The constant $\varepsilon$ is fixed and smaller than $1$, while $\frac{1-s}{1+s}$ limit to $1$ monotonically from below as $s>0$ tend to $0$. I conclude that there is a segment $(0,S)\subset (0,1)$ such that for all $s\in (0,S)$, $sB$ is $\G$-free (independently of $r$).

Assume now that $x\ne x_0$ and $u(r)\preceq_\infty \varepsilon r$. As above, if $\g x_0\in B(y,sr)$ then it is the centre of a $\La$-free ball of radius $(1-s)r$, and so $(1-s)r\leq u(|\g|)$. I wish to use the $\varepsilon$-linear bound on $u$ as I did before, only this time $u$ is only asymptotically smaller than $\varepsilon r$. To circumvent this I just need to show that $|\g|$ is large enough. Indeed since $B\big(\g x_0,(1-s)r\big)$ is $\La$-free it does not contain $x_0\in\La x_0$ and in particular $(1-s)r\leq d(x_0, \g x_0)=|\g|$. For some $R_1(s)=R_1(s,u)$ one therefore has for all $r>R_1(s)$

$$(1-s)r\leq u(|\g|)\leq \varepsilon|\g|.$$

On the other hand $|y|\leq d(x,y)+d(x,x_0)=r+|x|$. Consequently $|\g|\leq (1+s)r+|x|$ and, for $r>R_1(s)$, 

$$(1-s)r\leq u(|\g|)\leq \varepsilon|\g|\leq \varepsilon \big((1+s)r+|x|\big).$$

This means that $s$ for which $\G x_0\cap B(y,sr)\ne\emptyset$ must admit, for all $r>R_1(s)$,

\begin{equation} \label{U: eq: constraint on s if gamma in By,sr}
  \frac{1-s}{1+s+\frac{|x|}{r}}=\frac{(1-s)r}{(1+s)r+|x|}\leq \varepsilon<1.  
\end{equation}

The rest of the proof is just Calculus 1, and concerns with finding $S=S(x,\varepsilon,u)\in (0,1)$ so that for any $s\in (0,S)$ there is $R(s)$ such that all $r>R(s)$ satisfy

\begin{equation} \label{U: eq: constraint on s if it is small and r large}
  \varepsilon<\frac{1-S}{1+S+\frac{|x|}{r}}\leq \frac{1-s}{1+s+\frac{|x|}{r}} 
\end{equation}
$$$$

The lemma readily follows from Inequalities~\ref{U: eq: constraint on s if gamma in By,sr} and \ref{U: eq: constraint on s if it is small and r large}.

Explicitly, fix $\varepsilon'>\varepsilon$. As before, the monotonic approach of $\frac{1-s}{1+s}$ to $1$ allows one to fix $S\in (0,1)$ for which $\varepsilon<\varepsilon'<\frac{1-s}{1+s}$ for all $s\in (0,S)$. Next note that for any fixed $s\in (0,S)$, $\lim_{r\rightarrow \infty}\frac{1-s}{1+s+\frac{|x|}{r}}=\frac{1-s}{1+s}$, and that the approach in monotonically increasing with $r$. Since $\varepsilon<\varepsilon'$, this limit implies that for some $R_2>R_1(S)$, all $r>R_2$ admit $\varepsilon<\frac{1-S}{1+S+\frac{|x|}{r}}$. Finally notice that for any fixed $r$ the function $\frac{1-s}{1+s+\frac{|x|}{r}}$ is again monotonically increasing as $s$ tends to $0$ from above. Therefore inequality~\ref{U: eq: constraint on s if it is small and r large} holds for every $s\in  (0,S)$ and all $r>R_2(S)$ (capital $S$ is intentional and important). 

To conclude the proof, notice that if moreover $r>R_1(s)$ (again lowercase $s$ is intentional and important) then Inequalities \ref{U: eq: constraint on s if gamma in By,sr} and \ref{U: eq: constraint on s if it is small and r large} both hold. This means that for any $s\in (0,S)$ there is $R(s):=\max\{R_1(s),R_2(S)\}$ such that $r>R(s)\Rightarrow B(y,sr)$ is $\G$-free. The constants $R_1(s),R_2(S)$ have the desired dependencies, hence so does $R(s)$, proving the lemma. 
\end{proof}

\begin{cor} [Theorem~\ref{U: thm: main epsilon}]
There is a uniform bound on $\{d_\g\}_{\g\in\G}$, i.e., $\G\subset \mn_D (\La)$ for some $D>0$. In particular, $\La$ is a uniform lattice.
\end{cor}

\section{\texorpdfstring{$\mathbb{Q}$}{-}-rank 1 Lattices}\label{sec: qr}

In this section I prove sublinear rigidity for \qr lattices: 

\begin{thm} \label{qr: thm: main for qr}
Let $G$ be a real finite-centre semisimple Lie group without compact factors, $\G\leq G$ an irreducible non-uniform \qr lattice, $\La\leq G$ a discrete irreducible subgroup. If $\G\subset\mn_u(\La)$ for some sublinear function $u$, then $\La$ is a lattice. 
\end{thm}

If $\G\subset\mn_D(\La)$ for some $D>0$, then $\La$ could be a uniform lattice. An obvious obstacle to that is if $\La\subset\mn_{u'}(\G)$ for some sublinear function $u'$. In this case the groups must lie at bounded Hausdorff distance.

\begin{prop}\label{qr: prop: bounded and sublinear implies bounded}\label{qr: prop: bounded and sublinear implies bunded and bounded}
Let $G$ be a real finite-centre semisimple Lie group without compact factors, $\G\leq G$ an irreducible lattice of \qrc $\La\leq G$ a discrete subgroup, $D>0$ and $u$ a sublinear function. If $\G\subset \mn_D(\La)$ and $\La\subset \mn_u(\G)$, then $\La\subset \mn_{D'}(\G)$ for some $D'>0$. Moreover, if $G$ is of $\rr$, the conclusion holds under the relaxed assumption that  $u(r)\preceq_\infty \varepsilon r$ for some $\varepsilon<1$. 
\end{prop}

\begin{rmk}
While the setting of Proposition~\ref{qr: prop: bounded and sublinear implies bounded} seems rather limited, the situation where $\G\subset\mn_u(\La)$ and $\La\subset\mn_u(\G)$ arises naturally from the motivating example of SBE-completeness. However, Theorem~\ref{intro: thm: from SBE to sublinearly close} is not known for groups with $\rr$ factors, which is the only setting for which I can prove Proposition~\ref{qr: prop: bounded and sublinear implies bounded}.
\end{rmk}


\subsection{Strategy}
\paragraph{\textbf{Lattice Criteria.}}

I use three different lattice criteria, depending on the $\mathbb{R}$-rank of $G$ and on whether or not $\G\subset\mn_D(\La)$ for some $D>0$. My proof is motivated by a conjecture of Margulis, that can be viewed as an algebraic converse to the geometric structure of the compact core described in Theorem~\ref{qr: prelims: sym spc: thm: Leuzinger compact core}. The conjecture was resolved in many cases by Oh, Benoist, and Miquel. Recently, Benoist-Miquel gave a unified proof resolving the conjecture in full generality (Section $1.2$ in~\cite{BenoistMiquelArithmeticity} discusses the history of this result).

\begin{thm}[Theorem $2.16$ in~\cite{BenoistMiquelArithmeticity}] \label{qr: thm: Benoist-Miquel arithmeticity}
Let $G$ be a semisimple real algebraic Lie group of real rank at least $2$ and $U$ be a non-trivial horospherical subgroup of $G$. Let $\Delta$ be a
discrete Zariski-dense subgroup of $G$ that contains an indecomposable lattice
$\Delta_U$ of $U$. Then $\Delta$ is a non-uniform irreducible arithmetic lattice of $G$.
\end{thm}

See Definition~\ref{qr: defn: indecompsable and irreducibale lattice} for the precise meaning of an \emph{indecomposable} horospherical lattice.
When $G$ is of high rank, the strategy is to show that $\Lambda$ admits the hypotheses of Theorem~\ref{qr: thm: Benoist-Miquel arithmeticity}.

If $G$ is of $\rr$, the following theorem by Kapovich and Liu states that a subgroup is geometrically finite so long as `most' of its limit points are conical. Recall $\mathcal{L}(\Delta)$ is the limit set of $\Delta\leq \mathrm{Isom}(X)$, and  $\mathcal{L}_{\mathrm{con}}(\Delta)$ is the set of its conical limit points.

\begin{thm}[Theorem $1.5$ in \cite{KapovichLiuGeometricFI}]\label{qr: thm: Kapovich-Liu conical limit points}
Let $X$ be a $\rr$ symmetric space. A discrete subgroup $\Delta\leq \mathrm{Isom}(X)$ is geometrically infinite if and only if the set $\mathcal{L}(\Delta)\setminus\mathcal{L}_{\mathrm{con}}(\Delta)$ of non-conical limit points has the cardinality of the continuum.
\end{thm}

\begin{cor}[Lattice Criterion in $\rr$]\label{qr: cor: conical limit set criterion}
Let $X$ be a $\rr$ symmetric space, $\G\leq G=\mathrm{Isom}(X)$ a non-uniform lattice and $\La\leq G$ a discrete subgroup. If $\mathcal{L}(\La)=X(\infty)$ and $\mathcal{L}_{\mathrm{con}}(\G)\subset\mathcal{L}_{\mathrm{con}}(\La)$, then $\La$ is a lattice. 
\end{cor}

\begin{proof}
Since $\G$ is a lattice, $\mathcal{L}(\G)=X(\infty)$ and it is geometrically finite. Theorem~\ref{qr: thm: Kapovich-Liu conical limit points} implies the cardinality of $X(\infty)\setminus \mathcal{L}_{\mathrm{con}}(\G)$ is strictly smaller than the continuum. The assumption  $\mathcal{L}_{\mathrm{con}}(\G)\subset\mathcal{L}_{\mathrm{con}}(\La)$ implies the same holds for $\La$, and in particular that $\La$ is geometrically finite. The assumption that $\mathcal{L}(\La)=X(\infty)$ implies that $\La$ is geometrically finite if and only if it is a lattice. 
\end{proof}

\begin{cor}\label{qr: cor: bounded case real rank 1 one sided}
Let $X$ be a $\rr$ symmetric space, $\G\leq G=\mathrm{Isom}(X)$ a non-uniform lattice and $\La\leq G$ a discrete subgroup. If $\G\subset\mn_D(\La)$ for some $D>0$, then $\La$ is a lattice.
\end{cor}

\begin{proof}
In this setting $\mathcal{L}(\La)=\mathcal{L}(\G)=X(\infty)$ and $\mathcal{L}_{\mathrm{con}}(\G)\subset\mathcal{L}_{\mathrm{con}}(\La)$, hence $\La$ is a lattice by Corollary~\ref{qr: cor: conical limit set criterion}.
\end{proof}

Also in high rank the inclusion $\G\subset\mn_D(\La)$ implies that $\La$ is a lattice. This result is due to Eskin.

\begin{thm}[Eskin, see Remark~\ref{qr: rmk: problem with bounded case proofs} below]\label{qr: thm: eskin bounded case}
Let $G$ be a real finite-centre semisimple Lie group without compact factors and of high rank, $\G\leq G$ an irreducible non-uniform lattice, $\La\leq G$ a discrete subgroup such that $\G\subset\mn_D(\La)$ for some $D>0$. Then $\La$ is a lattice. If moreover $\La\subset\mn_D(\G)$ then $\La$ and $\G$ are commensurable. 
\end{thm}

Theorem~\ref{qr: thm: eskin bounded case} was used in the proof of quasi-isometric rigidity for high rank non-uniform lattices in \cite{DrutuQI} and \cite{EskinLatticeClassification}. In the (earlier) $\rr$ case, Schwartz~\cite{Schwartz} used an analogous statement that requires one extra assumption. 

\begin{thm}[Schwartz, see Section $10.4$ in \cite{Schwartz} and Remark~\ref{qr: rmk: problem with bounded case proofs} below]\label{qr: thm: Schwartz: finite distance imples lattice}
Let $G$ be a real finite-centre simple Lie group of $\rr$, $\G\leq G$ non-uniform lattice, $\La\leq G$ a discrete subgroup such that both $\G\subset\mn_D(\La)$ and $\La\subset\mn_D(\G)$ for some $D>0$. Then $\La$ is a lattice. If moreover $G$ is not locally isomorphic to $\mathrm{SL}_2(\mathbb{R})$, then $\La$ and $\G$ are commensurable.
\end{thm}

\begin{rmk}\label{qr: rmk: problem with bounded case proofs}
Theorem~\ref{thm: main Intro} should be viewed as a generalization of the bounded case depicted in Theorems~\ref{qr: thm: eskin bounded case} and~\ref{qr: thm: Schwartz: finite distance imples lattice}, which were known to experts in the field in the late 1990's. Complete proofs for these statements were never given in print, and I take the opportunity to include them here in  Section~\ref{sec: qr: bounded case}. I thank Rich Schwartz and Alex Eskin for supplying me with their arguments and allowing me to include them in this paper. I also thank my thesis examiner Emmanuel Breuillard for encouraging me to find and make these proofs public.
\end{rmk}

\paragraph{\textbf{Argument and Use of Sublinearity.}}
Lattices of \qr admit a concrete geometric structure that is manifested in the geometry of an orbit in the corresponding symmetric space (see Section~\ref{sec: qr: compact core and rational Tits Building}). The key property is the existence of a set of horoballs which the orbit intersects only in the bounding horospheres, and in each such horosphere the orbit forms a (metric) cocompact lattice. 

Corollary~\ref{qr: cor: bounded case real rank 1 one sided} and Theorem~\ref{qr: thm: eskin bounded case} reduce the proof to the case where $\G\not\subset\mn_D(\La)$ for any $D>0$. I prove that in this case there are horospheres in $X$ which a $\La$-orbit intersects in a cocompact lattice. Together with some control on the location of these horospheres, I prove two strong properties:
\begin{enumerate}
    \item $\La x_0$ intersects a horosphere $\h\subset X$ in a cocompact lattice (Proposition~\ref{qr: prop: Lambda cocompact horosphreres}).
    \item Every $\G$-conical limit point is also a $\La$-conical limit point (Corollary~\ref{qr: cor: every Gamma conical is Lambda conical}).
\end{enumerate}

The $\rr$ case of Theorem~\ref{qr: thm: main for qr} follows directly from Corollary~\ref{qr: cor: conical limit set criterion} using the second item above. The high rank case requires a bit more, namely to deduce from the above items that $\La$ meets the hypotheses of Theorem~\ref{qr: thm: Benoist-Miquel arithmeticity}. I use a well known geometric criterion (Lemma~\ref{qr: lem: geometric criterion for zariski dense subgroups}) in order to deduce that $\La$ is Zariski-dense, and a lemma of Mostow (Lemma~\ref{qr: prelims: translation: lem: Mostow nilpotent radical intersection}) to show that it intersects a horospherical subgroup in a lattice.

\paragraph{\textbf{Outline for Section~\ref{sec: qr}.}} 
Section~\ref{sec: qr: cocompact horosphere} proves Proposition~\ref{prop: Lambda cocompact horosphreres Intro}. The included figures may take a few moments to comprehend, but I believe they are worth the effort. Section~\ref{sec: qr: bounded case} deals with the case where $\G\subset\mn_D(\La)$, and elaborates on Schwartz's and Eskin's proofs of Theorem~\ref{qr: thm: eskin bounded case} and Theorem~\ref{qr: thm: Schwartz: finite distance imples lattice}. Section~\ref{sec: qr: geometry to algebra} is devoted to the translation of the geometric results of Section~\ref{sec: qr: cocompact horosphere} to the algebraic language used in Theorem~\ref{qr: thm: Benoist-Miquel arithmeticity}. In Section~\ref{sec: qr: conclusion} I combine  everything to complete the proof of Theorem~\ref{qr: thm: main for qr}. Finally, the proof of Theorem~\ref{thm: commensurability and uniform statement intro} is presented in Section~\ref{sec: commensurability theorem proof}.

I highly recommend the reader to have a look at the uniform case in Section~\ref{sec: u} before reading this one.



\subsection{A \texorpdfstring{$\La$}{}-Cocompact Horosphere} \label{sec: qr: cocompact horosphere}

In this section I prove Proposition~\ref{prop: Lambda cocompact horosphreres Intro}. The standing assumptions of this section are those of the proposition: $\G$ is an irreducible \qr lattice and $\{d_\g\}_\g\in \G$ is unbounded (recall $d_\g:=d(\g x_0,\la_\g x_0)$). The geometric properties of $\G$ that are used in the proof are: 

\begin{enumerate}
     \item In high rank, the characterization of $\wqg$ using conical / non-horospherical limit points (Corollary~\ref{qr: cor: limit points characterization of wqg rational building}). In $\rr$, the dichotomy of limit points being either non-horospherical or conical (Theorem \ref{qr: sym spc: thm: in rank 1 all horospherical limit points are conical}). 
     
    \item $\G$-cocompactness along the horospheres of $\G$.
    \item For every point $x\in X$ and $C>0$ there is a bound $K(C)$ on the number of horospheres of $\G$ that intersect $B(x,C)$ (Corollary~\ref{qr: prelims: sym spc: cor: every ball intersects finitely many horoballs of Gamma}).
\end{enumerate}

The proof proceeds by chasing down the geometric implications of this setting. The reader may consult the concise overview of the argument in Section~\ref{Paragraph: proof of cocompact horosphere overview}.

\subsubsection{\texorpdfstring{$\La$}{}-Free Horoballs}\label{sec: Lambda free horoballs}
I retain the notations and objects defined in Section~\ref{sec: U: Notations}.

\begin{lem} \label{qr: lem: existence of Lambda free horoballs}
There exists a $\La$-free horoball tangent to $x_0$.
\end{lem}

\begin{proof}
Since $\{d_\g\}_{\g\in\G}$ is unbounded, there are $\g_n\in\G$ with $d_n=d_{\g_n}=d(\g_n,\la_n)\rightarrow \infty$ monotonically, where $\la_n\in \La$ is a $\La$-orbit point closest to $\g_n$. Denote $x_n'=\la_n^{-1}\g_n x_0$, $\eta_n:=[x_0,x_n']$, and $v_n\in S_{x_0}X$ the initial velocity vectors $v_n:=\dot{\eta_n}(0)$. The tangent space $S_{x_0}X$ is compact, so up to a subsequence, $v_{n}$ converge monotonically in angle to a direction $v\in S_{x_0}X$. Let $\eta$ be the unit speed geodesic ray emanating from $x_0$ with initial velocity $\dot{\eta}(0)=v$. Denote $\xi:=\eta(\infty)$ the limit point of $\eta$ in $X(\infty)$. 

I show that the horoball $\hb:=\cup_{t>0}B\big(\eta(t),t\big)$, based at $\xi$ and tangent to $x_0$, is $\Lambda$-free. Let $t>0$ and consider $\eta(t)$. For every $\varepsilon>0$, there is some angle $\alpha=\alpha(t,\varepsilon)$ such that any geodesic $\eta'$ with $\angle_{x_0}(\eta,\eta')<\alpha$ admits $d\big(\eta(t),\eta'(t)\big)<\varepsilon/2$. The convergence $v_n\rightarrow v$ implies $d\big(\eta(t),\eta_n(t)\big)<\varepsilon/2$ for all but finitely many $n\in\mathbb{N}$. In particular, $B\big(\eta(t),t\big)\subset \mn_\varepsilon \Big(B\big(\eta_n(t),t\big)\Big)$ for all such $n\in\mathbb{N}$. For a fixed $t\leq d_n$, it is clear that $B\big(\eta_n(t),t\big)\subset B_n'=B(x_n',d_n)$. For a fixed $t$, $t < d_n$ for all but finitely many $n\in\mathbb{N}$, therefore for $n=n(t)$ large enough

$$B\big(\eta(t),t\big)\subset \mn_\varepsilon \Big(B\big(\eta_n(t),t\big)\Big)\subset \mn_\varepsilon B_n'$$

Therefore, for every $\varepsilon>0$, $\hb\subset \bigcup_n \mn_\varepsilon(B_n')=\mn_\varepsilon \big(\bigcup_n B_n'\big)$. This implies that any point in the interior of $\hb$ is contained in the interior of one of the $\La$-free balls $B_n'$, proving $\hb$ is $\La$-free.

\end{proof}

\begin{lem}\label{qr: lem: Lambda free horoballs are based at rational tits building wqg}
Suppose $\hb$ is a $\La$-free horoball, based at some point $\xi\in X(\infty)$. Then $\xi\in\wqg$. 
\end{lem}

\begin{proof}
For any geodesic $\eta$ with limit $\xi$, the size $d(x_0,\g x_0)$ of the $\G$-orbit points $\g x_0$ that lie boundedly close to $\eta$ grows linearly in the distance to any fixed horosphere based at $\xi$, and in particular to $\h=\partial\hb$. The sublinear constraint $d(\g x_0,\la_\g x_0)\leq u(|\g|)$ together with the fact that $\hb$ is $\La$-free imply that the size of such $\g$ is bounded. In $\rr$ every limit point is either conical or in $\wqg$, proving the lemma in this case.  For high rank, the above argument actually shows more: a point $\xi'\in \mn_\frac{\pi}{2}(\xi)$ is not conical, because every geodesic with limit $\xi'\in \mn_\frac{\pi}{2}(\xi)$ entres $\hb$ at a linear rate (Lemma~\ref{qr: prelims: sym spc: lem: Hattori penetration}).  Hattori's characterization of $\wqg$ (Corollary~\ref{qr: cor: limit points characterization of wqg rational building}) implies $\xi\in\wqg$. 

\end{proof}

\begin{defn}
By Lemma~\ref{qr: lem: Lambda free horoballs are based at rational tits building wqg}, a $\La$-free horoball $\hbl$ gives rise to a horoball of $\G$ based at the same point at $X(\infty)$. Call this the \emph{horoball corresponding to $\hbl$}, denote it $\hbg$, with bounding horosphere $\hg$.
\end{defn}

\begin{rmk}
Assuming $\La\subset\mn_u(\G)$ implies that $\La$ preserves $\wqg$ (see Lemma~\ref{qr: lem: if Lambda is in sublinear of Gamma, then Lambda preserves wqg rational tits building}). This is the case in the motivating setting where  $\La$ is an abstract finitely generated group that is SBE to $\G$ (Theorem~\ref{intro: thm: from SBE to sublinearly close}). In discussions with Omri Solan, he suggested that in the case of  $\operatorname{SL}_2(\mathbb{R})$, the action of $\La$ on the Bruhat-Tits tree of $\operatorname{SL}_2(\mathbb{Q}_p)$ (for all primes $p$) and the classification to elliptic and hyperbolic elements (for each $p$ separately) might allow one to deduce that $\La$ actually lies in $\operatorname{SL}_2(\mathbb{Z})$. We did not pursue that path nor its possible generalization to the $\operatorname{SL}_n$ case and general Bruhat-Tits buildings.
\end{rmk}

\subsubsection{A \texorpdfstring{$\G$-Orbit}{} Point Lying Deep Inside a \texorpdfstring{$\La$}{}-Free Horoball} \label{sec: Gamma points deep in Lambda free horoballs}

In some sense, proving that $\La$ is a lattice is proving that $\La x_0$ is `almost everywhere'. It may seem odd that the first step was to find large $\La$-free regions (horoballs). But this fits perfectly well with the algebraic statement that non-uniform lattices must admit unipotent elements (see Proposition $5.3.1$ in \cite{IntroArithmeticDAVEWITTEMORRIS}). The goal of this section is to obtain some control on the location of the $\La$-free horoballs, in order to conclude that some $\g x_0$ lies deep inside $\hbl$. This results in yet more $\La$-free regions, found on the bounding horosphere.

I thank Panos Papazoglou for noticing a mistake in the original formulation.

\begin{lem} \label{qr: lem: property 1 of sublinear functions}
Let $u$ be a sublinear function, $f,g:\mrnn\longrightarrow\mrp$ two positive monotone functions with $\lim_{x\rightarrow \infty}f(x)+g(x)=\infty$. If for all large enough $x$ it holds that $f(x)\leq u\big(f(x)+g(x)\big)$, then for every $1<s$ and all large enough $x$ it holds that $f(x)\leq u\big(s g(x)\big)$. In particular  $f(x)\leq u'\big(g(x)\big)$ for some sublinear function $u'\in O(u)$.
\end{lem}

\begin{proof}
Assume as one may that $u$ is non-decreasing, so $\lim_{x\rightarrow\infty}\frac{u\big(f(x)+g(x)\big)}{f(x)+g(x)}=0$ implies $\lim_{x\rightarrow\infty}\frac{f(x)}{f(x)+g(x)}=0$. This means that for every $\varepsilon>0$ one has  $f(x)\leq \varepsilon g(x)$ for all large enough $x$, resulting in 
$$f(x)\leq u\big(f(x)+g(x)\big)\leq u\big((1+\varepsilon)\cdot g(x)\big).$$

Notice that for a fixed $s>0$, the function $u'(x)=u(sx)$ is sublinear, as $$\lim_{x\rightarrow \infty}\frac{u(sx)}{x}=\lim_{x\rightarrow \infty} s \frac{u(sx)}{sx}=0$$ 

\end{proof}

\begin{lem} \label{qr: lem: distance bound between hbl and hbg near x_0}
Let $C>0$. There is $L=L(C)$ such that if $\hbl$ is any $\La$-free horoball tangent to a point $x\in B(x_0,C)$ then $d(\hl,\hg)\leq L$. Moreover, there is a sublinear function $u'\in O(u)$ such that:

 \begin{equation*}
L(C) \leq 
    \begin{cases}
        u'(C) & \text{if } \hbg\subset \hbl \\
        C & \text{if } \hbl\subset \hbg
    \end{cases}
\end{equation*}

\end{lem}

\begin{proof}
If $\hbl\subset\hbg$, then clearly $d(\hl,\hg)\leq C$, simply because $\hbg$ is $\G$-free and in particular cannot contain $x_0$. Therefore $\hg$ must separate $\hl$ from $x_0$ and in particular $d(\hl,\hg)\leq C$.

Assume that $\hbg\subset\hbl$, and denote $l=d(\hl,\hg)$. The horoball $\hbg$ is a horoball of $\G$, hence $\G x_0$ is $D_\G$-cocompact along $\hg$ and there is an element $\g\in \G$ with $|\g|\leq C+l+D_\G$ and $\g x_0\in \hg$. Since $\hbl$ is $\La$-free one has 

$$l\leq d(\g x_0,\la_\g x_0)\leq u(|\g|)\leq u(C+l+D_\G)$$

This inequality can only occur for boundedly small $l$, say $l<L'(C)$, since $C$ and $D_\G$ are fixed. Consult figure~\ref{fig: bound on distance of hbl and hbg horoballs at x_0} for a geometric visualization of this situation. 

It remains to show that $L'(C)$ is indeed sublinear in $C$. First define $L(C)$ to be the minimal $L$ that bounds the distance $d(\hl,\hg)$ for all possible $\hbl$ tangent to a point $x\in B(x_0,C)$. This is indeed a minimum, since by Corollary~\ref{qr: prelims: sym spc: cor: every ball intersects finitely many horoballs of Gamma} there are only finitely many horoballs of $\G$ intersecting $B\big(x_0,C+L'(C)\big)$. For every $C$ there is thus a horoball $\hbg_C$ and an element $\g\in \G$ such that $\g x_0\in \hg$, $d(\hl_C,\hg_C)=L(C)$ and $|\g|\leq C+L(C)+D_\G$. The fact that $\hbl_C$ is $\La$-free implies 
$$L(C)=d(\hl_C,\hg_C)\leq u(|\g|)\leq u\big(C+L(C)+D_\G\big).$$

By Lemma~\ref{qr: lem: property 1 of sublinear functions}, $L(C)\leq u'(C)$ for some sublinear function $u'$.
\end{proof}

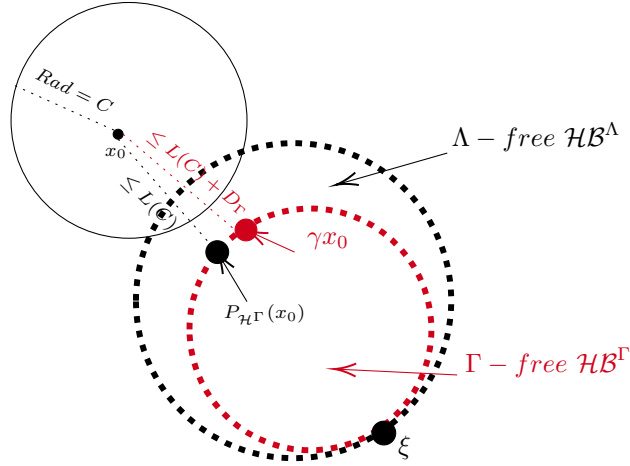
\begin{figure}
    \centering

\begin{tikzpicture}[x=0.75pt,y=0.75pt,yscale=-1,xscale=1]

\draw  [fill={rgb, 255:red, 0; green, 0; blue, 0 }  ,fill opacity=1 ] (120.5,74.5) .. controls (120.5,73.12) and (119.38,72) .. (118,72) .. controls (116.62,72) and (115.5,73.12) .. (115.5,74.5) .. controls (115.5,75.88) and (116.62,77) .. (118,77) .. controls (119.38,77) and (120.5,75.88) .. (120.5,74.5) -- cycle ;
\draw   (64,67.5) .. controls (64,34.64) and (90.64,8) .. (123.5,8) .. controls (156.36,8) and (183,34.64) .. (183,67.5) .. controls (183,100.36) and (156.36,127) .. (123.5,127) .. controls (90.64,127) and (64,100.36) .. (64,67.5) -- cycle ;
\draw  [dash pattern={on 0.84pt off 2.51pt}]  (66,50) -- (120.5,74.5) ;
\draw  [dash pattern={on 2.53pt off 3.02pt}][line width=2.25]  (127,158.5) .. controls (127,114.59) and (162.59,79) .. (206.5,79) .. controls (250.41,79) and (286,114.59) .. (286,158.5) .. controls (286,202.41) and (250.41,238) .. (206.5,238) .. controls (162.59,238) and (127,202.41) .. (127,158.5) -- cycle ;
\draw    (284,84) -- (229.92,99.45) ;
\draw [shift={(228,100)}, rotate = 344.05] [color={rgb, 255:red, 0; green, 0; blue, 0 }  ][line width=0.75]    (10.93,-3.29) .. controls (6.95,-1.4) and (3.31,-0.3) .. (0,0) .. controls (3.31,0.3) and (6.95,1.4) .. (10.93,3.29)   ;
\draw  [color={rgb, 255:red, 208; green, 2; blue, 27 }  ,draw opacity=1 ][dash pattern={on 2.53pt off 3.02pt}][line width=2.25]  (154,173) .. controls (154,139.31) and (181.31,112) .. (215,112) .. controls (248.69,112) and (276,139.31) .. (276,173) .. controls (276,206.69) and (248.69,234) .. (215,234) .. controls (181.31,234) and (154,206.69) .. (154,173) -- cycle ;
\draw  [fill={rgb, 255:red, 0; green, 0; blue, 0 }  ,fill opacity=1 ] (258,225.33) .. controls (258,228.65) and (255.31,231.33) .. (252,231.33) .. controls (248.69,231.33) and (246,228.65) .. (246,225.33) .. controls (246,222.02) and (248.69,219.33) .. (252,219.33) .. controls (255.31,219.33) and (258,222.02) .. (258,225.33) -- cycle ;
\draw [color={rgb, 255:red, 208; green, 2; blue, 27 }  ,draw opacity=1 ]   (289,196) -- (229,193.1) ;
\draw [shift={(227,193)}, rotate = 2.77] [color={rgb, 255:red, 208; green, 2; blue, 27 }  ,draw opacity=1 ][line width=0.75]    (10.93,-3.29) .. controls (6.95,-1.4) and (3.31,-0.3) .. (0,0) .. controls (3.31,0.3) and (6.95,1.4) .. (10.93,3.29)   ;
\draw  [dash pattern={on 0.84pt off 2.51pt}]  (118,77) -- (168,134) ;
\draw  [color={rgb, 255:red, 208; green, 2; blue, 27 }  ,draw opacity=1 ][fill={rgb, 255:red, 208; green, 2; blue, 27 }  ,fill opacity=1 ] (188,122.83) .. controls (188,119.8) and (185.54,117.33) .. (182.5,117.33) .. controls (179.46,117.33) and (177,119.8) .. (177,122.83) .. controls (177,125.87) and (179.46,128.33) .. (182.5,128.33) .. controls (185.54,128.33) and (188,125.87) .. (188,122.83) -- cycle ;
\draw  [fill={rgb, 255:red, 0; green, 0; blue, 0 }  ,fill opacity=1 ] (173.75,134) .. controls (173.75,130.82) and (171.18,128.25) .. (168,128.25) .. controls (164.82,128.25) and (162.25,130.82) .. (162.25,134) .. controls (162.25,137.18) and (164.82,139.75) .. (168,139.75) .. controls (171.18,139.75) and (173.75,137.18) .. (173.75,134) -- cycle ;
\draw [color={rgb, 255:red, 208; green, 2; blue, 27 }  ,draw opacity=1 ]   (207,134) -- (184.32,123.66) ;
\draw [shift={(182.5,122.83)}, rotate = 24.5] [color={rgb, 255:red, 208; green, 2; blue, 27 }  ,draw opacity=1 ][line width=0.75]    (10.93,-3.29) .. controls (6.95,-1.4) and (3.31,-0.3) .. (0,0) .. controls (3.31,0.3) and (6.95,1.4) .. (10.93,3.29)   ;
\draw    (183,159) -- (169.03,135.71) ;
\draw [shift={(168,134)}, rotate = 59.04] [color={rgb, 255:red, 0; green, 0; blue, 0 }  ][line width=0.75]    (10.93,-3.29) .. controls (6.95,-1.4) and (3.31,-0.3) .. (0,0) .. controls (3.31,0.3) and (6.95,1.4) .. (10.93,3.29)   ;
\draw [color={rgb, 255:red, 208; green, 2; blue, 27 }  ,draw opacity=1 ] [dash pattern={on 0.84pt off 2.51pt}]  (120.5,74.5) -- (177,122.83) ;

\draw (77.93,38.48) node [anchor=north west][inner sep=0.75pt]  [font=\scriptsize,rotate=-25.83]  {$Rad=C$};
\draw (110,80) node [anchor=north west][inner sep=0.75pt]  [font=\scriptsize,rotate=0]  {$x_0$};
\draw (285,67) node [anchor=north west][inner sep=0.75pt]    {$\Lambda -free\ \mathcal{HB}^{\Lambda }$};
\draw (258,225.33) node [anchor=north west][inner sep=0.75pt]    {$\xi $};
\draw (292,181) node [anchor=north west][inner sep=0.75pt]  [color={rgb, 255:red, 208; green, 2; blue, 27 }  ,opacity=1 ]  {$\Gamma -free\ \mathcal{HB}^{\Gamma }$};
\draw (124.5,91.5) node [anchor=north west][inner sep=0.75pt]  [font=\scriptsize,rotate=-44.86]  {$\leq L( C)$};
\draw (212,123) node [anchor=north west][inner sep=0.75pt]  [color={rgb, 255:red, 208; green, 2; blue, 27 }  ,opacity=1 ]  {$\gamma x_{0}$};
\draw (170,159) node [anchor=north west][inner sep=0.75pt]  [font=\scriptsize]  {$P_{\mathcal{H}^{\Gamma }}( x_{0})$};
\draw (136.69,68.46) node [anchor=north west][inner sep=0.75pt]  [font=\scriptsize,color={rgb, 255:red, 208; green, 2; blue, 27 }  ,opacity=1 ,rotate=-39.32]  {$\leq L( C) +D_\G$};

\end{tikzpicture}

    \caption[Bound on the distance of horoballs near $x_0$]{Lemma~\ref{qr: lem: distance bound between hbl and hbg near x_0}. A $\La$-free horoball $\hbl$ intersects a ball of radius $C$ about $x_0$. The associated $\G$-free horoball $\hbg$ is boundedly close, essentially due to the uniform cocompactness of $\G x_0$ along the $\G$ horospheres.}
    \label{fig: bound on distance of hbl and hbg horoballs at x_0}
\end{figure}

\begin{cor}\label{qr: cor: finitely many lambda horoballs at x_0}
For every $C>0$ there is a bound $K=K(C)$ and a fixed set $\xi_1,\xi_2,\dots \xi_K\in \wqg\subset X(\infty)$ so that every $\La$-free horoball $\hbl$ which is tangent to some point $x\in B(x_0,C)$ is based at $\xi_i$ for some $i\in\{1,\dots,K\}$. In particular, for any specific point $x\in \mathcal{N}_C(\La x_0)$ there are at most $K$ $\Lambda$-free horoballs tangent to $x$.
\end{cor}

\begin{proof}
Let $\hbl$ be a horoball tangent to a point $x\in B(x_0,C)$. Lemma~\ref{qr: lem: distance bound between hbl and hbg near x_0} bounds $d(\hbl,\hbg)$ by $L(C)$, hence $\hbg$ is tangent to a point $x'\in B\big(x_0,C+L(C)\big)$. By Corollary~\ref{qr: prelims: sym spc: cor: every ball intersects finitely many horoballs of Gamma} there are only finitely many possibilities for such $\hbg$. In particular there are finitely many base-points for these horoballs, say $\xi_1,\xi_2,\dots,\xi_{K(C)}\in \wqg$. Finally, recall that a horoball is determined by a base point and a point $x\in X$ tangent to it, so the last statement of the corollary holds for any $x\in B(x_0,C)$. But the property in question is $\La$-invariant so the same holds for any point $x\in \La B(x_0,C)=\mn_C(\La x_0)$.
\end{proof}


Lemma~\ref{qr: lem: distance bound between hbl and hbg near x_0} bounds the distance between $\hbl$ and $\hbg$. The next lemma relates $\g x_0$ and a $\La$-free horoball $\hbl$ tangent to $\la_\g x_0$.

\begin{defn}
In light of Corollary~\ref{qr: cor: finitely many lambda horoballs at x_0}, there is a finite number $N$ of $\La$-free horoballs tangent to $x_0$. Denote:
\begin{enumerate}
    \item $\{\hbl_i\}_{1=1}^N$ are the $\Lambda$-free horoballs tangent to $x_0$. 
    
    \item $\xi^i\in\wqg$ is the base point of $\hbl_i$.
    
    \item $v^i\in S_{x_0}X$ is the unit tangent vector in the direction $\xi^i$.
    
    \item $\eta^i:=[x_0,\xi_i)$ is the unit speed geodesic ray emanating from $x_0$ with limit $\xi^i$. In particular  $v^i=\dot{\eta^i}(0)$.
    
    \item $\hbg_i$ is the horoball of $\G$ that corresponds to $\hbl_i$, based at $\xi^i$.
    
    \item $\hbl_{\la,i}, \xi^i_\la,\eta^i_\la$ are the respective $\lambda$-translates of the objects above. For example,  $\hbl_{\la,i}:=\lambda\cdot \hbl_i$.
    
    \item $\h$ decorated by the proper indices denotes the horosphere bounding $\hb$, the horoball with respective indices, e.g.\ $\hl_i:=\partial \hbl_i$. 
    \item For an angle $\alpha>0$ and a tangent vector $v_0\in S_xX$, define 
    \begin{enumerate}
        \item The \emph{$\alpha$-sector of $v$ in $S_xX$} is the set $\{v\in S_xX\mid v\in \mn_\alpha(v_0)\}$. Recall that the metric on $S_xX$ is the angular metric. 
        \item The \emph{$\alpha$-sector of $v$ in $X$} are all points $y\in X$ for which the tangent vector at $0$ of the unit speed geodesic $[x,y]$ lies in the $\alpha$-sector of $v$ in $SxX$.
    \end{enumerate}
    
\end{enumerate}\end{defn}

\begin{lem}\label{qr: lem: large d_g impliy small angle between [la_g,g] and hbl}
For every angle $\alpha\in(0,\fpi2)$ there exists $D=D(\alpha)$ such that if $d_\g>D$ then for some $i\in\{1,\dots,N\}$, $\g x_0$ lies inside the $\alpha$-sector of $v^i_{\la_\g}$ at 
$\la_\g x_0$. Furthermore whenever $\alpha$ is uniformly small enough, there is a unique such $i=i(\g)$, independent of $\alpha$.
\end{lem}

\begin{proof}
It is enough to prove the claim for $x'_\g:=\la_\g^{-1}\g x_0$, since translating by the isometry $\lambda_\g^{-1}$ preserves angles and distances. Assume towards contradiction that there is $\alpha\in(0,\fpi2)$ and a sequence $\g_n\in \G$, $\la_n:=\la_{\g_n}\in \La$ with $d_{\g_n}$ unbounded, and $x'_n:=\la_n^{-1}\g_n x_0$ not lying in the union of the $\alpha$-sectors of $v^i$. There are only finitely many $v^i$, so w.l.o.g I may assume the $\alpha$-sectors of all $v^i$ in $S_{x_0}X$ are 
pairwise disjoint. 

Let $\eta_n:=[x_0,x_n']$. Compactness of $S_{x_0}X$ gives rise to a converging subsequence $v'_{n}:=\dot{\eta_n}(0)$, with limit direction $v'$. Denote by $\eta'$ the geodesic ray emanating from $x_0$ with initial velocity $v'$. The exact same argument of Lemma~\ref{qr: lem: existence of Lambda free horoballs} proves that $\eta'(\infty)$ is the base point of a $\La$-free horoball tangent to $x_0$. But this means $v'=v^i$ for some $i\in\{1,\dots,N\}$, contradicting the fact that all $v'_n$ lie outside the $\alpha$-sectors of the $v^i$. This proves that there is a bound $D=D(\alpha)$ such that if $d_\g>D$ then $x_\g'$ lies within the $\alpha$-sector of some $v^i$. Taking $\alpha$ small enough so that the $\alpha$-sectors of the $v^i$ are pairwise disjoint forces the uniqueness of $v^i$. 

\end{proof}

\begin{rmk}
In the proof of Lemma~\ref{qr: lem: existence of Lambda free horoballs} I used compactness of $S_{x_0}X$ to induce a converging subsequence of directions. Lemma~\ref{qr: lem: large d_g impliy small angle between [la_g,g] and hbl} actually shows that the fact there are finitely many $\La$-free horoballs tangent to $x_0$ implies \emph{a posteriori} that there was not much choice in the process - all directions $[x_0,x_\g']$ must fall into one of the finitely many directions $v^i$.

\end{rmk}

Next, I need to gain control not just on angles but on the actual location of certain points with respect to the $\hbl$-s, which is much more delicate. Recall that large $\La$-free balls near $x_0$ imply large concentric $\G$-free balls. The precise quantities and bounds are given by Lemma~\ref{U: lem: large Lambda free balls near base point imply large Gamma free balls} (one can use Lemma~\ref{U: lem: sublinear version of large Lambda free balls near base point imply large Gamma free balls} to obtain a slightly cleaner statement).

\begin{prop} \label{qr: prop: gamma lies deep inside the la ga horoball}
Let $S\in (0,1)$ be the constant given by Lemma~\ref{U: lem: large Lambda free balls near base point imply large Gamma free balls}, and let $s\in (0,S)$. There is a bound $D=D(s)$ such that $d_\g>D$ implies that $\g x_0$ lies $sd_\g$ deep in $\hbl_{\la_\g}$.
\end{prop}

\begin{proof}
The idea is to use $\G$-free balls near $x_0$ to produce a $\G$-free cylinder, which forces that cylinder to stay inside a horoball of $
\G$. As above it is enough to prove that $x'=\lambda_\g^{-1}\g x_0$ is $sd_\g$ deep inside $\hbl_{i(\g)}$. I start with proving that $x'\in\hbg_{i(\g)}$ (note that even this is not a triviality). For $B_\g=B(\gamma x_0,d_\g)$, the ball $B'_\g=\la_\g^{-1}B_\g$ is a $\Lambda$-free ball of radius $d_\g$ about $x'=\la_\g^{-1}\g x_0$. Denote by $x'_t$ the point at time $t$ along the unit speed geodesic $\eta':=[x_0,x']$. It holds that $|x_t'|=t$ and, for $t\leq d_\g$, $x_t'$ is the centre of a $\La$-free ball of radius $t$ tangent to $x_0$. The constant $s$ is fixed and by Lemma~\ref{U: lem: large Lambda free balls near base point imply large Gamma free balls} there is $T'=T'(s)$ such that if $t>T'$, the ball $sB\big(x'_t,t\big)$ is $\G$-free. 

The next goal is to show that $x_T'\in \hbg$ for some adequate $T$. For any time $T>0$, let $\alpha=\alpha(\varepsilon,T)$ be an angle for which $d\big(\eta(T),\eta^{i(\g)}(T)\big)<\varepsilon$ for every $\eta$ in the $\alpha$-sector of $v^{i(\g)}$. By perhaps taking smaller $\alpha$ I may assume that $\alpha$ is uniformly small as stated in Lemma~\ref{qr: lem: large d_g impliy small angle between [la_g,g] and hbl}. Let $D(\alpha)$ be the bound given by Lemma~\ref{qr: lem: large d_g impliy small angle between [la_g,g] and hbl} guaranteeing, since $x'$ and $x'_T$ lie on the same geodesic segment, that

$$D(\alpha)<d_\g \Rightarrow d\big(x'_T,\eta^{i(\g)}(T)\big)<\varepsilon.$$

For my needs in this lemma $\varepsilon$ may as well be chosen to be $1$. I now choose a specific time $T$ for which I want $x_T'$ and $\eta^{i(\g)(T)}$ to be close. There are only finitely many $\Lambda$-free horoballs $\{\hbl_i\}_{i\in\N}$ tangent to $x_0$, giving rise to a uniform bound $L=\max_{i\in\N} \{d(\hl_i,\hg_i)\}$ on the distance $d(\hl_{i(\g)},\hg_{i(\g)})$. Fix $T$ to be any time in the open interval $(T'+L+\varepsilon,d_\g)$. The fact that $L+\varepsilon<T$ implies that $\eta^{i(\g)}(T)$ lies at least $\varepsilon$-deep inside $\hbg_{i(\g)}$, and therefore $\eta'(T)\in \hbg_{i(\g)}$. Recall that any point on $\hg$ is $D_\G$-close to a point $\g x_0\in \hg$. By perhaps enlarging $T$ and shrinking $\alpha$ if necessary, I may assume that $D_\G<sT$. Thus for all $T<t\leq d_\gamma$, $x_t'$ is the centre of a $\Gamma$-free ball of radius $st>sT>D_\G$, hence $\{x'_t\}_{T\leq t\leq d_\g}$ does not cross a horosphere of $\Gamma$. Since $x_T'\in \hbg_{i(\g)}$, this implies that $x_t'$ stays in $\hbg_{i(\g)}$ for all $T<t\leq d_\g$. In particular $x'_{d_\g}=x'\in\hbg_{i(\g)}$.

To get the result of the proposition, recall that $sB_\g'=B(x',sd_\g)$ is $\G$-free, so $x'$ must be at distance at least $sd_\g-D_\G$ from any horosphere of $\Gamma$, and in particular from $\hg_{i(\g)}$. In terms of Busemann functions, this means that $b_{\eta^{i(\g)}}(x')\leq -sd_\g+D_\G$ whenever one can find such $T'+L+\varepsilon<T<d_\g$. Since $\hbl_{i(\g)}$ is tangent to $x_0$, the corresponding horoball $\hbg_{i(\g)}$ lies inside it, and so $x'$ lies $(sd_\g-D_\G)$-deep inside $\hbl_{i(\g)}$. A close look at the argument yields the desired bound $D=D(s)$ such that the above holds whenever $d_\g>D(s)$. I reiterate the choice of constants and their dependencies as they appear in the proof: 
\begin{enumerate}
    \item Fix $\varepsilon=1$.
    
    \item Let $T'=T'(s)$ the constant from Lemma~\ref{U: lem: large Lambda free balls near base point imply large Gamma free balls} and $L=\max_{i\in\{1,\dots,N\}}\{d(\hbl_i,\hbg_i)\}$.
    
    \item Fix $T>T'+L+1$.
    
    \item Fix $\alpha=\alpha(1,T)$.
    
    \item Fix $D(s)=\max\{D(\alpha),T+1\}$. 
\end{enumerate}

One should not worry about $D_\G$, which appears in the final bound but not in the statement, since (a) $D_\G$ is a fixed universal constant and may as well be ignored, and (b) the discrepancy can be formally corrected by taking a slightly larger $s<s'$ to begin with and as a result perhaps enlarging the bound $D$ for $d_\g$. Also note that $L=L(\La)$ is a universal constant. 

\end{proof}

\subsubsection{Intersection of \texorpdfstring{$\La$}{}-Free Regions and the Existence of a \texorpdfstring{$\La$}{}-Cocompact Horosphere}    \label{sec: cocompact horosphere}

In this section I find $\La$-orbit points that lie close to the bounding horosphere of a $\La$-free horoball $\hbl$. In order to find such points I need to make sure $\hbl$ is not contained inside a much larger $\La$-free horoball. 

\begin{defn}\label{qr: def: almost maximal Lambda free horoball}
A $\La$-free horoball $\hbl$ is called \emph{maximal} if it is tangent to a point $x=\lambda x_0\in\La x_0$. It is called \emph{$\varepsilon$-almost maximal} if $d(\La\cdot x_0,\hl)<\varepsilon$.
\end{defn}

\begin{rmk}\label{qr: rmk: blow up to maximal lambda free horoball}
It may happen that a discrete group admits free but not \emph{maximally free} horoballs - see discussion in section $4$ of \cite{SullivanGarnnet}. In any case, it is clear that for every $\varepsilon>0$, any $\La$-free horoball can be `blown-up' to an $\varepsilon$-almost maximal $\La$-free horoball, and that every two  parallel $\varepsilon$-almost maximal horoballs are at distance at most $\varepsilon$. For my needs any fixed $\varepsilon$ would suffice, and I fix $\varepsilon=1$. 
\end{rmk}

\begin{lem} \label{qr: lem: almost Lambda cocompact horospheres}
There is $D_\La>0$ such that if $\hbl$ is $1$-almost maximal $\La$-free horoball then $\hl\subset \mn_{D_\La}(\La x_0)$, i.e.\ $d(x,\La x_0)\leq D_\La$ for all $x\in\hl$.
\end{lem}

Notice that Lemma~\ref{qr: lem: almost Lambda cocompact horospheres} does not state $\La x_0$ even intersects $\hl$.

\begin{proof}
Consider a $1$-maximal horoball and a point $x$ on its bounding horosphere with $d(x,\La x_0)=D$. One may translate this situation to $x_0$, which results in a $\La$-free horoball $\hbl$ intersecting the (closed) $D$-ball about $x_0$ at a point $w$ with $B(w,D)$ $\La$-free. 

Building on Lemma~\ref{qr: lem: distance bound between hbl and hbg near x_0}, the proof differs depending on whether $\hbg\subset \hbl$ or the other way round:
\begin{enumerate}
    \item If $\hbg\subset\hbl$, there is a sublinear bound on $d(\hbl,\hbg)$, which readily yields a bound on $D$. 
    
    \item if $\hbl\subset \hbg$ there is a bound on $d(x_0,\hbg)$ that is independent of $D$. So there are only finitely many possibilities for $\hbg$, independent of $D$. Hence there are only finitely many possible base points for $\hbg$. These in turn correspond to possible base points for such $\hbl$, and this finiteness yields a bound on the distance $d(\hbg,\hbl)<L$ that is independent of $D$. The rest of the proof is quite routine. 
\end{enumerate}

Concretely, Let $\hbl$ be a $1$-almost maximal $\Lambda$-free horoball. By definition there is $\la\in\La$ and $z\in\hl$ such that $d(\la x_0,z)<1$. I show that if there is some $z'\in\hl$ for which $d(z',\La x_0)\geq D$, then $D$ must be uniformly small. Exactly how small will be set in the course of the proof. 

Fix $D>1$ and assume that there is $z'\in\hl$ with $d(z,\La x_0)\geq D$. Up to sliding $z'$ along $\hl$, the continuity of the function $x\mapsto d(x,\La x_0)$ and Intermediate Value Theorem allows one to assume that $d(z',\La x_0)=D$. Let $\la'\in\La$ be the element for which $d(z',\la' x_0)=D$. Translating by $\la'^{-1}$ yields

\begin{enumerate}
    \item A $\La$-free horoball $\hbl_0:=\la'^{-1}\hbl$.
    \item A point $w:=\la'^{-1}z'\in\hl_0$ for which  $|w|=d(w,x_0)=d(w,\La x_0)=D$.
\end{enumerate}

Assume first that $\hbg_0\subset\hbl_0$. By Lemma~\ref{qr: lem: distance bound between hbl and hbg near x_0} there is a sublinear function $u'$ such that $d(\hg_0,\hl_0)\leq u'(D)$. This yields a point $\g x_0\in \hg_0$ for which $d(w,\g x_0)\leq u'(D)+D_\G$. Thus $|\g x_0|\leq D+u'(D)+D_\G$ and the reverse triangle inequality gives

$$D-\big(u'(D)+D_\G\big)\leq d(w,\la_\g x_0)-d(w,\g x_0)<d(\g x_0,\la_\g x_0).$$

Together with the bound $d(\g x_0,\la_\g x_0)\leq u(|\g x_0|)$ and rearranging, one obtains

$$D\leq u\big(D+u'(D)+D_\G\big)+u'(D)+D_\G.$$

The right hand side is clearly a sublinear function in $D$, hence this inequality may hold only for boundedly small $D$, say $D<D_1$. I conclude that $\hbg_0\subset\hbl_0$ may occur only when $D<D_1$. Notice that $D_1$ depends only on $u$ and $u'$, and not on $\hbl$.

Assume next that $\hbl_0\subset\hbg_0$, and that the containment is strict. Since $x_0\in\G x_0$, the geodesic $\tau:=[x_0,w]$ is of length $D$ and intersects $\hg_0$. Denote by $t_0\in[0,D)$ the time in which $\tau$ intersects $\hg_0$, and let $w':=\tau(t_0)\in\hg_0$ be the intersection point. In particular $|w'|=t_0$. It is clear that $B(w',t_0)$ is $\Lambda$-free, as a subset of the ball $B(w,D)$. Again there is $\g x_0\in B(w',D_\G)\cap\hg_0$ and so $|\g x_0|\leq t_0+D_\G$. By reverse triangle inequality 

$$t_0-D_\G\leq d(w',\la_\g x_0)-d(w',\g x_0)\leq d(\g x_0,\la_\g x_0)$$

and the sublinear constraint gives $t_0-D_\G\leq u(t_0+D_\G)$. This can only happen for boundedly small $t_0$, say $t_0<T$. I conclude that if $\hbl_0\subset \hbg_0$, then $\hbg_0$ is a horoball of $\G$ tangent to some point $y\in B(x_0,T)$. 

By Corollary~\ref{qr: prelims: sym spc: cor: every ball intersects finitely many horoballs of Gamma} there are finitely many horoballs of $\Gamma$ tangent to points in $B(x_0,T)$. In particular there is a finite set $\{\xi_1',\dots,\xi_K'\}\in\wqg$ of possible base points for $\hbg_0$. This set depends only on $T$, and since the choice of $T$ was completely independent of $D$, the set of possible base points is independent of $D$ as well. Let $\widetilde{\hbg_i}$ be the horoball of $\G$ based at $\xi_i'$. 

I can now bound the distance $d(\hbg_0,\hbl_0)$. Let $1\leq i\leq K$ be an index for which there is a $\La$-free horoball based at $\xi_i'$ that is contained in $\widetilde{\hbg_i}$. There is thus some $1$-almost-maximal $\La$-free horoball based at $\xi'_i$. Fix an arbitrary such $1$-almost-maximal $\La$-free horoball $\widetilde{\hbl_i}$ for each such $i$, and let $L_i:=d(\widetilde{\hbl_i},\widetilde{\hbg_i})$. Finally, define $L:=\max\{L_i\}+1$ among such $i$. As stated in Remark~\ref{qr: rmk: blow up to maximal lambda free horoball}, $d(\hbl_0, \widetilde{\hbl_i})\leq 1$ for some $i$, therefore $d(\hbg_0,\hbl_0)\leq L$.

Recall $|w|=D$ and $B(w,D)$ is $\Lambda$-free. It holds that $d(w,\hg_0)\leq L$, and so there is $\g x_0\in \hg_0$ for which $d(w,\g x_0)\leq L+D_\G$, so $|\g x_0|\leq D+L+D_\G$. Reverse triangle inequality gives

$$D-(L+D_\G)\leq d(w,\la_\g x_0)-d(w,\g x_0)\leq d(\g x_0,\la_\g x_0)$$

and from the sublinear constraint I conclude $D-(L+D_\G)\leq u(D+L+D_\G)$. Since $L,D_\G$ are fixed constants independent of $D$, this can only hold for boundedly small $D$, say $D<D_2$. In particular, one gets a uniform bound $D_\La:=\max\{D_1,D_2\}$ such that $x\in\hl\Rightarrow d(x,\La x_0)<D_\La$.

\end{proof}

\begin{figure}
    \centering

\tikzset{every picture/.style={line width=0.75pt}} 

\begin{tikzpicture}[x=0.75pt,y=0.75pt,yscale=-1,xscale=1]

\draw  [fill={rgb, 255:red, 0; green, 0; blue, 0 }  ,fill opacity=1 ] (318.21,229.11) .. controls (318.21,226.38) and (316.1,224.16) .. (313.5,224.16) .. controls (310.89,224.16) and (308.78,226.38) .. (308.78,229.11) .. controls (308.78,231.85) and (310.89,234.07) .. (313.5,234.07) .. controls (316.1,234.07) and (318.21,231.85) .. (318.21,229.11) -- cycle ;
\draw  [dash pattern={on 1.69pt off 2.76pt}][line width=1.5]  (178.94,151.31) .. controls (178.94,100.49) and (218.16,59.29) .. (266.54,59.29) .. controls (314.92,59.29) and (354.14,100.49) .. (354.14,151.31) .. controls (354.14,202.13) and (314.92,243.33) .. (266.54,243.33) .. controls (218.16,243.33) and (178.94,202.13) .. (178.94,151.31) -- cycle ;
\draw  [fill={rgb, 255:red, 0; green, 0; blue, 0 }  ,fill opacity=1 ] (144.87,46.63) .. controls (144.87,43.89) and (142.76,41.67) .. (140.16,41.67) .. controls (137.55,41.67) and (135.44,43.89) .. (135.44,46.63) .. controls (135.44,49.36) and (137.55,51.58) .. (140.16,51.58) .. controls (142.76,51.58) and (144.87,49.36) .. (144.87,46.63) -- cycle ;
\draw  [color={rgb, 255:red, 208; green, 2; blue, 27 }  ,draw opacity=1 ][dash pattern={on 1.69pt off 2.76pt}][line width=1.5]  (153.95,135.33) .. controls (153.95,74.98) and (200.53,26.06) .. (257.98,26.06) .. controls (315.44,26.06) and (362.01,74.98) .. (362.01,135.33) .. controls (362.01,195.69) and (315.44,244.61) .. (257.98,244.61) .. controls (200.53,244.61) and (153.95,195.69) .. (153.95,135.33) -- cycle ;
\draw   (129.47,88.68) .. controls (129.47,46.33) and (162.15,12) .. (202.47,12) .. controls (242.79,12) and (275.47,46.33) .. (275.47,88.68) .. controls (275.47,131.04) and (242.79,165.37) .. (202.47,165.37) .. controls (162.15,165.37) and (129.47,131.04) .. (129.47,88.68) -- cycle ;
\draw  [fill={rgb, 255:red, 0; green, 0; blue, 0 }  ,fill opacity=1 ] (207.19,88.68) .. controls (207.19,85.95) and (205.07,83.73) .. (202.47,83.73) .. controls (199.87,83.73) and (197.76,85.95) .. (197.76,88.68) .. controls (197.76,91.42) and (199.87,93.64) .. (202.47,93.64) .. controls (205.07,93.64) and (207.19,91.42) .. (207.19,88.68) -- cycle ;
\draw    (88.25,54.18) -- (140.02,81.36) ;
\draw [shift={(141.79,82.29)}, rotate = 207.71] [color={rgb, 255:red, 0; green, 0; blue, 0 }  ][line width=0.75]    (10.93,-3.29) .. controls (6.95,-1.4) and (3.31,-0.3) .. (0,0) .. controls (3.31,0.3) and (6.95,1.4) .. (10.93,3.29)   ;
\draw  [color={rgb, 255:red, 208; green, 2; blue, 27 }  ,draw opacity=1 ][fill={rgb, 255:red, 208; green, 2; blue, 27 }  ,fill opacity=1 ] (189.24,56.57) .. controls (189.24,53.84) and (187.13,51.62) .. (184.52,51.62) .. controls (181.92,51.62) and (179.81,53.84) .. (179.81,56.57) .. controls (179.81,59.31) and (181.92,61.53) .. (184.52,61.53) .. controls (187.13,61.53) and (189.24,59.31) .. (189.24,56.57) -- cycle ;
\draw  [fill={rgb, 255:red, 0; green, 0; blue, 0 }  ,fill opacity=1 ] (179.24,70.16) .. controls (179.24,67.42) and (177.13,65.2) .. (174.53,65.2) .. controls (171.92,65.2) and (169.81,67.42) .. (169.81,70.16) .. controls (169.81,72.89) and (171.92,75.11) .. (174.53,75.11) .. controls (177.13,75.11) and (179.24,72.89) .. (179.24,70.16) -- cycle ;
\draw    (385.13,155.14) -- (301.96,153.9) ;
\draw [shift={(299.96,153.87)}, rotate = 0.86] [color={rgb, 255:red, 0; green, 0; blue, 0 }  ][line width=0.75]    (10.93,-3.29) .. controls (6.95,-1.4) and (3.31,-0.3) .. (0,0) .. controls (3.31,0.3) and (6.95,1.4) .. (10.93,3.29)   ;
\draw [color={rgb, 255:red, 208; green, 2; blue, 27 }  ,draw opacity=1 ]   (355.93,47.79) -- (289.79,47.79) ;
\draw [shift={(287.79,47.79)}, rotate = 360] [color={rgb, 255:red, 208; green, 2; blue, 27 }  ,draw opacity=1 ][line width=0.75]    (10.93,-3.29) .. controls (6.95,-1.4) and (3.31,-0.3) .. (0,0) .. controls (3.31,0.3) and (6.95,1.4) .. (10.93,3.29)   ;
\draw  [dash pattern={on 0.84pt off 2.51pt}]  (140.16,46.63) -- (202.47,88.68) ;

\draw (207.55,85.52) node [anchor=north west][inner sep=0.75pt]  [font=\large]  {$w$};
\draw (122.84,27.73) node [anchor=north west][inner sep=0.75pt]  [font=\small]  {$x_{0}$};
\draw (2.83,37.15) node [anchor=north west][inner sep=0.75pt]    {$ \begin{array}{l}
\Lambda -free\\
\ B( w,D)
\end{array}$};
\draw (157.05,72.57) node [anchor=north west][inner sep=0.75pt]  [font=\normalsize]  {$\tau ( t_{0})$};
\draw (152.77,46.85) node [anchor=north west][inner sep=0.75pt]    {$\tau $};
\draw (408.65,140.87) node [anchor=north west][inner sep=0.75pt]    {$\Lambda -free\ \mathcal{HB}^{\Lambda }{}$};
\draw (319.5,231.11) node [anchor=north west][inner sep=0.75pt]    {$\xi $};
\draw (373.93,38.62) node [anchor=north west][inner sep=0.75pt]  [color={rgb, 255:red, 208; green, 2; blue, 27 }  ,opacity=1 ]  {$\Gamma -free\ \mathcal{HB}^{\Gamma }{}$};
\draw (188.71,51.01) node [anchor=north west][inner sep=0.75pt]  [color={rgb, 255:red, 208; green, 2; blue, 27 }  ,opacity=1 ]  {$\gamma x_{0}$};

\end{tikzpicture}

    \caption[A $\G$-free horoball trapped between $x_0$ and the corresponding $\La$-free horoball]{Lemma~\ref{qr: lem: almost Lambda cocompact horospheres}, case $\hbl\subset\hbg$. The red horosphere of $\G$ is trapped between $x_0$ and $\hl$, and is at distance $t_0$ from $x_0$. A $\G$-orbit point on the red horosphere close to $x_0$ allows one to use sublinearity to get a bound on $t_0$.}
    \label{fig: almost lambda cocompact horospheres}
\end{figure}
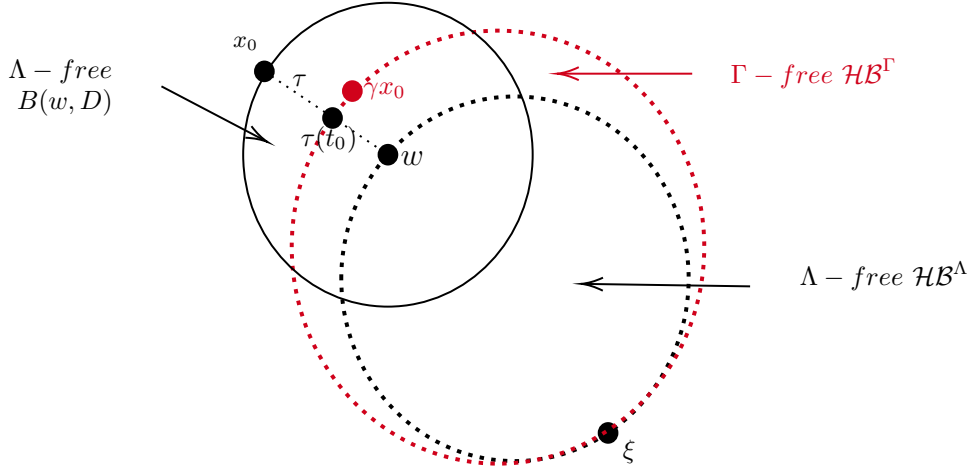

\begin{cor}\label{qr: cor: every Gamma conical is Lambda conical}
Every $\G$-conical limit point is a $\La$-conical limit point.
\end{cor}

\begin{proof}
Let $\xi\in X(\infty)$ be a $\G$-conical limit point. Let $\eta:\mrnn\rightarrow X$ be a geodesic with $\eta(\infty)=\xi$. By definition there is a bound $D>0$ and sequences $t_n\rightarrow \infty,\g_n\in \G$ such that $d\big(\g_n x_0,\eta(t_n)\big)<D$. Consider the corresponding $\la_n:=\la_{\g_n}$ and $\la_n x_0$. If $d_n$ is uniformly bounded, then $\xi$ is $\La$-conical by definition. Otherwise, assume $d_n$ is monotonically increasing to $\infty$. For some fixed $s\in (0,1)$ it holds that for all but finitely many $n\in\mathbb{N}$, $\g_n x_0$ is $sd_n$ deep inside $\hbl_n:=\hbl_{\la_n,i(\g_n)}$. I assume $d_n$ is large enough so that $sd_n>D$, and in particular $\eta(t_n)\in \hbl_n$. Let $\xi_n\in\wqg$ be the respective base points of $\hbl_n$. The point $\xi$ is $\G$-conical,  and by Theorem~\ref{qr: thm: HattoriB} $\frac{\pi}{2}\leq d(\xi,\wqg)\leq d(\xi,\xi_n)$.

The proof differs depending on whether the above inequality is strict or not for any $n\in\mathbb{N}$. 
Assume first that for some $m\in\mathbb{N}$,  $d(\xi,\xi_m)=\frac{\pi}{2}$. By item $2$ of Lemma~\ref{qr: prelims: sym spc: lem: Hattori penetration}, $d\big(\hl_m,\eta(t)\big)$ is uniformly bounded, i.e., there is $C>0$ such that for every $t>0$ there is $x_t\in \hl_m$ for which $d\big(x_t,\eta(t)\big)<C$. By Lemma~\ref{qr: lem: almost Lambda cocompact horospheres}, $d(x_t,\La x_0)<D_\La$, hence $d\big(\eta(t_n),x_{t_n}\big)\leq C+D_\La$. This means that $\xi$ is $\La$-conical. 

Otherwise, for all $n\in\mathbb{N}$ it holds that  $\frac{\pi}{2}<d(\xi,\xi_n)$. The fact that $\eta(t_n)\in \hbl_n$ together with Lemma~\ref{qr: prelims: sym spc: lem: Hattori penetration} implies that at some later time the geodesic ray $\eta$ leaves $\hbl_n$. Thus there is $s_n>t_n$ for which $\eta(s_n)\in \hl_n$. Since $\hbl_n$ are maximal $\La$-free horoballs, Lemma~\ref{qr: lem: almost Lambda cocompact horospheres} gives rise to points $\la_n x_0$ such that $d\big(\la_n x_0,\eta(s_n)\big)\leq D_\La$. This renders $\xi$ as a $\La$-conical limit point, as wanted. 

\end{proof}

\begin{proof} [Proof of Proposition~\ref{qr: prop: Lambda cocompact horosphreres}]

The strategy is as follows. For $\hbl=\hbl_{{\la_\g},i(\g)}$, one uses Proposition~\ref{qr: prop: gamma lies deep inside the la ga horoball} to get that $\hbg\subset \hbl$ and that the distance $d(\hl,\hg)$ is large with $d_\g$. The horosphere $\hg$ admits a $\G$-cocompact metric lattice, and so the projections of these metric lattice points onto $\hl$ form a cocompact metric lattice in $\hl$. It remains to show that for each $\g' x_0\in\Gamma x_0\cap \hg$, the corresponding $\la'=\la_{\g'}$ indeed lies on the same $\hl$ and boundedly close to the projection $P_{\hl}(\g' x_0)$. This is done by putting together all the geometric facts obtained up to this point, specifically Lemma~\ref{qr: lem: almost Lambda cocompact horospheres}. One delicate fact that will be of use is that two maximal $\La$-free horoballs that are based at the same point must be equal, because none of them can contain a $\La$-orbit point while on the other hand both bounding horospheres intersect $\La x_0$. 

Fix $s>0$ for which Proposition~\ref{qr: prop: gamma lies deep inside the la ga horoball} yields a corresponding bound $D(s)$, and let $\g\in \G$ such that $sd_\g>2\cdot \big(D_\La+D(s)\big)$. Consider the (maximal) $\La$-free horoball $\hbl_{\la_\g,i(\g)}$ based at $\xi^{i(\g)}_{\la}$. I show that $\Lambda x_0\cap \hl_{\la_\g,i(\g)}$ is a cocompact metric lattice in $\hl_{\la_\g,i(\g)}$. I keep the subscript notation because the proof is a game between $\hbl_{\la_\g,i(\g)}$ and another $\La$-free horoball.

Let $\hbg_{\la_\g,i(\g)}$ be the $\G$-horoball corresponding to $\hbl_{\la_\g,i(\g)}$. I can conclude that $\hbg_{\la_\g,i(\g)}\subset \hbl_{\la_\g,i(\g)}$, because the choice of $d_\g>D(s)$ guarantees $\g x_0$ is $sd_\g$ deep inside $\hbl_{\la_\g,i(\g)}$. In particular $\hbl_{\la_\g,i(\g)}$ is not $\G$-free. Moreover, it holds that $L:=d(\hl_{\la_\g,i(\g)},\hg_{\la_\g,i(\g)})\geq sd_\g$. Let $\g'\in \G$ be any element in the cocompact metric lattice $\G x_0\cap \hg_{\la_\g,i(\g)}$, and consider two associated points: (a) $\la' x_0=\la_{\g'} x_0$ and (b) the projection of $\g' x_0$ on $\hl_{\la_\g,i(\g)}$, denoted $p_\g':=P_{\hl_{\la_\g,i(\g)}}(\g' x_0)\in\hl_{\la_\g,i(\g)}$. The horoball $\hbl_{\la_\g,i(\g)}$ is a maximal $\La$-free horoball so it is also $1$-almost maximal, hence $d(p_\g',\La x_0)\leq D_\La$ and the following holds: 

\begin{equation} \label{qr: eq: Cocompact lattice: distance to projections of gamma'}
    sd_\g\leq L\leq d_{\g'}\leq d(\g' x_0,p_\g')+d(p_\g',\La x_0)\leq L+D_\La
\end{equation}

Consider $\xi^{i(\g')}_{\la'}$, and assume towards contradiction that $\xi^{i(\g')}_{\la_{\g'}}\ne \xi^{i(\g)}_{\la_\g}$. Both points lie in $\wqg$ and therefore must be at Tits distance $\pi$ of each other. Therefore the fact that $\g' x_0$ lies in $\hbl_{\la_\g,i(\g)}$ implies that the geodesic $[\g' x_0,\xi^{i(\g')}]$ leaves $\hbl_{\la_\g,i(\g)}$ at some point $z\in \hl_{\la_\g,i(\g)}$. 

The fact that $D(s)\leq sd_\g\leq d_{\g'}$ implies that $\g' x_0$ lies $s^2 d_\g$ deep inside $\hbl_{\la_{\g'},i(\g')}$. Therefore the point $z$ also lies at least $s^2 d_\g$ deep inside $\hbl_{\la_{\g'},i(\g')}$, hence $z$ is the centre of a $\La$-free ball of radius at least $s^2 d_\g$. 


By choice of $d_\g$ the point $z$ therefore admits a $2D_\La$ neighbourhood that is $\La$-free. But $z$ lies on $\hl_{\la_\g,i(\g)}$, a maximal horosphere of $\La$, contradicting Lemma~\ref{qr: lem: almost Lambda cocompact horospheres}. I conclude that $\xi^{i(\g')}_{\la_{\g'}}= \xi^{i(\g)}_{\la_\g}$, so $\hbl_{\la_\g,i(\g)}$ and $\hbl_{\la_{\g'},i(\g')}$ are two $\La$-free horoballs that are tangent to a $\La x_0$ point  and based at the same point at $\infty$. This implies $\hbl_{\la_\g,i(\g)}=\hbl_{\la_{\g'},i(\g')}$, and in particular $\la_{\g'}x_0\in\hl$.  Finally, it is clearly seen from Inequality~\ref{qr: eq: Cocompact lattice: distance to projections of gamma'} that

$$d(\la' x_0,p_\g')\leq d(\la' x_0,\g' x_0)+d(\g' x_0,p_\g')\leq d_{\g'}+L\leq L+D_\La+L.$$ 

The element $\g'x_0\in \G x_0\cap \hg_{\la_\g,i(\g)}$ was an arbitrary element, and the above argument shows that the corresponding $\La$-orbit points satisfy: 
\begin{enumerate}
    \item $\la' x_0$ all lie on $\hl_{\la_\g,i(\g)}$.
    \item Each $p_\g'$ is $2L+D_\La$ close to the point $\la' x_0$.
\end{enumerate}

This shows that the cocompact metric lattice $\{p_\g'\mid\g' x_0\in \hg_{\la_\g,i(\g)}\}$ lies in a bounded neighbourhood of the set of points $\La x_0\cap \hl_{\la_\g,i(\g)}$, proving that $\La x_0\cap \hl_{\la_\g,i(\g)}$ is a cocompact metric lattice in $\hl_{\la_\g,i(\g)}$. Lemma~\ref{qr: lem: cocompact metric lattice of horospheremay be intersected with the stabilizer of horosphere} elevates this to 

$$\big(\La\cap\mathrm{Stab}_G(\hl_{\la_\g,i(\g)})\big)\cdot x_0\cap  \hl_{\la_\g,i(\g)}$$

being a cocompact metric lattice in $ \hl_{\la_\g,i(\g)}$, completing the proof. 
\end{proof}

\begin{figure}
    \centering
    
\tikzset{every picture/.style={line width=0.75pt}} 

\begin{tikzpicture}[x=0.75pt,y=0.75pt,yscale=-1,xscale=1]

\draw  [dash pattern={on 1.69pt off 2.76pt}][line width=1.5]  (141,124) .. controls (141,53.31) and (193.59,-4) .. (258.46,-4) .. controls (323.33,-4) and (375.91,53.31) .. (375.91,124) .. controls (375.91,194.69) and (323.33,252) .. (258.46,252) .. controls (193.59,252) and (141,194.69) .. (141,124) -- cycle ;
\draw  [fill={rgb, 255:red, 0; green, 0; blue, 0 }  ,fill opacity=1 ] (335.87,20.63) .. controls (335.87,17.89) and (333.76,15.67) .. (331.16,15.67) .. controls (328.55,15.67) and (326.44,17.89) .. (326.44,20.63) .. controls (326.44,23.36) and (328.55,25.58) .. (331.16,25.58) .. controls (333.76,25.58) and (335.87,23.36) .. (335.87,20.63) -- cycle ;
\draw  [color={rgb, 255:red, 208; green, 2; blue, 27 }  ,draw opacity=1 ][dash pattern={on 1.69pt off 2.76pt}][line width=1.5]  (230.08,222.2) .. controls (230.08,205.73) and (242.79,192.39) .. (258.46,192.39) .. controls (274.13,192.39) and (286.83,205.73) .. (286.83,222.2) .. controls (286.83,238.66) and (274.13,252) .. (258.46,252) .. controls (242.79,252) and (230.08,238.66) .. (230.08,222.2) -- cycle ;
\draw  [color={rgb, 255:red, 208; green, 2; blue, 27 }  ,draw opacity=1 ][fill={rgb, 255:red, 208; green, 2; blue, 27 }  ,fill opacity=1 ] (318.24,47.57) .. controls (318.24,44.84) and (316.13,42.62) .. (313.52,42.62) .. controls (310.92,42.62) and (308.81,44.84) .. (308.81,47.57) .. controls (308.81,50.31) and (310.92,52.53) .. (313.52,52.53) .. controls (316.13,52.53) and (318.24,50.31) .. (318.24,47.57) -- cycle ;
\draw  [fill={rgb, 255:red, 0; green, 0; blue, 0 }  ,fill opacity=1 ] (258.46,252) .. controls (258.46,249.86) and (256.72,248.13) .. (254.58,248.13) .. controls (252.44,248.13) and (250.71,249.86) .. (250.71,252) .. controls (250.71,254.14) and (252.44,255.88) .. (254.58,255.88) .. controls (256.72,255.88) and (258.46,254.14) .. (258.46,252) -- cycle ;
\draw [color={rgb, 255:red, 208; green, 2; blue, 27 }  ,draw opacity=1 ]   (342,240) -- (270.77,228.12) ;
\draw [shift={(268.79,227.79)}, rotate = 9.47] [color={rgb, 255:red, 208; green, 2; blue, 27 }  ,draw opacity=1 ][line width=0.75]    (10.93,-3.29) .. controls (6.95,-1.4) and (3.31,-0.3) .. (0,0) .. controls (3.31,0.3) and (6.95,1.4) .. (10.93,3.29)   ;
\draw  [fill={rgb, 255:red, 0; green, 0; blue, 0 }  ,fill opacity=1 ] (138,43.88) .. controls (138,41.73) and (136.27,40) .. (134.13,40) .. controls (131.98,40) and (130.25,41.73) .. (130.25,43.88) .. controls (130.25,46.02) and (131.98,47.75) .. (134.13,47.75) .. controls (136.27,47.75) and (138,46.02) .. (138,43.88) -- cycle ;
\draw  [color={rgb, 255:red, 208; green, 2; blue, 27 }  ,draw opacity=1 ][fill={rgb, 255:red, 208; green, 2; blue, 27 }  ,fill opacity=1 ] (242,200.88) .. controls (242,198.73) and (240.27,197) .. (238.13,197) .. controls (235.98,197) and (234.25,198.73) .. (234.25,200.88) .. controls (234.25,203.02) and (235.98,204.75) .. (238.13,204.75) .. controls (240.27,204.75) and (242,203.02) .. (242,200.88) -- cycle ;
\draw    (331.16,20.63) .. controls (293.1,33.87) and (233.33,198.4) .. (253.93,246.7) ;
\draw [shift={(254.58,248.13)}, rotate = 244.17] [color={rgb, 255:red, 0; green, 0; blue, 0 }  ][line width=0.75]    (10.93,-3.29) .. controls (6.95,-1.4) and (3.31,-0.3) .. (0,0) .. controls (3.31,0.3) and (6.95,1.4) .. (10.93,3.29)   ;
\draw  [dash pattern={on 0.84pt off 2.51pt}]  (331.16,25.58) -- (313.52,47.57) ;
\draw  [fill={rgb, 255:red, 0; green, 0; blue, 0 }  ,fill opacity=1 ] (163,56.88) .. controls (163,54.73) and (161.27,53) .. (159.13,53) .. controls (156.98,53) and (155.25,54.73) .. (155.25,56.88) .. controls (155.25,59.02) and (156.98,60.75) .. (159.13,60.75) .. controls (161.27,60.75) and (163,59.02) .. (163,56.88) -- cycle ;
\draw  [dash pattern={on 1.69pt off 2.76pt}][line width=1.5]  (77,147.19) .. controls (77,84.89) and (123.35,34.38) .. (180.52,34.38) .. controls (237.69,34.38) and (284.03,84.89) .. (284.03,147.19) .. controls (284.03,209.49) and (237.69,260) .. (180.52,260) .. controls (123.35,260) and (77,209.49) .. (77,147.19) -- cycle ;
\draw  [fill={rgb, 255:red, 0; green, 0; blue, 0 }  ,fill opacity=1 ] (210.88,236) .. controls (210.88,233.86) and (209.14,232.13) .. (207,232.13) .. controls (204.86,232.13) and (203.13,233.86) .. (203.13,236) .. controls (203.13,238.14) and (204.86,239.88) .. (207,239.88) .. controls (209.14,239.88) and (210.88,238.14) .. (210.88,236) -- cycle ;
\draw  [dash pattern={on 0.84pt off 2.51pt}]  (134.13,47.75) -- (159.13,60.75) ;
\draw    (66.13,79.14) -- (118.12,98.31) ;
\draw [shift={(120,99)}, rotate = 200.23] [color={rgb, 255:red, 0; green, 0; blue, 0 }  ][line width=0.75]    (10.93,-3.29) .. controls (6.95,-1.4) and (3.31,-0.3) .. (0,0) .. controls (3.31,0.3) and (6.95,1.4) .. (10.93,3.29)   ;
\draw  [fill={rgb, 255:red, 0; green, 0; blue, 0 }  ,fill opacity=1 ] (180.52,260) .. controls (180.52,257.86) and (178.78,256.13) .. (176.64,256.13) .. controls (174.5,256.13) and (172.77,257.86) .. (172.77,260) .. controls (172.77,262.14) and (174.5,263.88) .. (176.64,263.88) .. controls (178.78,263.88) and (180.52,262.14) .. (180.52,260) -- cycle ;
\draw    (134.13,47.75) .. controls (186.47,116.31) and (193.38,191.6) .. (177.14,254.23) ;
\draw [shift={(176.64,256.13)}, rotate = 285] [color={rgb, 255:red, 0; green, 0; blue, 0 }  ][line width=0.75]    (10.93,-3.29) .. controls (6.95,-1.4) and (3.31,-0.3) .. (0,0) .. controls (3.31,0.3) and (6.95,1.4) .. (10.93,3.29)   ;
\draw  [dash pattern={on 0.84pt off 2.51pt}]  (159.13,56.88) -- (238.13,197) ;
\draw  [dash pattern={on 0.84pt off 2.51pt}]  (238.13,200.88) -- (207,236) ;
\draw   (188.76,236) .. controls (188.76,225.42) and (196.92,216.84) .. (207,216.84) .. controls (217.08,216.84) and (225.24,225.42) .. (225.24,236) .. controls (225.24,246.58) and (217.08,255.16) .. (207,255.16) .. controls (196.92,255.16) and (188.76,246.58) .. (188.76,236) -- cycle ;
\draw    (90.13,247.14) -- (196,240.13) ;
\draw [shift={(198,240)}, rotate = 176.21] [color={rgb, 255:red, 0; green, 0; blue, 0 }  ][line width=0.75]    (10.93,-3.29) .. controls (6.95,-1.4) and (3.31,-0.3) .. (0,0) .. controls (3.31,0.3) and (6.95,1.4) .. (10.93,3.29)   ;
\draw    (407.13,120.14) -- (332.96,105.39) ;
\draw [shift={(331,105)}, rotate = 11.25] [color={rgb, 255:red, 0; green, 0; blue, 0 }  ][line width=0.75]    (10.93,-3.29) .. controls (6.95,-1.4) and (3.31,-0.3) .. (0,0) .. controls (3.31,0.3) and (6.95,1.4) .. (10.93,3.29)   ;

\draw (308.71,53.01) node [anchor=north west][inner sep=0.75pt]  [color={rgb, 255:red, 208; green, 2; blue, 27 }  ,opacity=1 ]  {$\gamma x_{0}$};
\draw (319.31,39.94) node [anchor=north west][inner sep=0.75pt]  [font=\scriptsize,rotate=-310.75]  {$d_{\gamma }$};
\draw (221,172) node [anchor=north west][inner sep=0.75pt]  [color={rgb, 255:red, 208; green, 2; blue, 27 }  ,opacity=1 ]  {$\gamma 'x_{0}$};
\draw (128.22,15.27) node [anchor=north west][inner sep=0.75pt]  [font=\footnotesize]  {$\lambda _{\gamma '} x_{0}$};
\draw (23,50) node [anchor=north west][inner sep=0.75pt]    {$\mathcal{HB}_{\lambda _{\gamma '} ,i( \gamma ')}^{\Lambda }$};
\draw (169.13,54.88) node [anchor=north west][inner sep=0.75pt]  [font=\small]  {$z':=\pi _{\mathcal{HB}_{\lambda _{\gamma '} ,i( \gamma ')}^{\Lambda }}( \gamma 'x_{0})$};
\draw (147.67,34.37) node [anchor=north west][inner sep=0.75pt]  [font=\tiny,rotate=-29.18]  {$\overbrace{}^{D_{1}}_{1}$};
\draw (148.5,259) node [anchor=north west][inner sep=0.75pt]    {$\xi _{\lambda _{\gamma '}}^{i( \gamma ')}$};
\draw (242.5,252) node [anchor=north west][inner sep=0.75pt]    {$\xi _{\lambda _{\gamma }}^{i( \gamma )}$};
\draw (200.51,117.64) node [anchor=north west][inner sep=0.75pt]  [font=\footnotesize,rotate=-12.99]  {$\geq \frac{1}{2} d_{\gamma }$};
\draw (195,240) node [anchor=north west][inner sep=0.75pt]  [font=\large]  [color={rgb, 255:red, 2; green, 27; blue, 208 }  ,opacity=1 ] {$\mathbf{z}$};
\draw (191.84,210.53) node [anchor=north west][inner sep=0.75pt]  [font=\scriptsize,rotate=-327.08]  {$\geq \frac{1}{2} d_{\gamma '}$};
\draw (6.83,215) node [anchor=north west][inner sep=0.75pt]    {$ \begin{array}{l}
\Lambda -free\ ball\\
\ B\left( z,\frac{1}{2} d_{\gamma }\right)
\end{array}$};
\draw (404,112) node [anchor=north west][inner sep=0.75pt]    {$\mathcal{HB}_{\lambda _{\gamma } ,i( \gamma )}^{\Lambda }$};
\draw (342,235) node [anchor=north west][inner sep=0.75pt]  [color={rgb, 255:red, 208; green, 2; blue, 27 }  ,opacity=1 ]  {$\mathcal{HB}_{\lambda _{\gamma } ,i( \gamma )}^{\Gamma }$};
\draw (339.22,0.27) node [anchor=north west][inner sep=0.75pt]    {$\lambda _{\gamma } x_{0}$};

\end{tikzpicture}

    \caption[A contradiction to almost cocompactness]{Proposition~\ref{qr: prop: Lambda cocompact horosphreres}. Assuming towards contradiction that $\xi^{i(\g)}_{\la_\g}\ne \xi^{i(\g')}_{\la_{\g'}}$ results in a point $z\in\hbl_{\la_\g,i(\g)}$ (blue coloured and bold faced in the bottom part of the figure) admitting a large $\La$-free neighbourhood, contradicting almost cocompactness.}
    \label{fig: proposition lambda cocompactness on horospheres}
\end{figure}
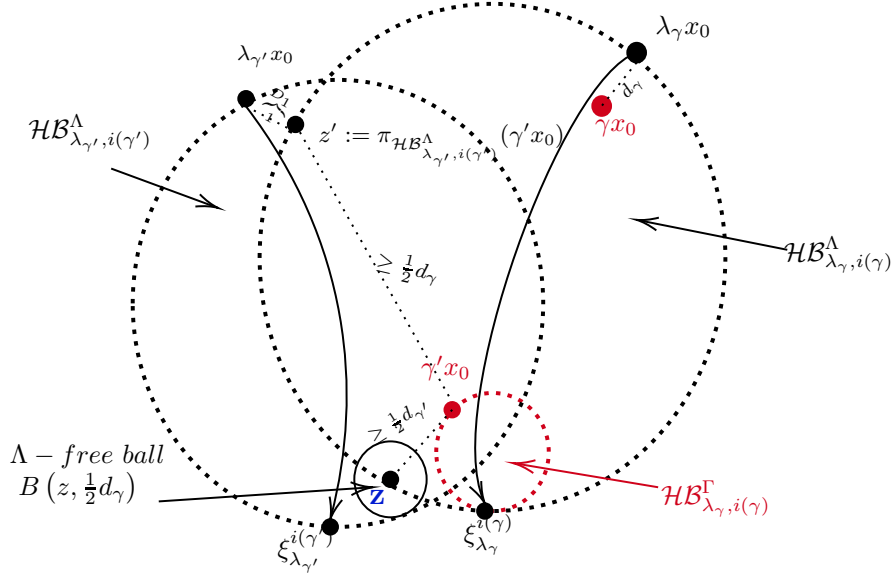

\subsection{The Bounded Case}\label{sec: qr: bounded case}

Proposition~\ref{qr: prop: Lambda cocompact horosphreres} is enough in order to prove Theorem~\ref{qr: thm: main for qr} in the case that $\G$ does not lie in a bounded neighbourhood of $\La$. The case where $\G$ and $\La$ lie at bounded Hausdorff distance, i.e.\ where $\G\subset\mn_D(\La)$ and $\La\subset\mn_D(\G)$ for $D>0$, arose naturally in the context of quasi-isometric completeness of non-uniform lattices. The precise statements are given in Theorems~\ref{qr: thm: eskin bounded case} and~\ref{qr: thm: Schwartz: finite distance imples lattice} above. I present their proofs in Section~\ref{sec: qr: bounded: Schwartz} below.

The notable difference between Theorem~\ref{qr: thm: eskin bounded case} and Theorem~\ref{qr: thm: Schwartz: finite distance imples lattice} is that for high rank groups, the inclusion $\La\subset\mn_D(\G)$ is only required to prove commensurability. In view of Corollary~\ref{qr: cor: bounded case real rank 1 one sided}, this allows me to omit that assumption from Theorem~\ref{thm: main}. Notice also that for groups with property (T) the result easily follows from the (much more recent) result by  Leuzinger in Theorem~\ref{T: thm: Leuzinger criterion}. 

In the context of commensurability in the sublinear setting, I can only prove the limited result in Theorem~\ref{thm: commensurability and uniform statement intro}, which relies on a reduction to the bounded case (Proposition~\ref{qr: prop: bounded and sublinear implies bounded}).

\subsubsection{A Reduction}
The first step towards Proposition~\ref{qr: prop: bounded and sublinear implies bunded and bounded} is proving that $\La$ must preserve $\wqg$. 

\begin{lem}\label{qr: lem: if Lambda is in sublinear of Gamma, then Lambda preserves wqg rational tits building}
Let $G$ be a real finite-centre semisimple Lie group without compact factors, $\G\leq G$ an irreducible non-uniform lattice of \qrc $\La\leq G$ a discrete subgroup. Assume that $\G\subset \mn_u(\La)$ and that $\La\subset \mn_{u'}(\G)$ for sublinear functions $u,u'$. Then $\La\cdot \wqg\subset\wqg$. Moreover, if $G$ is of $\rr$, the conclusion holds under the relaxed assumption that $u'(r)\preceq_\infty \varepsilon r$ for some $\varepsilon<1$.
\end{lem}

\begin{proof}
The proof is similar to the argument of Lemma~\ref{qr: lem: Lambda free horoballs are based at rational tits building wqg}. Let $\xi\in \wqg$, and let $\hg$ be a horosphere bounding a $\G$-free horoball $\hbg$ with $\hg\cap \G x_0$ a metric lattice in $\hg$. Assume first that $u'$ is sublinear. Since $\hbg$ is $\G$-free and $\La\subset \mn_{u'}(\G)$, I can conclude that $\La x_0 \cap \hbg\subset \mn_{u'}(\hg)$. Recall that every geodesic ray $\eta$ with limit point $\xi'\in \mn_\fpi{2}(\xi)$ penetrates $\hbg$ at linear rate~(Lemma~\ref{qr: prelims: sym spc: lem: Hattori penetration}). Therefore for every such geodesic ray $\eta$ and every sublinear function $v$ there is $R=R(\eta,v)>0$ for which $\mn_v(\eta_{\restriction_{r>R}})$ is $\La$-free. 

On the other hand, let $\la\in \La$, and assume towards contradiction that $\la\xi\notin \wqg$. Then by Proposition~\ref{qr: cor: limit points characterization of wqg rational building} there is a $\G$-conical limit point $\xi'\in \mn_\fpi{2}(\la\xi)$. The hypothesis that $\G\subset\mn_u(\La)$ then implies that for every $R>0$, $\mn_u(\eta_{\restriction{r>R}})\cap \La x_0\ne\emptyset$. Translating by $\la^{-1}$ yields a contradiction to the previous paragraph. I conclude that $\la\xi\in\wqg$. 

I now modify the argument to include $u'(r)\preceq_\infty \varepsilon r$ when $G$ is of $\rr$. In this case, the only point $\xi'\in\mn_\fpi{2}(\la\xi)$ is $\la\xi$ itself. Therefore by the same argument as above, the assumption that $\la\xi\notin\wqg$ implies that the $u$-sublinear neighbourhood of every geodesic ray with limit point $\xi$ intersects $\La x_0$. I.e., for every $\eta$ with limit point $\xi$ and every $R>0$ it holds that $\mn_u(\eta_{\restriction_{r>R}})\cap \La x_0\ne\emptyset$. On the other hand, every such geodesic penetrates $\hbg$ at $1$-linear rate. This amounts to the following fact: if $v'(r)=\varepsilon r$ for some $\varepsilon\in(0,1)$, then for some $R>0$, the set $\mn_u(\eta_{\restriction_{r>R}})\cap\mn_{v'}(\hg)=\emptyset$. This is a contradiction to $\La\subset\mn_{u'}(\G)$.
\end{proof}

\begin{proof}[Proof of Proposition~\ref{qr: prop: bounded and sublinear implies bunded and bounded}]
Assume towards contradiction that there are $\la_n$ for which $d(\la_n x_0,\G x_0)>n$. Recall that $\G x_0$ is a cocompact metric lattice in the compact core of $\G$. This implies that there is a number $D'>0$ such that any $\la\in \La$ for which $\la x_0\notin\mn_{D'}(\G x_0)$ must lie at least $\frac{1}{2}D'$-deep inside a horoball of $\G$. I can assume that for all $n\in \mathbb{N}$ there are corresponding horoballs of $\G$, which I denote $\hbg_n$, for which $\la_n x_0\in \hbg_n$. The fact that $\G\subset\mn_D(\La)$ then implies that $\mn_D(\La x_0)$ covers a cocompact metric lattice in $\hg_n$, namely the metric lattice $\G x_0\cap \hg_n$. In the terminology of Section~\ref{sec: qr: cocompact horosphere}, $\hg_n$ is almost $\La$-cocompact, or $D$-almost $\La$-cocompact. 

I first prove that every horoball of $\G$ contains a $\La$-free horoball (this is of course immediate if $\La\subset\mn_{C}(\G)$ for some $C>0$). Assume towards contradiction that there is a horoball $\hbg$ of $\G$ that does not contain a $\La$-free horoball. Denote $\hg:=\partial\hbg$. In the notations of the previous paragraph, I may assume without losing generality that $\hbg=\hbg_n$ for all $n\in\mathbb{N}$. Denote by $\xi$ the base point of $\hbg$, fix some arbitrary $x\in \hg$ and consider the geodesic ray $\eta:=[x,\xi)$. The constraint that $\La\subset\mn_u(\G)$ implies that for every $R>0$ there is some $L>0$ for which the ball $B\big(\eta(L+t),R\big)$ is $\La$-free, for all $t\geq 0$. In particular, for all large enough $n\in \mathbb{N}$ (depending on $R$), the horosphere $\h(\xi,\la_n x_0)$ that is parallel to $\hg$ and that passes through $\la_n x_0$ contains a point that is the centre of $\La$-free ball of radius $R$. This property is $\La$-invariant, as well as the fact that $\hbg$ is based at $\wqg$. In particular, these two properties hold for the horoballs $\hb_n:=\la_n^{-1}\cdot \hbg$, whose respective base points I denote $\xi_n:=\la_n^{-1}\xi\in \wqg$.

Fix $R=D+2D_\G$ (recall that $D_\G$ is such that every horosphere $\h$ of $\G$ admits $\h\subset\mn_{D_\G}(\G x_0\cap \h)$). Let $L=L(D+2D_\G)$ be the corresponding bound from the previous paragraph. For every $n>L$ the horoball $\hb_n$ has bounding horosphere $\h_n$ that admits a point $z_n\in \h_n$ for which $B(z_n,D+2D_\G)$ is $\La$-free. Moreover, the same is true for every horosphere that is parallel to $\h_n$ which lies inside $\hb_n$. Since $\G\subset \mn_D(\La)$, this means that every horosphere that lies inside $\hb_n$ admits a point that is the centre of a $\G$-free ball of radius $2D_\G$. I conclude that none of those horospheres could be the horosphere of $\G$ corresponding to the parabolic limit point $\xi_n\in \wqg$. Since $x_0\in \h_n$ it must therefore be that $\h_n$ is a horosphere of $\G$. But this contradicts the fact that $z_n\in\h_n$ and $B(z_n,2D_\G)$ is $\G$-free. This shows that no horoball of $\G$ contains a sequence of $\La$-orbit points that lie deeper and deeper in that horoball. Put differently, it shows that every horoball of $\G$ contains a $\La$-free horoball. 

I remark that the above argument shows something a bit stronger, which I will not use but which I find illuminating. It proves that as soon as $d(\la x_0,\G x_0)$ is uniformly large enough, say more than $M$, then $\la x_0$ must lie on a $(D+2D_\G$)-almost $\La$-cocompact horosphere parallel to $\hg$, where $\hg$ is the bounding horosphere of any horoball of $\G$ in which $\la x_0$ lies (recall that it must lie in at least one such horoball). On the other hand if $d(\la x_0,\G x_0)<M$, then since every point in the $\G$-orbit lies on a horosphere of $\G$ one concludes that $\la x_0$ lies on a horosphere $\h$ based at $\wqg$ that is $(M+D)$-almost $\La$-cocompact.

I can now assume that every $\hbg_n$ contains a $\La$-free horoball. In particular it contains a $1$-almost maximal $\La$-free horoball $\hbl_n$ (see Definition~\ref{qr: def: almost maximal Lambda free horoball}). By definition there is a point $\la_n' x_0$ that is at distance at most $1$ from $\hl_n=\partial \hbl_n$. Up to enlarging $d(\la_n x_0,\G x_0)$ or decreasing it by at most $1$, I can assume $\la_n=\la_n'$ to begin with. Consider $\hb_n:=\la_n^{-1}\cdot \hbl_n$ with $\h_n=\partial \hb_n$. This is a sequence of horoballs, each of which contains a $\La$-free horoball at depth at most $1$,  based at corresponding parabolic limit points $\xi_n\in \wqg$, and tangent to points that are at distance at most $1$ from $x_0$, i.e., $\h_n\cap B(x_0,1)\ne\emptyset$.

Since $\G\subset \mn_D(\La)$, each $\hb_n$ contains a $\G$-free horoball at depth at most $D+1$. Therefore the horoball of $\G$ that is based at $\xi_n$ must have its bounding horosphere intersecting $B(x_0,D+2)$. By Corollary~\ref{qr: prelims: sym spc: cor: every ball intersects finitely many horoballs of Gamma} there are only finitely many such horoballs, hence there are finitely many points $\xi'_1,\dots,\xi'_K\in\wqg$ such that for every $n\in\mathbb{N}$ there is $i(n)\in\{1,\dots,K\}$ with $\xi_n=\xi'_{i(n)}$. From the Pigeonhole Principle there is some $\xi'\in\{\xi'_1\dots,\xi'_K\}$ for which $\xi_n=\xi'$ for infinitely many $n\in\mathbb{N}$. Passing to a subsequence I assume that this is the case for all $n\in \mathbb{N}$. 

To begin with the $\hbg_n$ are horoballs of $\G$, and therefore as in the first case the bounding horospheres $\hg_n$ are $D+2D_\G$-almost $\La$-cocompact. This is a $\La$-invariant property and therefore the same holds for the $\la_n^{-1}$ translate of it. These are the horospheres which are based at $\xi'$ and lie outside $\hb_n$ at distance $d(\la_n x_0,\G x_0)>n-1$ from $\h_n$. They form a sequence of outer and outer horospheres based at the same point at $\wqg$, all of which are $D+2D_\G$-almost $\La$-cocompact. This is a contradiction, since the union of such horospheres intersect every horoball of $X$, contradicting the existence of $\La$-free horoballs. Formally, take some $\zeta\in \wqg$ different from $\xi'$. Since both $\xi'$ and $\zeta$ lie in $\wqg$, they admit $d_T(\zeta,\xi')=\pi$ and there is a geodesic $\eta$ with $\eta(-\infty)=\xi'$ and $\eta(\infty)=\zeta$. Let $\hbg_\zeta$ be the horoball of $\G$ that is based at $\zeta$. By the first step of this proof, every such horoball must contain a $\La$-free horoball $\hbl_\zeta$. Therefore there is some $T>0$ such that for all $t>T$ the point $\eta(t)$ lies $2(D+2D_\G)$ deep in $\hbl_\zeta$. I conclude that for all $t>T$, $B\big(\eta(t),2D\big)$ is $\La$-free. On the other hand, for arbitrarily large $t$ it holds that the horosphere based at $\xi'$ and tangent to $\eta(t)$ is $D+2D_\G$-almost $\La$-cocompact, and in particular $d\big(\eta(t),\La x_0\big)<D+2D_\G$, a contradiction.

I conclude that $\La\subset\mn_{D'}(\G)$ for some $D'>0$, as claimed. 
\end{proof}

\subsubsection{The Arguments of Schwartz and Eskin}\label{sec: qr: bounded: Schwartz}

\paragraph{\textbf{The $\rr$ case.}}
The statement of Theorem~\ref{qr: thm: Schwartz: finite distance imples lattice} is a slight modification of Schwartz's original formulation. His framework leads to a discrete subgroup $\Delta\leq G$ such that:

\begin{enumerate}
    \item Every element of $\Delta$ \emph{quasi-preserves} the compact core of the lattice $\G$. Namely, each element of $\Delta$ is an isometry of $X$ that preserves $\wqg$ and that maps every horosphere of $\G$ to within the $D=D(\Delta)$ neighbourhood of some other horosphere of $\G$. 
    
    \item It holds that $\G\subset\mn_D(\Delta)$.
\end{enumerate}

From these two properties Schwartz is able to deduce that $\Delta$ has finite covolume, i.e.\ that $\Delta$ is a lattice in $G$. Here is a sketch of his argument, which works whenever $\G$ is a \qr lattice. 

\begin{thm}\label{qr: thm: Schwartz argument}
In the setting described above, $\Delta$ is a lattice in $G$.
\end{thm}

\begin{proof}[Proof sketch]
Consider $X_0':=\bigcup_{g\in\Delta}g X_0$, where $X_0$ is the compact core of $\G$. This space serves as a `compact core' for $\Delta$: the fact that $\Delta$ quasi-preserves $X_0$ implies that $X_0'\subset\mn_D(X_0)$. It is a $\Delta$-invariant space, and therefore one gets an isometric action of $\Delta$ on $X_0'$. This action is cocompact: the reason is that $\G$ acts cocompactly on $X_0$, and $\G\subset\mn_D(\Delta)$. Formally, every point in $X_0'$ is $D$-close to a point in $X_0$. Every point in $X_0$ is $D_\G$-close to a point in $\G x_0$. Every point in $\G x_0$ is $D$-close to a point in $\Delta x_0$. Therefore the ball of radius $2D+D_\G$ contains a fundamental domain for the action of $\Delta$ on $X_0'$. 

It remains to see that the action of $\Delta$ on $X\setminus X_0'$ is of finite covolume. As a result of the cocompact action of $\Delta$ on $X_0'$, there is $B:=B(x_0,R)$ so that $X_0'\subset \Delta\cdot B$. $X_0'$ is the complement of a union of horoballs, which one may call \emph{horoballs of $\Delta$}, with bounding \emph{horospheres of $\Delta$}. The fact that $\G$ is of \qr means that the horoballs of $\G$ are disjoint, and therefore those of $\Delta$ are almost disjoint: there is some $C>0$ such that for every horosphere $\h$ of $\Delta$ and every point $x\in \h$, $d(x,X_0')<C$. Up to enlarging the radius of $B$ by $C$, I can assume that $\h\subset \Delta\cdot B$ for every horosphere $\h$ of $\Delta$. 

Each horoball of $\La$ is based at $\wqg$, and each lies uniformly boundedly close to the corresponding horoballs of $\G$. From Corollary~\ref{qr: prelims: sym spc: cor: every ball intersects finitely many horoballs of Gamma} one therefore sees that there are finitely many horoballs of $\Delta$ that intersect $B$. Denote them by $\hb_1,\dots,\hb_N$, their bounding horospheres by $\h_i=\partial\hb_i$, and their intersection with $B$ by $B_i:=B\cap \hb_i$. Let also $\xi_i\in\wqg$ denote the base point of each $\hb_i$. Each $B_i$ is pre-compact and therefore the projection of each $B_i$ on $\h_i$ is pre-compact as well (this is a consequence e.g.\ of the results of Heintze-Im hof recalled in Remark~\ref{qr: prelims: rmk: heintze-im hof results about metric on horospheres}). Let $D_i\subset \h_i$ be a compact set that contains this projection, i.e.\ $P_{\h_i}(B_i)\subset D_i\subset\h_i$. In particular $B\cap\h_i\subset D_i$.

Observe now that for every horoball $\hb$ of $\Delta$, with bounding horosphere $\h=\partial\hb$, the $\Delta$-orbit of every point $x\in \h$ intersects some $D_i$. For $x\in \h$ the choice of $B$ implies that the  $\Delta$-orbit of $x$ must intersect $B$, say $gx\in B$. In particular $g\h\cap B\ne\emptyset$, and since $g\hb$ is a horoball of $\Delta$ then by definition $g\h=\h_i$ for some $i\in\{1,\dots,N\}$. One concludes that indeed $gx\subset D_i$. 

Moreover, let $y\in X$ be any point that lies inside a horoball $\hb$ of $\Delta$, and $x=P_\h(y)$ its projection on the bounding horosphere $\h=\partial \hb$. By the previous paragraph there is some $g\in\Delta$ and $i\in\{1,\dots,N\}$ for which $gx\in D_i$, and therefore it is clear that $gy$ lies on a geodesic emanating from $D_i$ to $\xi_i$.

Finally, define $\mathrm{Cone}(D_i)$ to be the set of all geodesic rays that emanate from $D_i$ and with limit point $\xi_i$. The previous paragraph proves that $\bigcup_{i=1}^N \mathrm{Cone}(D_i)$ contains a fundamental domain for the action of $\Delta$ on $X\setminus X_0'$. Moreover, the fact that $D_i\subset\h_i$ is compact readily implies that each $\mathrm{Cone}(D_i)$ has finite volume, and so this fundamental domain is of finite volume.

To conclude, $B\cup \big(\bigcup_{i=1}^N \mathrm{Cone}(D_i)\big)$ is a set of finite volume and it contains a fundamental domain for the $\Delta$-action on $X$, as claimed. The proof of  commensurability of $\Delta$ and $\G$ is given in full in \cite{Schwartz}. 

\end{proof}

There is one essential difference between Theorem~\ref{qr: thm: Schwartz: finite distance imples lattice} and Theorem~\ref{qr: thm: Schwartz argument}, namely the assumption that $\La\subset\mn_D(\G)$ rather than that $\La$ quasi-preserves the compact core of $\G$. In Schwartz's work, the fact that $\Delta\subset\mn_D(\G)$ is not relevant (even though it easily follows from the construction of his embedding of $\Delta$ in $G$). He only uses the two properties described above, namely the quasi-preservation of $X_0$ and $\G\subset\mn_D(\Delta)$.

The assumption that $\La$ quasi-preserves the compact core of $\G$ does not seem appropriate in the context of sublinear rigidity, while the metric condition $\La\subset\mn_D(\G)$ is more natural. It is a stronger condition as I now show. By Lemma~\ref{qr: lem: if Lambda is in sublinear of Gamma, then Lambda preserves wqg rational tits building}, $\La\cdot \wqg\subset \wqg$. Let $\hg_1$ be a horosphere of $\G$, based at $\xi\in \wqg$, and let $\g x_0\in \hg_1$ be some point on the metric lattice of $\G x_0$ on $\hg_1$. There is an element $\la\in \La$ such that $d(\la x_0,\g x_0)<D$. Moreover, since $\La\subset \mn_D(\G)$ one knows that the parallel horoball that lies $D$-deep inside $\hbg_1$ is $\La$-free.

Let $\la'\in \La$ be an arbitrary element, and consider $\la' \hg_1$. The last statement in the previous paragraph is $\La$-invariant, and so the horoball that lies $D$-deep inside $\la' \hbg_1$ is $\La$-free. The fact that $\G\subset\mn_D(\La)$ then implies that the parallel horoball that lies $2D$ deep inside $\la'\hbg_1$ is $\G$-free. Let $\hg_2$ be the horosphere of $\G$ that is based at $\la'\xi$. The last statement amounts to saying that $\hg_2$ lies at most $2D$-deep inside $\la'\hbg_1$. On the other hand, one has $d(\la'\la x_0,\la'\hg_1)=d(\la x_0,\hg_1)\leq D$, so there is a $\La$-orbit point that lies within $D$ of $\la'\hg_1$. The parallel horoball that lies $D$-deep inside $\hbg_2$ must also be $\La$-free, so $\hg_2$ must be contained in the parallel horoball to $\la'\hbg_1$ which contains it and that is at distance $D$ from it. I conclude that  $d(\la'\hg_1,\hg_2)\leq 2D$, and so that $\La$ quasi-preserves $X_0$. 

\paragraph{\textbf{High rank.}}

Eskin's proof uses ergodic theory, and is based on results of Mozes~\cite{mozesEpimorphic} and Shah~\cite{ShahMumbai}. I present it here without the necessary preliminaries, which are standard.

\begin{proof}[Proof of Theorem~\ref{qr: thm: eskin bounded case}]
To prove that $\La$ is a lattice amounts to finding a finite non-zero $G$-invariant measure on $\La\backslash G$. By Theorem~$2$ in \cite{mozesEpimorphic}, if $P\leq G$ is a parabolic subgroup then every $P$-invariant measure on $\La\backslash G$ is automatically $G$-invariant. Fix a minimal parabolic subgroup  $P\leq G$ and let $\mu_0$ be some fixed probability measure on $\La\backslash G$. Since $P$ is amenable it admits a tempered F\o{}lner sequence $F_n\subset P$, and one can average $\mu_0$ along each $F_n$ to get a sequence of probability measures $\mu_n$. The weak* compactness of the unit ball in the space of measures on $\La\backslash G$ implies that there exists a weak* limit $\mu$ of the $\mu_n$. The measure $\mu$ is automatically a finite $P$-invariant measure. It remains to show that $\mu$ is not the zero measure. To see this it is enough to show that for some compact set $C_\La\subset\La\backslash G$ and some $\La g=x\in \La\backslash G$, one has 

\begin{equation}\label{bounded case: Eskin: eq: liminf}
0<\liminf_n\frac{1}{|F_n|}\int_{F_n}{\mathbbm{1}_{C_\La}(xp^{-1})dp}.
\end{equation}

Fix some compact neighbourhood $C_\G\subset\G\backslash G$ of the trivial coset $\G e$. The hypothesis $\G\subset\mn_D(\La)$ implies that there is a corresponding compact neighbourhood $C_\La\subset \La\backslash G$ of the trivial coset $\La e$ such that for any $p\in P$, it holds that $\G gp^{-1}\in C_\G\Rightarrow \La gp^{-1}\in C_\La$ (simply take $C_\La$ to be the $D+1$-blowup of $C_\G$). The action of $P$ on $\G\backslash G$ is uniquely ergodic, therefore

$$0<\mu_\G(C_\G)=\lim_n\frac{1}{|F_n|}\int_{F_n}{\mathbbm{1}_{C_\G}(\G p^{-1})dp}$$

where $\mu_\G$ denotes the natural $G$-invariant measure on $\G\backslash G$. The defining property of $C_\La$ ensures that Inequality~\eqref{bounded case: Eskin: eq: liminf} is satisfied, implying that $\mu$ is a non-zero $P$-invariant probability measure on $\La\backslash G$. One concludes that $\mu$ is also $G$-invariant, and that $\La$ is a lattice.  If moreover $\La\subset\mn_D(\G)$, one may use Shah's Corollary \cite{ShahMumbai} to conclude that $\La$ is commensurable to $\G$.
\end{proof}

\subsection{Translating Geometry into Algebra}\label{sec: qr: prelims: algebra} \label{sec: qr: geometry to algebra}

The goal of this section is to prove that the results of Section~\ref{sec: qr: cocompact horosphere} imply that $\La$ satisfies the hypotheses of Theorem~\ref{qr: thm: Benoist-Miquel arithmeticity}: that it is Zariski-dense, and that it intersects a horospherical subgroup in a cocompact indecomposable lattice. The required geometric data, given by Section~\ref{sec: qr: cocompact horosphere}, is that for some horosphere $\h$ bounding a $\La$-free horoball, $\La\cap\mathrm{Stab}_G(\h) \cdot x_0$ intersects $\h$ in a cocompact metric lattice (Proposition~\ref{qr: prop: Lambda cocompact horosphreres}), and that the set of $\La$-conical limit points contains the set of $\G$-conical limit points (Corollary~\ref{qr: cor: every Gamma conical is Lambda conical}). Note that since $K$ is compact the former implies that $\La\cap\mathrm{Stab}_G(\h)$ is a uniform lattice in $\mathrm{Stab}_G(\h)$.

\subsubsection{A Horospherical Lattice}

\begin{lem}[Lemma $3.9$ in \cite{MostowAtirhmeticIntersectionNilRad}] \label{qr: prelims: translation: lem: Mostow nilpotent radical intersection}
Let H be a Lie group having no compact connected normal semisimple non-trivial Lie subgroups, and let $N$ be the maximal connected nilpotent normal Lie subgroup of $H$. Let $\G\leq H$ be a lattice. Then $N/N\cap \G$ is compact.
\end{lem}

\begin{rmk}
In the original statement Mostow uses the term `analytic group', which I replaced here with `connected Lie subgroup'. This appears to be Mostow's definition of an analytic group. See e.g.\ Section $10$, Chapter $1$ in \cite{knappliegroups}. In Chevalley's \emph{Theory of Lie Groups}, he defines a \emph{Lie group} as a locally connected topological group whose identity component is an analytic group (Definition $1$, Section $8$, Chapter $4$ in \cite{ChavalleyTheoryofLieGroups}), and proves (Theorem $1$, Section $4$, Chapter $4$ therein) a 1-1 correspondence between analytic subgroups of an analytic group and Lie subalgebras of the corresponding Lie algebra. 
\end{rmk}

Lemma~\ref{qr: prelims: translation: lem: Mostow nilpotent radical intersection} lays the rationale for the rest of this section. Explicitly, I prove that $\mathrm{Stab}_G(\h)$ admits a subgroup that is a horospherical subgroup $U$ of $G$ (Corollary~\ref{qr: prelims: translation: cor: structure of horosphere stabilizer}), and that $U$ is maximal connected nilpotent normal Lie subgroup of $\mathrm{Stab}_G(\h)$ (Corollary~\ref{qr: prelims: translation: cor: horospherical is maximal nilpotent in stabilizer of horosphere}).

\begin{defn}
In the notation $h_\xi^t=\exp(tX)$ and $A_\xi=\exp\big(Z(X)\cap \mf{p}\big)$ of Proposition~\ref{qr: prop: Eberline Generalized Iwasawa}, define $A_\xi^\perp$ to be the codimension $1$ submanifold of $A_\xi$ that is orthogonal to $\{h_\xi(t)\}_{t\in \mathbb{R}}$ (with respect to the Killing form in the Lie algebra).
\end{defn}

\begin{claim}\label{qr: prelims: translation: claim: A perp stabilizes horosphere}
Every element $a\in A_\xi^\perp$ stabilizes $\h=\h(x_0,\xi)$.
\end{claim}

\begin{proof}
An element in $A_\xi$ is an element that maps $x_0$ to a point on a flat $F\subset X$ that contains the geodesic ray $[x_0,\xi)$. If $a\in A_\xi^\perp$, then the geodesic $[x_0,ax_0]$ is orthogonal to $[x_0,\xi)$, and lies in $F$. From Euclidean geometry and structure of horospheres in Euclidean spaces, it is clear that $ax_0\in \h(x_0,\xi)$. Since $a\in G_\xi$, this means $a\h=\h(ax_0,\xi)=\h(x_0,\xi)=\h$.
\end{proof}

\begin{cor}\label{qr: prelims: translation: cor: structure of horosphere stabilizer}
Let $\h$ be a horosphere based at $\xi$. Then  $\mathrm{Stab}_G(\h)^\circ=(K_\xi A_\xi^\perp)^\circ N_\xi$, and in particular it contains a horospherical subgroup of $G$. Moreover, $\mathrm{Stab}_G(\h)^\circ$ is normal in $\mathrm{Stab}_G(\xi)^\circ$ and acts transitively on $\h$. 
\end{cor}

\begin{proof}
Clearly $(K_\xi A_\xi^\perp)^\circ N_\xi$ is a codimension $1$ subgroup of $\mathrm{Stab}_G(\xi)^\circ$. Since $\mathrm{Stab}_G(\h)\ne \mathrm{Stab}_G(\xi)$ (e.g.\ $h_\xi^t\notin \mathrm{Stab}_G(\h)$ for $t\ne 0$), it is enough to show that $(K_\xi A_\xi^\perp)^\circ N_\xi\leq \mathrm{Stab}_G(\h)$. Let $kan\in (K_\xi A_\xi^\perp)^\circ N_\xi$. It fixes $\xi$, so it is enough to show that $kanx_0\in \h$. Since $k\in K_\xi$ and $kx_0=x_0$, it stabilizes $\h$. By Claim~\ref{qr: prelims: translation: claim: A perp stabilizes horosphere} $a\in \mathrm{Stab}_G(\h)$, so it remains to check that $N_\xi$ stabilizes $\h$.  Any $g\in G_\xi$ admits $l(g)\in\mathbb{R}$ such that, if $\{\h_t\}_{t\in \mathbb{R}}$ is the filtration by horospheres based at $\xi$, then $g\h_t=\h_{t+l(g)}$. Clearly $l\big(h_\xi(t)\big)=t$. The element $n$ fixes $\xi$, so one has 

$$h_\xi^{-t}nh_\xi^t\h_0 = h_\xi^{-t}\h_{t+l(n)}=\h_{t+l(n)-t}=\h_{l(n)}$$

The fact that $n\in Ker(T_\xi)$, i.e.\ that $\lim_{t\rightarrow \infty}h_\xi^{-t}nh_\xi^t=e_G$, readily implies that necessarily $l(n)=0$. I conclude that $(K_\xi A_\xi^\perp)^\circ N_\xi=\mathrm{Stab}_G(\h)^\circ$, as wanted.

Next recall that $\mathrm{Stab}_G(\h)^\circ$ acts transitively on $X$.  Let $x,y\in \h$, and consider $g\in \mathrm{Stab}_G(\xi)^\circ$ with $gx=y$. Writing an element $g\in G_\xi$ as $ka_ta_\perp n\in K_\xi h_\xi^tA_\xi^\perp N_\xi$, the argument above shows that $kh_\xi^ta_\perp n\h_0=\h_0$ if and only if $t=0$, i.e., if and only if $g\in \mathrm{Stab}_G(\h)^\circ$.

Finally, let $g\in \mathrm{Stab}_G(\xi)$ and $h\in \mathrm{Stab}_G(\h)$. By the discussion above $h\cdot\h_t=\h_t$ for all $t\in \mathbb{R}$. Clearly $-l(g)=l(g^{-1})$, and therefore 
$$ghg^{-1}\h_0=gh\h_{-l(g)}=g \h_{-l(g)}=\h_0$$

Therefore $\mathrm{Stab}_G(\h)$ is normal in $\mathrm{Stab}_G{\xi}$, and the same is true for the respective identity components. 
\end{proof}

\begin{cor}\label{qr: prelims: translation: cor: stab horosphere has no compact factors}
$\mathrm{Stab}_G(\h)^\circ$ is a connected Lie group with no connected compact normal semisimple non-trivial Lie subgroups. 
\end{cor}

\begin{proof}
Every compact subgroup of $G$ fixes a point. Let $H\leq G$ be some closed subgroup. It is standard to note that a normal $N\leq H$ that fixes a point $x\in X$ must fix every point in the orbit $H x$: $hnh^{-1}hx=hx$. Since $H=\mathrm{Stab}_G(\h)^\circ$ acts transitively on $\h$, it shows that a normal compact subgroup of $\mathrm{Stab}_G(\h)^\circ$ fixes every point in $\h$. An isometry fixing a horosphere pointwise while fixing its base point is clearly the identity, proving the claim.  
\end{proof}

\begin{lem}[Theorem $11.13$ in \cite{sagleIntroductionLieGroupsAlgebras}]\label{qr: prelims: translation: lem: N connected nilpotent if an only if lie algebra is nilpotent}
Let $N$ be a connected real Lie group. Then $\mathrm{Lie}(N)$ is a nilpotent Lie algebra if and only if $N$ is a nilpotent Lie group.
\end{lem}

\begin{prop}[Proposition $13$, Section $4$, Chapter $1$ in  \cite{bourbakiLie1-3}]\label{qr: prelims: translation: prop: n is maximal nilpotent ideal in parabolic}
In the notation of Proposition~\ref{qr: prop: Eberline Generalized Iwasawa},  $\mf{n}_\xi=\mathrm{Lie}(N_\xi)$ is a maximal nilpotent ideal in $\mf{g}_\xi=\mathrm{Lie}(G_\xi)$.
\end{prop}

\begin{rmk}
\begin{enumerate}
    \item The presentation of $\mf{n}_\xi$ in \cite{bourbakiLie7-9} is given by means of the root space decomposition of $\mathrm{Stab}_G(\xi)$, that appears in Proposition $2.17.13$ in \cite{eberlein1996geometry}. 
    
    \item There are two main objects in the literature that are referred to as the \emph{nilpotent radical} or the \emph{nilradical} of a Lie algebra. These are: (a) the maximal nilpotent ideal of the Lie algebra, and (b) the intersection of the kernels of all irreducible finite-dimensional representations. Proposition $13$ in Section $4$ of Chapter $9$ in \cite{bourbakiLie1-3} shows that in the case of Lie algebras of parabolic Lie groups, these notions coincide.
\end{enumerate}
\end{rmk}

\begin{cor}\label{qr: prelims: translation: cor: horospherical is maximal nilpotent in stabilizer of horosphere}
$N_\xi$ is a maximal connected nilpotent normal Lie subgroup of the identity connected component  $\mathrm{Stab}_G(\h)^\circ$.
\end{cor}

\begin{proof}
Lemma~\ref{qr: prelims: translation: lem: N connected nilpotent if an only if lie algebra is nilpotent} implies $N_\xi$ is nilpotent. Since  $\mathrm{Stab}_G{\h}\vartriangleleft \mathrm{Stab}_G(\xi)$, every normal subgroup of $\mathrm{Stab}_G(\h)$ containing $N_\xi$ is in fact a normal subgroup of $\mathrm{Stab}_G(\xi)$, still containing $N_\xi$. It remains to prove maximality of $N_\xi$ among all connected nilpotent normal Lie subgroups of $\mathrm{Stab}_G(\xi)$. Any such subgroup $N'\vartriangleleft \mathrm{Stab}_G(\xi)$ gives rise to an ideal $\mf{n}'$ of $\mf{g}_\xi=\mathrm{Lie}\big(\mathrm{Stab}_G(\xi)\big)$, and by Lemma~\ref{qr: prelims: translation: lem: N connected nilpotent if an only if lie algebra is nilpotent} it is a nilpotent ideal. Therefore by Proposition~\ref{qr: prelims: translation: prop: n is maximal nilpotent ideal in parabolic} it is contained in $\mf{n}_\xi=\mathrm{Lie}(N_\xi)$, implying that $N'\leq N_\xi$. 
 
\end{proof}

\begin{cor}\label{qr: prelims: translation: cor: Mostow hypotheses met}
A lattice in $\mathrm{Stab}_G(\h)$ intersects the horospherical subgroup $N_\xi$ in a lattice.
\end{cor}

\begin{proof}
Corollaries~\ref{qr: prelims: translation: cor: stab horosphere has no compact factors} and \ref{qr: prelims: translation: cor: horospherical is maximal nilpotent in stabilizer of horosphere} imply that the pair $N_\xi\vartriangleleft \mathrm{Stab}_G(\h)$ satisfy the hypotheses of Mostow's Lemma~\ref{qr: prelims: translation: lem: Mostow nilpotent radical intersection}. 
\end{proof}

\subsubsection{Indecomposable Horospherical Lattices}\label{sec: indecomposible horospherical lattice}
It is shown in \cite{BenoistMiquelArithmeticity} that if a horospherical lattice is contained in a Zariski-dense discrete subgroup, then the indecomposability condition is equivalent to the irreducibility of the ambient group. The latter is imposed on $\La$ as a hypothesis in Theorem~\ref{qr: thm: main for qr}. The precise definitions and statements are as follows.

\begin{defn}[Definition $2.14$ in \cite{BenoistMiquelArithmeticity}]\label{qr: defn: indecompsable and irreducibale lattice}

For a semisimple real algebraic Lie group $G$ and $U$ a horospherical subgroup of $G$, let $\Delta_U$ be a lattice in $U$. 
\begin{enumerate} 
    \item $\Delta_U$ is \emph{irreducible} if for any proper normal subgroup $N$ of $G^\circ$, one has $\Delta_U\cap N=\{e\}$.
    
    \item $\Delta_U$ is \emph{indecomposable} if one cannot write $G^\circ$ as a product $G^\circ=N'N''$ of two proper normal subgroups $N',N''\vartriangleleft G$ with finite intersection such that the group
    
    $$\Delta_U':=(\Delta_U\cap N')(\Delta_U\cap N'')$$
    has finite index in $\Delta_U$. 
\end{enumerate}
\end{defn}

\begin{defn}[Section $2.4.1$ in \cite{BenoistMiquelArithmeticity}]
Let $G$ be a semisimple real algebraic Lie group. A discrete subgroup $\La\leq G$ is said to be \emph{irreducible} if, for all proper normal subgroups $N\vartriangleleft G$, the intersection $\La\cap N$ is finite.
\end{defn}

\begin{lem}[Lemma $4.3$ in \cite{BenoistMiquelArithmeticity}]\label{qr: prelims: translation: lem:  indecomposability irredicubility criterion}
Let $G$ be a semisimple real algebraic Lie group, $U\subset G$ a non-trivial horospherical subgroup, and $\Delta_U\leq U$ a lattice of $U$ which is contained in a discrete Zariski-dense subgroup $\Delta$ of $G$. Then the following are equivalent: 
\begin{enumerate}
    \item $\Delta$ is irreducible.
    \item $\Delta_U$ is irreducible.
    \item $\Delta_U$ is indecomposable.
\end{enumerate}
\end{lem}

\subsubsection{Zariski Density}\label{sec: Zarsiky density subsection}
The last requirement is for $\La$ to be Zariski-dense. I use a geometric criterion which is well known to experts.

\begin{lem}[Proposition $2$ in \cite{KimMarkedLengthRigidity}]\label{qr: lem: geometric criterion for zariski dense subgroups}
Let $X$ be a symmetric space of noncompact type, $G=\mathrm{Isom}(X)^\circ$. A subgroup $\Delta\leq G$ is Zariski-dense if and only if:
\begin{enumerate}
    \item The group $\Delta$ does not
globally fix a point in $X(\infty)$, i.e.\ $\Delta\not\leq \mathrm{Stab}_G(\zeta)$ for any $\zeta\in X(\infty)$.
    \item The identity component of the Zariski closure of $\Delta$ does not leave invariant any proper totally geodesic submanifold in $X$. 
\end{enumerate} 
\end{lem}

I thank Elyasheev Leibtag for explaining some of the algebraic groups theory needed for this section. 

\begin{lem}\label{qr: alg: lem: identity comp of Zariski closure is finite index}

Let $\Delta\leq G$ be a discrete subgroup, and let $H\leq G$ be the Zariski closure of $\Delta$. Then $\Delta\cap H^\circ$ is of finite index in $\Delta$. 
\end{lem}

\begin{proof}
The subgroup $H^\circ$ is normal and of finite index in $H$.
\end{proof}

The following fact is probably known to experts. It appears in a recent work by Bader and Leibtag~\cite{Sheevy}.

\begin{lem}[Lemma $3.9$ in \cite{Sheevy}]\label{SHEEVY}
Let $k$ be a field, $\mathbf{G}$ a connected $k$ algebraic group, $P\leq G=\mathbf{G}(\mathbb{R})$ a parabolic subgroup. Then the centre of $G$ contains the centre of $P$.
\end{lem}

The following theorem of Dani generalizes the Borel Density Theorem.

\begin{thm}[See \cite{dani1982ergodic}]\label{qr: translation: thm: Dani Borel density for real solvable groups}
Let $\mathbf{S}$ be a real solvable algebraic group. If $S=\mathbf{S}(\mathbb{R})$ is $\mathbb{R}$-split, then every lattice $\G_S\leq S$ is Zariski-dense.
\end{thm}

\begin{rmk}\label{qr: translation: rmk: unipotent group splits Borel}
It is a fact (see Theorem $15.4$ and Section $18$ in \cite{BorelBookLinearAlgGroups}) that:
\begin{enumerate}
    \item Every unipotent group over $\mathbb{R}$ is $\mathbb{R}$-split. 
    \item For a field $k$ of characteristic $0$, a solvable linear algebraic $k$-group is $k$-split if and only if its maximal torus is $k$-split.
\end{enumerate} 
\end{rmk}

\begin{cor}\label{qr: cor: Lambda is Zariski dense}
Assume that: 
\begin{enumerate}
    \item $\big(\La\cap \mathrm{Stab}_G(\h) \big)\cdot x_0$ is a cocompact metric lattice in a horosphere $\h\subset X$ bounding a $\La$-free horoball. 
   
    \item Every $\G$-conical limit point is a $\La$-conical limit point. 
\end{enumerate}

Then $\La$ is Zariski-dense. 
\end{cor}

\begin{proof}
By Lemma~\ref{qr: lem: geometric criterion for zariski dense subgroups}, it is enough to show $\La$ does not stabilize any proper submanifold or fixes a point in $X(\infty)$. Denote the base point of $\h$ by $\xi$.  Recall that every element $g\in \operatorname{Stab}_G(\xi)$ acts by translation on the filtration $\{\h_t\}_{t\in \mathbb{R}}$ by horospheres based at $\xi$. Since $\xi$ admits a $\Lambda$-free horoball, the translation of every $\la\in\La\cap \operatorname{Stab}_G(\xi)$ is necessarily $0$. Since the $\La$-conical limit points are dense in the cone topology, Lemma~\ref{qr: prelims: sym spc: lem: Hattori penetration} and Lemma~\ref{qr: prelims: sym spc: prop: semicontinuity of angular metric} imply that there are $\La$-orbit points outside any bounded neighbourhood of $\h$, hence $\La$ cannot fix $\xi$.

Assume otherwise that $\La\leq \operatorname{Stab}_G(\zeta)$ for some other $\zeta\in X(\infty)$. Recall that $\La$ intersects the horospherical subgroup $N_\xi\leq \operatorname{Stab}_G(\xi)$ in a lattice (Corollary~\ref{qr: prelims: translation: cor: Mostow hypotheses met}), and that such a lattice is Zariski-dense in $N_\xi$ (Theorem~\ref{qr: translation: thm: Dani Borel density for real solvable groups}). Since a parabolic group is Zariski closed, if $\La\leq \operatorname{Stab}_G(\zeta)$ then $N_\xi\leq \operatorname{Stab}_G(\zeta)$. The latter implies $\operatorname{Stab}_G(\xi)\leq \operatorname{Stab}_G(\zeta)$ (see e.g. Proposition 7.78 in~\cite{knappliegroups}), which in turn implies that $\zeta$ lies in a face of the Weyl face determined by $\xi$ (Chapter 3 in~\cite{eberlein1996geometry}).

On the other hand since $\La\not\subset\operatorname{Stab}_G(\xi)$, there must be $\la\in \La\cap \big(\operatorname{Stab}_G(\zeta)\setminus \operatorname{Stab}_G(\xi)\big)$. Since $\la \xi$ admits a $\La$-free horoball, $\la\xi\in \wqg$ and therefore $d_T(\xi,\la \xi)=\pi$ and $d_T(\xi,\zeta)=\frac{\pi}{2}$. This is a contradiction, since points at Tits distance $\frac{\pi}{2}$ cannot determine Weyl faces that contain one another (see Proposition~3.6.26 and Proposition~3.6.28 in~\cite{eberlein1996geometry}).

Assume that $H:=\big(\overline{\La}^Z\big)^\circ$, the identity connected component of the Zariski closure of $\La$, stabilizes a totally geodesic submanifold $Y\subset X$. By Lemma~\ref{qr: alg: lem: identity comp of Zariski closure is finite index}, $\La_0:=\La\cap H$ is of finite index in $\La$, therefore $\La_0\cap \mathrm{Stab}_G(\h)$ is also a cocompact lattice in $\mathrm{Stab}_G(\h)$. The fact that $\big(\La_0\cap \mathrm{Stab}_G(\h)\big)\cdot x_0$ is a cocompact metric lattice in $\h$ readily implies that $\big(\La_0\cap \mathrm{Stab}_G(\h)\big)\cdot y$ is a cocompact metric lattice in $\h_y=\h(y,\xi)$. This goes to show that there is no loss of generality in assuming $x_0\in \h\cap Y$. It follows that $\La_0\cap \mathrm{Stab}_G(\h)\cdot x_0\subset Y\cap\h$, and therefore $\h\subset \mn_D(Y)$ for some $D>0$. A horosphere is a codimension $1$ submanifold, implying that $Y$ is either all of $X$ or a horosphere parallel to $\h$. The latter cannot occur as $\La\not\subset \operatorname{Stab}_G(\xi)$, hence $H$ does not stabilize any totally geodesic proper submanifold, and $\La$ is Zariski-dense. 
\end{proof}

\subsection{Proof of Theorem~\ref{qr: thm: main for qr}} \label{sec: qr: conclusion}

\begin{proof}[Proof of Theorem~\ref{qr: thm: main for qr}]
If $\{d_\g\}_{\g\in\G}$ is bounded, then $\La$ is a lattice by Corollary~\ref{qr: cor: bounded case real rank 1 one sided} or Theorem~\ref{qr: thm: eskin bounded case}, depending on the $\mathbb{R}$-rank of $G$.

If $\{d_\g\}_{\g\in\G}$ is unbounded, then the results of Section~\ref{sec: qr: cocompact horosphere} hold, and in particular every $\G$-conical limit point is $\La$-conical (Corollary~\ref{qr: cor: every Gamma conical is Lambda conical}). One has $\mathcal{L}(\La)=\mathcal{L}(\G)=X(\infty)$ (Lemma~\ref{qr: lem: limit set of Lambda equals limit set of Gamma}), so if $G$ is of $\rr$ then $\La$ is a lattice by Corollary~\ref{qr: cor: conical limit set criterion}.

The remaining case is when $G$ is of high rank. There is a horosphere $\mathcal{H}$ bounding a $\Lambda$-free horoball for which the $\La\cap \operatorname{Stab}_G(\mathcal{H})$-orbit intersects $\mathcal{H}$ in a cocompact metric lattice (Proposition~\ref{qr: prop: Lambda cocompact horosphreres}). Together with $\mathcal{L}_{\mathrm{con}}(\G)\subset \mathcal{L}_{\mathrm{con}}(\La)$, these are the only  assumptions needed for the results of Section~\ref{sec: qr: geometry to algebra} to hold, and I may conclude $\La$ is a discrete, Zariski-dense subgroup that contains an indecomposible horospherical lattice (Corollary~\ref{qr: cor: Lambda is Zariski dense}, Corollary~\ref{qr: prelims: translation: cor: Mostow hypotheses met} and Lemma~\ref{qr: prelims: translation: lem:  indecomposability irredicubility criterion}). Theorem~\ref{qr: thm: Benoist-Miquel arithmeticity} then renders $\La$ a lattice. 
\end{proof}

\begin{rmk}
The sublinear nature of the hypothesis in Theorem~\ref{thm: main Intro} induces coarse metric constraints. A horospherical lattice on the other hand is a very precise object. The proof above produces a whole lattice of unipotent elements in $\La$ (this is Corollary~\ref{qr: prelims: translation: cor: Mostow hypotheses met}); it is also the only proof that I know which produces even a single unipotent element in $\La$. 

\end{rmk}

\subsection{Proof of Theorem~\ref{thm: commensurability and uniform statement intro}}\label{sec: commensurability theorem proof}
\begin{proof}[Proof of Theorem~\ref{thm: commensurability and uniform statement intro}]
The hypothesis $\G\not\subset\mn_D(\La)$ for any $D>0$ translates into $\{d_\gamma\}_{\gamma\in \Gamma}$ being unbounded. By Proposition~\ref{qr: prop: Lambda cocompact horosphreres} there is a horosphere $\mathcal{H}$ such that $\Lambda\cap \operatorname{Stab}_G(\mathcal{H})$ acts cocompactly on $\mathcal{H}$. A lattice with this property is of \qr by Theorem~\ref{qr: prelims: sym spc: thm: Leuzinger compact core}, proving the first item. Assume otherwise that $\Gamma\subset \mathcal{N}_D(\Lambda)$ and moreover $\Gamma$ sublinearly covers $\Lambda$. By Proposition~\ref{qr: prop: bounded and sublinear implies bounded}, $\G$ and $\La$ lie at bounded Hausdorff distance. Commensurability follows by the results of Eskin and Schwartz (Theorems~\ref{qr: thm: eskin bounded case} and~\ref{qr: thm: Schwartz: finite distance imples lattice}).
\end{proof}

\bibliographystyle{plain}
\bibliography{bibliography.bib}

\end{document}